\documentclass[a4paper]{amsart}
\usepackage{graphicx}
\usepackage{amssymb}
\usepackage[leqno]{amsmath}
\usepackage{amsfonts}
\usepackage{amsopn}
\usepackage{amstext}
\usepackage{amsthm}
\usepackage{enumerate}
\usepackage{verbatim}
\usepackage[colorlinks,backref]{hyperref}
\usepackage{mathrsfs}
\usepackage{euscript}
\usepackage{todonotes}
\usepackage[normalem]{ulem}

\makeatletter\@namedef{subjclassname@2020}{%
 \textup{2020} Mathematics Subject Classification}
\makeatother

\newtheorem{theorem}{Theorem}[section]
\newtheorem{corollary}[theorem]{Corollary}
\newtheorem{proposition}[theorem]{Proposition}
\newtheorem{lemma}[theorem]{Lemma}
\theoremstyle{definition}
\newtheorem{definition}[theorem]{Definition}
\newtheorem{remark}[theorem]{Remark}
\newtheorem{example}[theorem]{Example}
\newtheorem{problem}[theorem]{Problem}

\newtheorem{thmx}{Theorem}

\newenvironment{cor}[1]{\begin{corollary}\label{cor:#1}}{\end{corollary}}
\newenvironment{dfn}[1]{\begin{definition}\label{def:#1}}{\end{definition}}
\newenvironment{exm}[1]{\begin{example}\label{exm:#1}}{\end{example}}
\newenvironment{lem}[1]{\begin{lemma}\label{lem:#1}}{\end{lemma}}
\newenvironment{pro}[1]{\begin{proposition}\label{pro:#1}}{\end{proposition}}
\newenvironment{rem}[1]{\begin{remark}\label{rem:#1}}{\end{remark}}
\newenvironment{thm}[1]{\begin{theorem}\label{thm:#1}}{\end{theorem}}
\newenvironment{prb}[1]{\begin{problem}\label{prb:#1}}{\end{problem}}

\newcommand{\COR}[1]{Corollary~\textup{\ref{cor:#1}}}
\newcommand{\DEF}[1]{Definition~\textup{\ref{def:#1}}}
\newcommand{\EXM}[1]{Example~\textup{\ref{exm:#1}}}
\newcommand{\LEM}[1]{Lemma~\textup{\ref{lem:#1}}}

\newcommand{\PRO}[1]{Proposition~\textup{\ref{pro:#1}}}
\newcommand{\REM}[1]{Remark~\textup{\ref{rem:#1}}}
\newcommand{\THM}[1]{Theorem~\textup{\ref{thm:#1}}}

\newcommand{\aaA}{\mathscr{A}}
\newcommand{\bbB}{\mathscr{B}}
\newcommand{\ddD}{\mathscr{D}}
\newcommand{\eeE}{\mathscr{E}}
\newcommand{\ggG}{\mathscr{G}}
\newcommand{\KKK}{\mathbb{K}}
\newcommand{\NNN}{\mathbb{N}}
\newcommand{\QQQ}{\mathbb{Q}}
\newcommand{\RRR}{\mathbb{R}}
\newcommand{\ZZZ}{\mathbb{Z}}
\newcommand{\Pp}{\mathfrak{P}}
\newcommand{\qQ}{\mathfrak{q}}
\newcommand{\zZ}{\mathfrak{z}}

\newcommand{\dd}{\colon}
\newcommand{\OPN}[1]{\operatorname{#1}}
\newcommand{\UP}[1]{\textmd{\textup{#1}}}
\newcommand{\varempty}{\varnothing}
\newcommand{\df}{\stackrel{\textup{def}}{=}}
\newcommand{\TFCAE}{The following conditions are equivalent:}
\newcommand{\tfcae}{the following conditions are equivalent:}
\newcommand{\Cantor}{Cantor space}
\newcommand{\epsi}{\varepsilon}
\newcommand{\compatible}[1]{\(#1\)-compati\-ble}
\newcommand{\morphism}{morphism}
\newcommand{\Can}{2^\NNN}

\newcommand{\zero}{\pmb{0}}
\newcommand{\unit}{\pmb{1}}
\newcommand{\BooV}[2]{#1 \vee #2}
\newcommand{\BooW}[2]{#1 \wedge #2}
\newcommand{\BooD}[2]{#1 \setminus #2}
\newcommand{\homspec}[1]{hom#1spec}

\begin{document}

\title[Automorphism groups of measures, II]{Automorphism groups of measures\\%
 on the Cantor space.\\Part II: Abstract homogeneous measures}
\author[P.~Niemiec]{Piotr Niemiec}
\address{Wydzia\l{} Matematyki i~Informatyki\\%
 Uniwersytet Jagiello\'{n}ski\\ul.\ \L{}o\-ja\-sie\-wi\-cza 6\\%
 30-348 Krak\'{o}w\\Poland}
\email{piotr.niemiec@uj.edu.pl}
\thanks{Research partially supported by the National Science Centre, Poland
 under the Weave-UNISONO call in the Weave programme [grant no
 2021/03/Y/ST1/00072].}
\begin{abstract}
The main aim of the paper is to introduce a new class of (semi\-group-valued)
measures that are ultrahomogeneous on the Boolean algebra of all clopen subsets
of the \Cantor\ and to study their automorphism groups. A characterisation, in
terms of the so-called trinary spectrum of a measure, of ultrahomogeneous
measures such that the action of their automorphism groups is (topologically)
transitive or minimal is given. Also sufficient and necessary conditions for
the existence of a dense (or co-meager) conjugacy class in these groups are
offered. In particular, it is shown that there are uncountably many full
non-atomic probability Borel measures \(\mu\) on the \Cantor\ such that \(\mu\)
and all its restrictions to arbitrary non-empty clopen sets have all
the following properties: this measure is ultrahomogeneous and not good,
the action of its automorphism group \(G\) is minimal (on a respective clopen
set), and \(G\) has a dense conjugacy class. It is also shown that any minimal
homeomorphism \(h\) on the \Cantor\ induces a homogeneous \(h\)-invariant
probability measure that is universal among all \(h\)-invariant probability
measures (which means that it `generates' all such measures) and this property
determines this measure (in a certain sense) up to a \(\QQQ\)-linear
isomorphism.
\end{abstract}
\maketitle

\tableofcontents

\section{Introduction}

Automorphism groups of good measures are intensively studied for the last two
decades. According to the results of Glasner and Weiss \cite{GW95} and of Akin
\cite{Akin05}, good measures can nicely be characterised as precisely those
(full probability Borel) measures (on the \Cantor) that are preserved by some
homeomorphisms that are both uniquely ergodic and minimal. These measures
provide interesting examples of transformation groups whose action is minimal
and which, in certain cases, have Rokhlin properties (consult, e.g.,
\cite{Akin05}, \cite{IM16}, \cite{DKMN1}). Currently it is not clear whether
such exceptional properties are specific precisely to good measures (among
probability Borel ones on the \Cantor) or if there are other measures that share
analogous properties. One of the main goals of this paper is to introduce a new
class of measures that we call supersolid (these are precisely those measures
that are ultrahomogeneous and have the property that the action of their
automorphism groups is minimal on the \Cantor). Among other things, we prove
that there are uncountably many such measures each of which is full, non-atomic
and not good, and, in addition, whose automorphism groups have the Rokhlin
property. However, to construct such probability measures we involve more
abstract set functions that we call \emph{group-valued measures}. As our ideas
work perfectly for semigroup-valued set functions as well, we discuss them in
this more general context. Our basic concept is motivated by a well-known
property of non-negative (finitely) additive set functions defined on
the algebra \(\bbB\) of all clopen subsets of the \Cantor\ that says that each
such a function extends uniquely to a (finite non-negative) Borel measure (see,
e.g., \cite{BieKuWa19}). Having this in mind, instead of dealing with the full
\(\sigma\)-algebra of all Borel sets, we restrict our considerations to
the aforementioned algebra \(\bbB\). A similar approach (but concentrated only
on non-negative real-valued measures) is presented in \cite{BieKuWa19}. However,
our investigations are independent of that work. Our most interesting examples
of (probability Borel) measures are obtained from measures taking values in
finite-dimensional vector spaces. Our main tools are the concept of
the so-called \emph{trinary spectrum} that is a full invariant of a homogeneous
measure, and the products of such sets. These trinary spectra (of homogeneous
measures) play the role of clopen values sets of good measures---they uniquely
determine the measure under considerations up to homeomorphism of the \Cantor.
In our investigations we are mostly interested in properties that are
\emph{hereditary} for a given homogenous measure \(\mu\) in a sense that not
only \(\mu\) shares this property but also all its restrictions to arbitrary
non-empty clopen sets (as measures on these clopen sets) share it.\par
Probably the most important result of the paper reads as follows.
\begin{thmx}[cf. Theorems~\ref{thm:PN-homeo-mini} and \ref{thm:PN-univ-exist}]
Let \(h\) be a minimal homeomorphism on the \Cantor.
\begin{enumerate}[\upshape(A)]
\item There exists a homogeneous probability Borel measure \(m\) that is
 \(h\)-invariant and has the following universal property: For any
 \(h\)-invariant probability measure \(\mu\) there exists a unique
 \(\QQQ\)-linear operator \(f\) from the \(\QQQ\)-linear span \(W(m)\) of \(V(m)
 = m(\bbB)\) into \(\RRR\) such that \(\mu(U) = (f \circ m)(U)\) for each clopen
 set \(U\). (In other words, there is a natural one-to-one correspondence
 between \(h\)-invariant probabilty measures \(\mu\) and \(\QQQ\)-linear
 operators \(f\) from \(W(m)\) into \(\RRR\) such that \(f(1) = 1\) and \(f(v)
 \geq 0\) for all \(v \in V(m)\).)
\item The above measure \(m\) is unique in the following sense: if \(m'\) has
 all the properties specified in \UP{(A)}, then there exists a unique
 \(\QQQ\)-linear isomorphism \(g\) from \(W(m)\) onto \(W(m')\) such that
 \(m'(U) = (g \circ m)(U)\) for any clopen set \(U\).
\item The above measure \(m\) is good if and only if \(h\) is uniquely ergodic.
\item The automorphism group of \(m\) coincides with the uniform closure of
 the full group of \(h\).
\end{enumerate}
\end{thmx}
Other main results of the paper (that are valid for measures taking values in
arbitrary Abelian semigroups with neutral elements) are:
\begin{itemize}
\item \THM{PN-4EP} which says that a measure is homogeneous iff it is
 ultrahomogeneous (both these properties are related to the property of
 extending partial isomorphisms preserving a given measure from a finite Boolean
 subalgebra of \(\bbB\) to automorphisms of \(\bbB\) that preserve the measure
 as well). In particular, homogeneity (for measures) is hereditary.
\item \THM{PN-ext-4EP} that gives an equivalent condition for two homogeneous
 measures to be conjugate by a Boolean algebra automorphism in terms of their
 trinary spectra (briefly defined in \eqref{eqn:PN-spectrum}).
\item \THM{PN-ext-spec} which gives a full characterisation of abstract sets
 (called \emph{\homspec{3}s}) that are trinary spectra of some homogeneous
 measures (it involves axioms (Sp0)--(Sp4b) listed in \PRO{PN-any-spectrum}).
\item \THM{PN-generic} on genericity of the conjugacy class of a homogeneous
 measure with a prescribed trinary spectrum among all measures that are
 compatible with this set.
\item \THM{PN-solid} and \THM{PN-supersolid} which introduce, respectively,
 \emph{solid} and \emph{supersolid} \homspec{3}s that are precisely the trinary
 spectra of homogeneous measures such that the action of their automorphism
 groups is (topologically) transitive or, respectively, minimal on the \Cantor.
 (If a probability Borel measure on the \Cantor\ is supersolid, then it is
 automatically full and non-atomic---see \PRO{PN-super-full-nonatom} below.)
 Characterisations given in these theorems enables us to produce new solid (or
 supersolid) \homspec{3}s by forming Cartesian products of such sets (as
 described in Section~\ref{sec:products} below).
\item \THM{PN-zero-super} that fully classifies all homogeneous
 (semigroup-valued) measures \(\mu\) such that the action of the automorphism
 group of \(\mu\) is minimal on \(\Can\) and \(\mu(A)\) is zero for some proper
 (non-empty) clopen set \(A\).
\item \THM{PN-univ-exist} on the existence of the so-called universal
 \(H\)-invariant probability measures for any set \(H\) of homeomorphisms of
 the \Cantor.
\item \COR{PN-super-nogood-dense} on the existence of probability (or signed)
 Borel measures that are supersolid and non-good, and  for which the property
 that their automorphism groups have the Rokhlin property is hereditary.
\item \THM{PN-conj-class} that establishes an equivalence between two
 \underline{hereditary} properties: of existence of dense vs. dense
 \(\ggG_{\delta}\) conjugacy classes in the automorphism groups of homogeneous
 measures.
\item \THM{PN-strong-good-like} that generalises Akin's \cite{Akin05} result on
 the existence of dense \(\ggG_{\delta}\) conjugacy classes (for certain good
 measures) to the context of the so-called good-like group-valued measures.
\end{itemize}
We also study two classes of measures that serve as a natural counterpart of
good measures (among probability ones) in the realm of group-valued measures; we
call them \(G\)-filling (where \(G\) is a countable group) and good-like (with
values in a linearly ordered countable group). All measures from the former
class are supersolid and they depend on a unique parameter running over all
elements of \(G\). Denoting by \(H\) the automorphism group of such a measure,
we give an equivalent condition for \(H\) to have the Rokhlin property when \(G
= \ZZZ^d\) (see \THM{PN-Z-dense} below) or when \(G\) is a vector space over
\(\QQQ\) or over the field \(\ZZZ/p\ZZZ\) where \(p\) is a prime (see
Theorems~\ref{thm:PN-0-filling} and \ref{thm:PN-Q-vector} below). The results
about \(G\)-filling measures resemble classical theorems on good measures
obtained in \cite{Akin05}, \cite{IM16} and by others.\par
The paper is organized as follows. In the next section we introduce homogenepus
semigroup-valued measures as well as their trinary spectra and abstract
\homspec{3}s, and present fundamental results on these notions. For example, we
show there that each homogeneous measure is ultrahomogeneous and that both these
properties are hereditary; we prove that two homogeneous measures are conjugate
(by a Boolean algebra automorphism) iff their trinary spectra coincide; and that
the class of all \homspec{3}s coincides with the class of all trinary spectra of
homogeneous measures---so, there is a one-to-one correspondence between
\homspec{3}s and such measures (considered up to conjugacy), and instead of
constructing homogeneous measures it is sufficient to build \homspec{3}s. We
also give an equivalent condition for an abstract measure defined on a finite
subalgebra of \(\bbB\) to be extendable to a homogeneous one (on \(\bbB\)) with
a prescribed trinary spectrum. Next we show that being a homogenous measure with
a prescribed trinary spectrum is a generic property in a certain Polish space of
semigroup-valued measures (see \THM{PN-generic} below). The section is concluded
with an example of a \emph{rigid} homogeneous measure (here rigidity means that
the automorphism group is trivial). This example serves as a motivation to study
`better' measures than only homogeneous, which we do in the third section,
devoted to solid and supersolid measures as well as such \homspec{3}s. These
properties are related to topological transitivity and minimality of the action
of the automorphism group of a homogeneous measure. In the fourth section we
present an application of our theory by showing that new solid or supersolid
measures can be produced by forming products of trinary spectra of already
existing such measures. We also introduce there the notion of \emph{independent}
measures (in the sense of homogeneity) and by means of this term we give
examples of probability Borel measures on the \Cantor\ that are supersolid but
not good. In Section~5 we discuss semigroup-valued homogeneous measures that
attain zero value on a proper clopen set. It turns out that each such a measure
is automatically solid, and clopen sets on which this measure is supersolid are
`densely' placed in \(\Can\). Finally, we give a full
classification/characterisation of all such measures that are supersolid.
Section~6 discusses the case of group-valued homogeneous measures. We introduce
\emph{binary spectra} of such measures and define \emph{\homspec{2}s} that are
in a one-to-one correspondence with \homspec{3}s in groups. (Although
\homspec{2}s use two parameters instead of three, some of their axioms are more
complicated or less transparent than their counterparts for \homspec{3}s.) We
also introduce more handy structures that we call \emph{solid partially ordered
spectra}---they are in a one-to-one correspondence with solid (group-valued)
measures that do not attain value zero on proper clopen sets. With the aid of
this notion we show that the so-called good-like measures (introduced therein)
are all solid. In Section~7 we introduce (both probability and group-valued)
universal invariant measures that, in a sense, generate all other invariant
measures, and we show that quite often such measures are automatically
homogeneous. This is the case when one deals with \(h\)-invariant probability
measures where \(h\) is a minimal homeomorphism on the \Cantor\ (see Theorem~A
above). Also all \(G\)-filling as well as all good-like measures are universal.
In Section~8 we address the question of when the automorphism group of
a homogeneous measure has the Rokhlin property (that is, when this group admits
a dense conjugacy class). We adapt there classical concepts and generalise
\cite[Theorem~4.8]{DKMN1} to the context of abstract homogeneous measures (see
\THM{PN-dense-conj} below). We also prove that there are uncountably many
probability Borel measures on the \Cantor\ that are non-good and supersolid, and
for which the aforementioned property concerning dense conjugacy classes is
hereditary. Finally, we study this issue for \(G\)-filling measures (their
definition is most similar to the one of good measures, but does not involve any
linear order). We fully answer the question about the existence of a dense
conjugacy class for countable vector spaces over \(\QQQ\) and for the groups
\(\ZZZ^d\) where \(d\) is finite. We also show that---for an arbitrary
group-valued homogeneous measure---if the property of having dense conjugacy
class is hereditary, then the underlying group (generated by the values of
the measure) is divisible (see \PRO{PN-hered-1/N} below). Finally, Section~8 is
devoted to the study of those (homogeneous) measures whose automorphism groups
have dense \(\ggG_{\delta}\) conjugacy class. We develop there main ideas from
\cite{DKMN1} to show that in certain cases this property is equivalent to
the Rokhlin property. We conclude the paper with a positive result in this
direction for good-like measures, and with two open problems.

\subsection*{Notation and terminology}
In this paper \(\NNN = \{0,1,2,\ldots\}\). We use \(\Can\) to denote
the \Cantor. A subset of \(\Can\) is \emph{clopen} if it is both open and
closed. \((\bbB,\zero,\unit,\vee,\wedge,\setminus)\) is reserved to denote
a countable atomless Boolean algebra (symbol ``\(\setminus\)'' stands for
the complement operator in \(\bbB\)). So, \(\bbB\) is a Boolean algebra model of
the \Cantor. For simplicity, we will write \(\BooD{a}{b}\) (where \(a, b \in
\bbB\)) instead of \(\BooW{a}{(\setminus b)}\). In particular, \(\setminus b =
\BooD{\unit}{b}\). We call two elements \(a\) and \(b\) of \(\bbB\)
\emph{disjoint} if \(\BooW{a}{b} = \zero\). Additionally, we call the elements
of \(\bbB \setminus \{\zero\}\) \emph{non-zero}; \(a \in \bbB\) is called
\emph{proper} if it is non-zero and differs from \(\unit\). Finally, by
a \emph{partition} (of \(\bbB\)) we mean any finite non-empty set \(\{a_1,
\ldots,a_N\} \subset \bbB\) that consists of pairwise disjoint non-zero elements
such that \(\BooV{\BooV{a_1}{\ldots}}{a_N} = \unit\).\par
All semigroups (as well as groups) considered in this section are Abelian and
have neutral elements. We use the additive notation for them. The neutral
element of a semigroup \(S\) is denoted by \(0_S\). All homomorphism between
semigroups preserve neutral elements.\par
A probability Borel measure on \(\Can\) is said to be \emph{full} if it does not
vanish on any non-empty open set (or, equivalently, if it does not vanish on any
non-empty clopen set). A probability Borel measure \(\mu\) on \(\Can\) is
\emph{good} if it is full and for any two clopen subsets \(U\) and \(V\) of
\(\Can\) such that \(\mu(U) < \mu(V)\) there exists a clopen set \(W \subset
V\) for which \(\mu(W) = \mu(U)\).

\section{Homogeneous semigroup-valued measures}

When representing the \Cantor\ as the Boolean algebra (of all its clopen
subsets), ultrahomogeneous measures (e.g., those resembling good ones) can be
defined in a more abstract and general context, as specified in
\DEF{PN-homo-meas} below.

\begin{dfn}{PN-S-measure}
Let \(S\) be a semigroup. A (finitely additive) \emph{\(S\)-valued measure} on
\(\bbB\) is any function \(\mu\dd \bbB \to S\) such that \(\mu(\zero) = 0_S\)
and \(\mu(\BooV{a}{b}) = \mu(a)+\mu(b)\) for all disjoint elements \(a, b \in
\bbB\). In a similar way one defines \(S\)-valued measures defined on Boolean
subalgebras of \(\bbB\).\par
If \(\bbB\) is identified with the Boolean algebra of all clopen subsets of
\(\Can\), then:
\begin{itemize}
\item for any semigroup \(S\), there is a natural one-to-one correspondence
 between \(S\)-valued measures on \(\bbB\) and homomorphisms from
 \(C(\Can,\NNN)\) to \(S\); and
\item similarly, for any group \(G\), there is a natural one-to-one
 correspondence between \(G\)-valued measures on \(\bbB\) and homomorphisms from
 \(C(\Can,\ZZZ)\) to \(G\)
\end{itemize}
(where \(C(\Can,\NNN)\) and \(C(\Can,\ZZZ)\) denote the semigroups equipped with
pointwise addition of all continuous functions from \(\Can\) to the respective
target spaces).
\end{dfn}

In this section we will deal with homogeneous semigroup-valued measures, which
we introduce in the following

\begin{dfn}{PN-homo-meas}
Let \(S\), \(\mu\) and \(\aaA\) be, respectively, a semigroup, an \(S\)-valued
measure on \(\bbB\) and a Boolean subalgebra of \(\bbB\). A \emph{partial
\(\mu\)-isomorphism} is a one-to-one Boolean algebra homomorphism \(h\dd \aaA
\to \bbB\) such that \(\mu(h(a)) = \mu(a)\) for all \(a \in \aaA\). If, in
addition to this, \(\aaA = \bbB\) and \(h(\bbB) = \bbB\), \(h\) is called
a \emph{\(\mu\)-automorphism}. We use \(\OPN{Aut}(\mu)\) to denote the group of
all \(\mu\)-automorphisms of \(\bbB\). As usual, \(\OPN{Aut}(\mu)\) becomes
a Polish group when equipped with the pointwise convergence topology induced by
the discrete topology on \(\bbB\).\par
An \(S\)-valued measure \(\mu\) on \(\bbB\) is said to be:
\begin{itemize}
\item \emph{homogeneous} if any partial \(\mu\)-isomorphism defined on
 an arbitrary 4-point Boolean subalgebra of \(\bbB\) extends to
 a \(\mu\)-automorphism; in other words, \(\mu\) is homogeneous iff for any two
 proper elements \(a\) and \(b\) in \(\bbB\) such that \(\mu(a) = \mu(b)\) and
 \(\mu(\BooD{\unit}{a}) = \mu(\BooD{\unit}{b})\) there exists \(\phi \in
 \OPN{Aut}(\mu)\) with \(\phi(a) = b\);
\item \emph{ultrahomogeneous} (or \emph{finitely homogeneous}) if any partial
 \(\mu\)-isomorphism defined on a finite Boolean subalgebra of \(\bbB\) extends
 to a \(\mu\)-automorphism of \(\bbB\).
\end{itemize}
\end{dfn}

One of our basic goals of this section is to show that all homogenenous
\(S\)-valued measures are automatically ultrahomogeneous. It is worth noting
here that all good measures are ultrahomogeneous (see, e.g., \cite{DKMN1}).\par
To formulate our first result on homogeneous semigroup-valued measures, we need
another notion, which turns out to be crucial in our considerations.

\begin{dfn}{PN-meas-4-point}
We say that an \(S\)-valued measure has \emph{4 elements property}, or
(\emph{4EP}) for short, if for any partition \(\{a,b,c\} \subset \bbB\) and each
\(d \in \bbB\) such that
\begin{equation}\label{eqn:PN-4EP}
\mu(a)+\mu(b) = \mu(d) \qquad \UP{and} \qquad \mu(c) = \mu(\BooD{\unit}{d})
\end{equation}
there exist two disjoint proper elements \(p\) and \(q\) in \(\bbB\) for which
\(\mu(p) = \mu(a)\), \(\mu(q) = \mu(b)\) and \(\BooV{p}{q} = d\).
\end{dfn}

The role of (4EP) is explained in \THM{PN-4EP} below. Before establishing this
result, we collect fundamental properties of \(S\)-valued measures with (4EP).
To formulate the first of them, we define for each \(S\)-valued measure
\(\mu\) on \(\bbB\) its \emph{trinary spectrum} as
\begin{equation}\label{eqn:PN-spectrum}
\Sigma(\mu) \df \{(\mu(a),\mu(b),\mu(c))\dd\ \{a,b,c\}\textup{ is a partition
of }\bbB\}.
\end{equation}
Utility of tbe above notion is shown by

\begin{thm}{PN-ext-4EP}
Let \(\mu\) and \(\lambda\) be two \(S\)-valued measures on \(\bbB\) that have
\UP{(4EP)} and let \(\phi\dd \aaA \to \bbB\) be a one-to-one Boolean algebra
homomorphism defined on a finite Boolean subalgebra \(\aaA\) of \(\bbB\) such
that \(\lambda(a) = \mu(\phi(a))\) for all \(a \in \aaA\). Then \(\phi\) extends
to a Boolean algebra automorphism \(\Phi\dd \bbB \to \bbB\) such that
\begin{equation}\label{eqn:PN-trans}
\lambda(b) = \mu(\Phi(b)) \qquad (b\in\bbB)
\end{equation}
if and only if \(\Sigma(\lambda) = \Sigma(\mu)\).\par
In particular, for two \(S\)-valued measures \(\mu\) and \(\lambda\) on \(\bbB\)
with \UP{(4EP)} there exists a Boolean algebra automorphism \(\Phi\) such that
\eqref{eqn:PN-trans} holds iff \(\Sigma(\lambda) = \Sigma(\mu)\).
\end{thm}
\begin{proof}
It is straightforward that if \(\Phi\) exists, then the trinary spectra of
\(\mu\) and \(\lambda\) coincide. To prove the converse, we assume \(\Sigma(\mu)
= \Sigma(\lambda)\). For simplicity, we set \(\Sigma = \Sigma(\mu)\) and
\(\Sigma' \df \{(a,b+c)\dd\ (a,b,c) \in \Sigma\}\). In a similar manner we
define \(\Sigma'(\mu)\) (which coincides with \(\Sigma'\)). Observe that
\(\Sigma'(\mu)\) consists precisely of all pairs of the form \((\mu(a),\mu(b))\)
where \(\{a,b\}\) is a partition of \(\bbB\).\par
By a standard back-and-forth argument, it suffices to show that for arbitrary
\(b \in \bbB \setminus \aaA\), \(\phi\) can be extended to a one-to-one Boolean
algebra homomorphism \(\psi\dd \ddD \to \bbB\) such that
\begin{equation}\label{eqn:PN-aux1}
\lambda(x) = \mu(\psi(x)) \qquad (x \in \ddD)
\end{equation}
where \(\ddD\) is the Boolean algebra generated by \(\aaA \cup \{b\}\). To this
end, we fix such \(b\) and take a (unique) partition \(\{a_1,\ldots,a_N\}
\subset \aaA\) of \(\bbB\) (where \(N > 0\)) such that for each non-zero element
\(c\) of \(\aaA\) there exists a unique finite non-empty set \(J \subset \{1,
\ldots,N\}\) with \(c = \bigvee_{j \in J} a_j\). (Note that \(a_j = \unit\) (for
some \(j\) or for all \(j\)) iff \(N = 1\).) Further, let \(I\) consist of all
indices \(j \in \{1,\ldots,N\}\) such that \(b_j \df \BooW{b}{a_j}\) is non-zero
and differs from \(a_j\). Then \(I \neq \varempty\) (because \(b \notin \aaA\))
and \(b' \df \bigvee_{j \in I} b_j\) satisfies \(\BooW{b'}{b} = b'\) and
\(\BooD{b}{b'} \in \aaA\). So, the Boolean algebra generated by \(\aaA \cup
\{b\}\) coincides with the one generated by \(\aaA \cup \{b'\}\), and therefore
we may and do assume that \(b' = b\). Now for any \(j \in I\):
\begin{itemize}
\item if \(N > 1\), then \((p,q,r) \df (\lambda(b_j),\lambda(\BooD{a_j}{b_j}),
 \lambda(\BooD{\unit}{a_j}))\) belongs to \(\Sigma = \Sigma(\mu)\) (resp.
 \((p,q) \df (\lambda(b_j),\lambda(\BooD{a_j}{b_j}))\) belongs to \(\Sigma' =
 \Sigma'(\mu)\) provided that \(N = 1\));
\item in particular, if \(N > 1\), then there exists a partition \(\{x,y,z\}\)
 of \(\bbB\) such that \(\mu(x) = p\), \(\mu(y) = q\) and \(\mu(z) = r\) (resp.
 there exists a partition \(\{x,y\}\) of \(\bbB\) such that \(\mu(x) = p\) and
 \(\mu(y) = q\), provided that \(N = 1\));
\item \(\mu(x)+\mu(y) = p+q = \lambda(a_j) = \mu(\phi(a_j))\) and, if \(N > 1\),
 \(\mu(z) = r = \lambda(\BooD{\unit}{a_j}) = \mu(\BooD{\unit}{\phi(a_j)})\);
\item consequently, it follows from (4EP) of \(\mu\) that there exists non-zero
 \(d_j \in \bbB\) such that \(\BooW{d_j}{\phi(a_j)} = d_j \neq \phi(a_j)\) and
 \(\mu(d_j) = p\) and \(\mu(\BooD{\phi(a_j)}{d_j}) = q\) (if \(N = 1\), it is
 sufficient to set \(d_j = x\));
\item then automatically \(d_j \notin \aaA\) (as \(\phi(a_1),\ldots,\phi(a_N)\)
 are all minimal non-zero elements of \(\phi(\aaA)\)).
\end{itemize}
Now it is enough to note that the elements \(a_j\) for \(j \notin J\) together
with \(b_j,\BooD{a_j}{b_j}\) where \(j \in J\), and---similarly---the elements
\(\phi(a_j)\) for \(j \notin J\) together with \(d_j,\BooD{\phi(a_j)}{d_j}\)
where \(j \in J\), form partitions of \(\bbB\) to convince oneself that
the assignments
\[\begin{cases}
a_j \mapsto \phi(a_j), & j \notin J\\
b_j \mapsto d_j, & j \in J\\
\BooD{a_j}{b_j} \mapsto \BooD{\phi(a_j)}{d_j}, & j \in J
\end{cases}\]
uniquely extend to a one-to-one Boolean algebra homomorphism \(\psi\dd \ddD \to
\bbB\) (where, as before, \(\ddD\) stands for the Boolean algebra generated by
\(\aaA \cup \{b\}\)). It follows from the construction that \(\psi\) extends
\(\phi\) and satisfies \eqref{eqn:PN-aux1} (we leave the details to
the reader).\par
Finally, if \(\Sigma(\mu) = \Sigma(\lambda)\) (and both \(\mu\) and \(\lambda\)
have (4EP)), then necessarily \(\mu(\unit) = \lambda(\unit)\) and it follows
from the first part of the result that the identity on \(\{\zero,\unit\}\) can
be extended to a Boolean automorphism satisfying \eqref{eqn:PN-trans}.
\end{proof}

As a corollary, we obtain

\begin{thm}{PN-4EP}
A homogenous \(S\)-valued measure has \UP{(4EP)}. Conversely, an \(S\)-valued
measure on \(\bbB\) with \UP{(4EP)} is ultrahomogeneous.\par
In particular, all homogeneous semigroup-valued measures on \(\bbB\) are
ultrahomogeneous.
\end{thm}
\begin{proof}
If \(\mu\) is homogeneous and \(\{a,b,c\}\) and \(d\) are, respectively,
a partition of \(\bbB\) and an element from \(\bbB\) such that
\eqref{eqn:PN-4EP} holds, then there exists \(\phi \in \OPN{Aut}(\mu)\) for
which \(\phi(\BooV{a}{b}) = d\) (by homogeneity). Consequently, \(p \df
\phi(a)\) and \(q \df \phi(b)\) are proper and disjoint and satisfy \(\mu(p) =
\mu(a)\), \(\mu(q) = \mu(b)\) and \(\BooV{p}{q} = d\), which means that \(\mu\)
has (4EP).\par
On the other hand, if \(\mu\) has (4EP) and \(\phi\dd \aaA \to \bbB\) is
a partial \(\mu\)-isomorphism where \(\aaA\) is a finite Boolean subalgebra of
\(\bbB\), then it follows from \THM{PN-ext-4EP} (applied to \(\lambda = \mu\))
that \(\phi\) extends to a member of \(\OPN{Aut}(\mu)\), which shows that
\(\mu\) is ultrahomogeneous.
\end{proof}

\THM{PN-ext-4EP} shows, among other things, that the trinary spectrum of
a homogeneous semi\-group-valued measure determines this measure up to
automorphism. So, it is important to study this characteristic, which we do in
the next two results.

\begin{pro}{PN-any-spectrum}
Let \(\mu\) be a homogeneous \(S\)-valued measure on \(\bbB\) and let
\(\Sigma\) stand for its trinary spectrum. Then:
\begin{enumerate}[\upshape(Sp1b)]
\item[\upshape(Sp0)] \(\Sigma \subset S \times S \times S\) is \UP(at most\UP)
 countable and non-empty;
\item[\upshape(Sp1)] \(\Sigma\) is symmetric; that is, for any \((a_1,a_2,a_3)
 \in \Sigma\) and each permutation \(\sigma\) of \(\{1,2,3\}\),
 \((a_{\sigma(1)},a_{\sigma(2)}, a_{\sigma(3)})\) belongs to \(\Sigma\) as well;
\item[\upshape(Sp2)] \(a+b+c = p+q+r\) for all \((a,b,c),(p,q,r) \in \Sigma\);
\item[\upshape(Sp3)] if \(a,b,c,d \in S\) are such that \((a,b,c+d),(a+b,c,d)
 \in \Sigma\), then \((a,b+c,d) \in \Sigma\) as well;
\item[\upshape(Sp4a)] for any \((a,b,c) \in \Sigma\) there exists \((x,y,z) \in
 \Sigma\) such that \(a = x+y\) and \(z = b+c\);
\item[\upshape(Sp4b)] if \((a,b,c),(p,q,r) \in \Sigma\) are such that \(a+b =
 p+q\) and \(c = r\); or \((a,b),(p,q) \in \Sigma'\) and \(r = 0_S\) where
 \(\Sigma' \df \{(f,g+h)\dd\ (f,g,h) \in \Sigma\}\), then there are
 \[(x,y,z) \in \Sigma \cup (\Sigma' \times \{0_S\}) \qquad and \qquad
 (u,v,w) \in \Sigma \cup (\{0_S\} \times \Sigma')\]
 or conversely \UP(that is, \((x,y,z) \in \Sigma \cup (\{0_S\} \times \Sigma')\)
 and \((u,v,w) \in \Sigma \cup (\Sigma' \times \{0_S\})\)\UP) for which \(x+z =
 p\), \(y = q+r\), \(u+w = q\), \(v = p+r\) and \(a = x+u\) and \(b = z+w\).
\end{enumerate}
\end{pro}
\begin{proof}
Items (Sp0) and (Sp1) are obvious. Further, fix a partition \(\{k,\ell,m\}\) of
\(\bbB\). Then \(\mu(k)+\mu(\ell)+\mu(m) = \mu(\unit)\), which is followed by
(Sp2). There are two disjoint proper elements \(g\) and \(h\) of \(\bbB\) such
that \(\BooV{g}{h} = k\). Consequently, \((\mu(g),\mu(h),\mu(\BooV{\ell}{m}))
\in \Sigma\) and \(\mu(k) = \mu(g)+\mu(h)\) and \(\mu(\BooV{\ell}{m}) =
\mu(\ell)+\mu(m)\), which shows (Sp4a).\par
Now we pass to (Sp3). Let \(a,b,c,d \in S\) be such that both the triples
\((a+b,c,d)\) and \((a,b,c+d)\) belong to \(\Sigma\). Then that there are two
partitions \(\{x,y,z\}\) and \(\{p,q,r\}\) of \(\bbB\) such that \((\mu(x),
\mu(y),\mu(z)) = (a+b,c,d)\) and similarly \((\mu(p),\mu(q),\mu(r)) = (a,b,
c+d)\). Since \(\mu(x) = \mu(p)+\mu(q)\) and \(\mu(\BooD{\unit}{x}) = \mu(y) +
\mu(z) = \mu(r)\), we infer from (4EP) that there are two disjoint proper
elements \(s\) and \(t\) of \(\bbB\) such that \((\mu(s),\mu(t)) = (a,b)\) and
\(x = \BooV{s}{t}\). Then \(\{s,\BooV{t}{y},z\}\) is a partition of \(\bbB\) and
hence \((a,b+c,d) = (\mu(s),\mu(\BooV{t}{y}),\mu(z))\) belongs to
\(\Sigma\).\par
Finally, we pass to (Sp4b). To this end, we fix two triples \((a,b,c)\) and
\((p,q,r)\) from \(\Sigma\) such that \(p+q = a+b\) and \(r = c\). Take two
partitions \(\{k,\ell,m\}\) and \(\{f,g,h\}\) of \(\bbB\) for which \((\mu(k),
\mu(\ell),\mu(m)) = (a,b,c)\) and \((\mu(f),\mu(g),\mu(h)) = (p,q,r)\). Since
then \(\mu(\BooV{k}{\ell}) = \mu(f)+\mu(g)\) and
\(\mu(\BooD{\unit}{(\BooV{k}{\ell})}) = \mu(h)\), we infer from (4EP) that
the partition \(\{f,g,h\}\) may be chosen so that
\begin{equation}\label{eqn:PN-aux10}
\BooV{f}{g} = \BooV{k}{\ell}.
\end{equation}
Thus, we do assume that \eqref{eqn:PN-aux10} holds. Now set \(\alpha =
\BooW{f}{k}\), \(\beta = \BooW{f}{\ell}\), \(\gamma = \BooW{g}{k}\) and \(\delta
= \BooW{g}{\ell}\). Note that all \(\alpha,\beta,\gamma,\delta\) are pairwise
disjoint (however, some of them can equal \(\zero\)) and it follows from
\eqref{eqn:PN-aux10} that
\begin{itemize}
\item \(f = \BooV{\alpha}{\beta}\) and \(g = \BooV{\gamma}{\delta}\); and
\item \(k = \BooV{\alpha}{\gamma}\) and \(\ell = \BooV{\beta}{\delta}\).
\end{itemize}
In particular, setting
\[\begin{cases}
x = \mu(\alpha)\\
y = \mu(\BooV{g}{h})\\
z = \mu(\beta)\\
u = \mu(\gamma)\\
v = \mu(\BooV{f}{h})\\
w = \mu(\delta)
\end{cases},\]
we obtain:
\begin{itemize}
\item \(p = \mu(f) = x+z\) and \(q = \mu(g) = u+w\);
\item \(a = \mu(k) = x+u\) and \(b = \mu(\ell) = z+w\);
\item \(y = q+r\) and \(v = p+r\); and
\item if both \(\alpha\) and \(\beta\) are non-zero, then \(\{\alpha,
 \BooV{g}{h},\beta\}\) is a partition of \(\bbB\) and, consequently, \((x,y,z)
 \in \Sigma\),
\item otherwise, if either \(\alpha\) or \(\beta\) is equal to \(\zero\), then
 \((x,y,z) \in (\{0_S\} \times \Sigma') \cup (\Sigma' \times \{0_S\})\);
\item similarly: if both \(\gamma\) and \(\delta\) are non-zero, then
 \(\{\gamma,\BooV{f}{h},\delta\}\) is a partition of \(\bbB\) and therefore
 \((u,v,w) \in \Sigma\),
\item otherwise, if either \(\gamma\) or \(\delta\) equals \(\zero\), then
 \((u,v,w) \in (\{0_S\} \times \Sigma') \cup (\Sigma' \times \{0_S\})\),
\item if \(\alpha = \zero\), then \(\beta\) and \(\gamma\) are non-zero;
 similarly: if \(\beta = \zero\), then \(\alpha\) and \(\delta\) are non-zero;
 if \(\gamma = \zero\), then \(\alpha\) and \(\delta\) are non-zero; and finally
 if \(\delta = \zero\), then \(\beta\) and \(\gamma\) are non-zero---this shows
 that \((x,y,z) \in \Sigma \cup (\Sigma' \times \{0_S\})\) and \((u,v,w) \in
 \Sigma \cup (\{0_S\} \times \Sigma')\) or conversely.
\end{itemize}
Similarly, if \((a,b)\) and \((p,q)\) are in \(\Sigma'\), then there are two
partitions \(\{k,\ell\}\) and \(\{f,g\}\) of \(\bbB\) such that \((a,b,f,g) =
(\mu(k),\mu(\ell),\mu(f),\mu(g))\). Setting additionally \(h = m \df 0_S\), we
may repeat the above proof (that starts from \eqref{eqn:PN-aux10}) to get
the same conclusion. The details are left to the reader.
\end{proof}

\begin{dfn}{PN-homspec}
A \emph{\homspec{3}} in a semigroup \(S\) is an arbitrary set \(\Sigma\)
satisfying all the conditions (Sp0)--(Sp4b) specified in \PRO{PN-any-spectrum}.
Whenever \(\Sigma\) is a \homspec{3}, we define \(e(\Sigma)\) as the unique
element of \(S\) such that \(e(\Sigma) = a+b+c\) for all \((a,b,c) \in \Sigma\)
(cf. (Sp2)), and denote by \(\Sigma'\) the set of all pairs of the form
\((a,b+c)\) where \((a,b,c)\) runs over all triples from \(\Sigma\) (cf.
(Sp4b)).
\end{dfn}

\begin{dfn}{PN-comp-spec}
Let \(\Sigma\) be a \homspec{3} in a semigroup \(S\). An \(S\)-valued measure
\(\mu\) defined on a Boolean subalgebra \(\aaA\) of \(\bbB\) is said to be
\emph{compatible} with \(\Sigma\) if a respective condition is fulfilled from
the following list:
\begin{itemize}
\item when \(\aaA = \{\zero,\unit\}\): \(\mu(\unit) = e(\Sigma)\); or
\item when \(\aaA = \{\zero,\unit,a,\BooD{\unit}{a}\}\) (with proper \(a \in
 \bbB\)): \((\mu(a),\mu(\BooD{\unit}{a})) \in \Sigma'\) (cf. \DEF{PN-homspec});
 or
\item when \(|\aaA| > 4\): \((\mu(a),\mu(b),\mu(c)) \in \Sigma\) for any
 partition \(\{a,b,c\} \subset \aaA\) of \(\bbB\).
\end{itemize}
\end{dfn}

\begin{thm}{PN-ext-spec}
Let \(\Sigma\) and \(\mu\) be, respectively, a \homspec{3} in a semigroup \(S\)
and an \(S\)-valued measure defined on a finite Boolean subalgebra \(\aaA\) of
\(\bbB\). \TFCAE
\begin{enumerate}[\upshape(i)]
\item \(\mu\) extends to a homogeneous measure \(\lambda\dd \bbB \to S\) with
 \(\Sigma(\lambda) = \Sigma\);
\item \(\mu\) is compatible with \(\Sigma\).
\end{enumerate}
In particular, for any set \(\Sigma \subset S \times S \times S\), there exists
a homogeneous \(S\)-valued measure \(\rho\) on \(\bbB\) with \(\Sigma(\rho) =
\Sigma\) iff \(\Sigma\) is a \homspec{3}.
\end{thm}

The proof of the above result is quite elementary but a little bit complicated
and needs some preparations. We proceed it by three auxiliary lemmas. In each of
them we assume \(\Sigma\) is a \homspec{3} in \(S\).

\begin{lem}{PN-aux0}
\begin{enumerate}[\upshape(A)]
\item If \((a,b) \in \Sigma'\), then \((b,a) \in \Sigma'\) as well.
\item If \(\mu\) is an \(S\)-valued measure defined on a Boolean subalgebra
 \(\aaA\) of \(\bbB\) that is compatible with \(\Sigma\), then:
 \begin{itemize}
 \item \(\mu(\unit) = e(\Sigma)\);
 \item \((\mu(a),\mu(b)) \in \Sigma'\) for any partition \(\{a,b\} \subset
  \aaA\) of \(\bbB\).
 \end{itemize}
\end{enumerate}
\end{lem}
\begin{proof}
To show (A), take \((p,q,r) \in \Sigma\) such that \(a = p\) and \(b = q+r\) and
apply (Sp4a) to find \((x,y,z) \in \Sigma\) for which \(p = x+y\) and \(z =
q+r\). Then \((z,x,y) \in \Sigma\) as well (by (Sp1)) and hence \((z,x+y) =
(b,a) \in \Sigma'\).\par
We pass to (B). Since \(p+q = e(\Sigma)\) for any \((p,q) \in \Sigma'\), we
easily get that \(\mu(\unit) = e(\Sigma)\) (cf. (Sp2)). So, if \(|\aaA| < 8\),
we have nothing more to do. Thus, assume that \(\aaA\) has at least \(8\)
elements. This means that there exists a partition of \(\bbB\) consisting of
three elements of \(\aaA\), say \(\{p,q,r\} \subset \aaA\). Then for any
partition \(\{a,b\} \subset \aaA\) of \(\bbB\), at least one of the elements
\(\BooW{a}{p},\BooW{a}{q},\BooW{a}{r}\) is non-zero, say \(a' \df \BooW{a}{p}\).
If \(a' \neq a\), then \(\{b,a',\BooD{a}{a'}\}\) is a partition of \(\bbB\)
contained in \(\aaA\) and it follows from compatibility of \(\mu\) with
\(\Sigma\) that \((\mu(b),\mu(a'),\mu(\BooD{a}{a'})) \in \Sigma\). Consequently,
\((\mu(b),\mu(a')+\mu(\BooD{a}{a'})) = (\mu(b),\mu(a))\) belongs to
\(\Sigma'\) and an application of part (A) finishes the proof in that case.\par
Finally, assume \(a' = a\). Then \(\BooW{(\BooV{q}{r})}{b} = \BooV{q}{r}\). So,
\(\BooW{q}{b} = q\) and \(\{a,q,\BooD{b}{q}\}\) is a partition of \(\bbB\)
contained in \(\aaA\), and thus (similarly as in the previous case) \((\mu(a),
\mu(q)+\mu(\BooD{b}{q})) = (\mu(a),\mu(b))\) belongs to \(\Sigma'\).
\end{proof}

\begin{lem}{PN-aux1}
Let \(\nu\) be an \(S\)-valued measure compatible with \(\Sigma\) that is
defined on a finite Boolean subalgebra \(\eeE\) of \(\bbB\) generated by
a partition \(\{e_0,\ldots,e_N\}\ (N \geq 0)\). Further, let \(J\) be
a non-empty subset of \(\{0,\ldots,N\}\) and for each \(j \in J\):
\begin{itemize}
\item let \(b_j\) and \(c_j\) be two disjoint proper elements of \(\bbB\) such
 that \(\BooV{b_j}{c_j} = e_j\);
\item if \(N > 0\), let \((\beta_j,\gamma_j,\delta_j) \in \Sigma\) \UP(resp. let
 \((\beta_j,\gamma_j) \in \Sigma'\) provided that \(N = 0\)\UP) be such that
 \(\beta_j+\gamma_j = \nu(e_j)\) and, if \(N > 0\), \(\delta_j =
 \nu(\BooD{\unit}{e_j})\).
\end{itemize}
Then a unique measure \(\tilde{\nu}\dd \ddD \to S\) where \(\ddD\) is
the Boolean algebra generated by \(\eeE \cup \{b_j,c_j\dd\ j \in J\}\) that
extends \(\nu\) and satisfies \(\tilde{\nu}(b_j) = \beta_j\) and
\(\tilde{\nu}(c_j) = \gamma_j\) \UP(for all \(j \in J\)\UP) is compatible with
\(\Sigma\) as well.
\end{lem}
\begin{proof}
By a simple induction argument (on the size of \(J\)), we may and do assume that
\(J\) is a singleton, and moreover that \(J = \{0\}\). For simplicity, we will
write \(b,c,\beta,\gamma,\delta\) instead of (resp.) \(b_0,c_0,\beta_0,\gamma_0,
\delta_0\). We consider two separate cases. If \(N = 0\), then \(\eeE = \{\zero,
\unit\}\) and \(\BooV{b}{c} = \unit = e_0\). So, \(c = \BooD{\unit}{b}\) and
thus the conclusion easily follows (thanks to \LEM{PN-aux0}). Below we assume
that \(N > 0\).\par
When \(N > 0\), the set \(F \df \{b,c,e_1,\ldots,e_N\}\) is a unique (up to
arrangement of its elements) partition of \(\bbB\) that generates \(\ddD\). What
is more, for any non-zero element \(x\) of \(\ddD\) there exists a unique
non-empty set \(I = I(x) \subset F\) such that \(x = \bigvee_{z \in I} z\). So,
we infer that \(\tilde{\nu}\dd \ddD \to S\) defined by the rules:
\begin{itemize}
\item \(\tilde{\nu}(\zero) = 0_S\) and \(\tilde{\nu}(e_k) = \nu(e_k)\) for \(k =
 1,\ldots,N\),
\item \(\tilde{\nu}(b) = \beta\) and \(\tilde{\nu}(c) = \gamma\),
\item \(\tilde{\nu}(x) = \sum_{w \in I(x)} \tilde{\nu}(w)\) for non-zero \(x \in
 \ddD \setminus F\)
\end{itemize}
is a correctly defined \(S\)-valued measure that extends \(\nu\) (because
\(\beta+\gamma = \nu(e_0)\)). So, we only need to check that it is compatible
with \(\Sigma\). To this end, note that \(|\ddD| > 4\) and fix a partition
\(\{x,y,z\} \subset \ddD\) of \(\bbB\). If each of the sets \(I(x),I(y),I(z)\)
either is disjoint from \(\{b,c\}\) or contains both \(b\) and \(c\), then \(x,
y,z \in \eeE\) and the conlusion follows from compatibility of \(\nu\) with
\(\Sigma\). Hence, we may and do assume (thanks to (Sp1)) that \(b \in I(x)\)
and \(c \in I(y)\). Set \(X \df I(x) \setminus \{b\}\), \(Y \df I(y) \setminus
\{c\}\) and \(Z \df I(z)\) and observe that \(X,Y,Z\) are pairwise disjoint and
their union coincides with \(\{e_1,\ldots,e_N\}\) (however, each of \(X\) and
\(Y\) can be empty). For simplicity, set \(x' \df \bigvee_{w \in X} w\) and \(y'
\df \bigvee_{w \in Y} w\) and \(p \df \tilde{\nu}(x')\), \(q \df
\tilde{\nu}(y')\) and \(r \df \tilde{\nu}(z)\) (note that \(x', y', z \in \eeE\)
and hence \((p,q,r) = (\nu(x'),\nu(y'),\nu(z))\)). Observe that \(p+q+r =
\delta\). We consider three cases:
\begin{itemize}
\item If \(X = Y = \varempty\), then \(x = b\), \(y = c\), \(z =
 \BooD{\unit}{e_0}\) and, consequently,
 \[(\tilde{\nu}(x),\tilde{\nu}(y),\tilde{\nu}(z)) = (\beta,\gamma,\delta)\]
 belongs to \(\Sigma\) (by the assumption of the lemma).
\item If \(X \neq \varempty = Y\), then \(x' \neq \zero = y'\), \(q = 0_S\) and
 both \((\gamma,\beta,p+r) = (\gamma,\beta,\delta)\) and \((\gamma+\beta,p,r) =
 (\nu(e_0),\nu(x'),\nu(z))\) belong to \(\Sigma\) (thanks to (Sp1) and
 compatibility of \(\nu\) with \(\Sigma\)). So, it follows from (Sp3) that
 \((\gamma,\beta+p,r) \in \Sigma\) as well, but \((\gamma,\beta+p,r) =
 (\tilde{\nu}(y),\tilde{\nu}(x),\tilde{\nu}(z))\). Thus, another application of
 (Sp1) yields \((\tilde{\nu}(x),\tilde{\nu}(y),\tilde{\nu}(z)) \in \Sigma\).
 A similar argument shows the last relation when \(X = \varempty \neq Y\).
\item Finally, assume that both \(X\) and \(Y\) are non-empty. Then both \(x'\)
 and \(y'\) are proper, and both \((\beta,\gamma,q+(p+r)) = (\beta,\gamma,
 \delta)\) and \((\beta+\gamma,q,p+r) = (\nu(e_0),\nu(y'),\nu(\BooV{x'}{z}))\)
 belong to \(\Sigma\). So, it follows from (Sp3) that \((\beta,\gamma+q,p+r) \in
 \Sigma\). Consequently, \((r+p,\beta,\gamma+q) \in \Sigma\) (by (Sp1)). But
 also \((r,p,\beta+(\gamma+q)) = (\nu(z),\nu(x'),\nu(\BooV{e_0}{y'}))\) belongs
 to \(\Sigma\) (since \(\nu\) is compatible with it). Therefore, we conclude
 from (Sp3) that \((r,p+\beta,\gamma+q) \in \Sigma\), and from (Sp1) that
 \((\tilde{\nu}(x),\tilde{\nu}(y),\tilde{\nu}(z)) = (p+\beta,q+\gamma,r) \in
 \Sigma\) as well.
\end{itemize}
All the above cases have led us to the same conclusion, which finishes
the proof.
\end{proof}

\begin{lem}{PN-aux2}
Let \(\rho\) be an \(S\)-valued measure compatible with \(\Sigma\) that is
defined on a finite Boolean subalgebra \(\aaA\) of \(\bbB\). Further, let \(w\)
be a non-zero element of \(\aaA\) and:
\begin{itemize}
\item if \(w\) is proper, let \((\alpha,\beta,\gamma) \in \Sigma\) be
 such that \(\alpha+\beta = \rho(w)\) and \(\gamma = \rho(\BooD{\unit}{w})\); or
\item if \(w = \unit\), let \((\alpha,\beta)\) be an arbitrary element of
 \(\Sigma'\).
\end{itemize}
Then \(\rho\) can be extended to an \(S\)-valued measure \(\tilde{\rho}\)
defined on a finite Boolean subalgebra \(\ddD \supset \aaA\) of \(\bbB\) such
that \(\tilde{\rho}\) is compatible with \(\Sigma\) and there exist two disjoint
proper elements \(u, v \in \ddD\) for which \(w = \BooV{u}{v}\) and
\((\tilde{\rho}(u),\tilde{\rho}(v)) = (\alpha,\beta)\).
\end{lem}
\begin{proof}
Since \(\aaA\) is finite, there is a partition \(\{a_0,\ldots,a_M\} \subset
\aaA\) of \(\bbB\) (where \(M \geq 0\)) that generates \(\aaA\). Further, there
exists a non-empty set \(Q \subset \{0,\ldots,M\}\) such that \(w =
\bigvee_{j \in Q} a_j\). Rearranging the elements \(a_0,\ldots,a_M\) (if
needed), we may and do assume that \(Q = \{M-k,\ldots,M\}\) where \(0 \leq k
\leq M\). The proof goes by induction on the size of \(Q\).\par
First assume that \(k = 0\) (that is, that \(Q\) is a singleton). Then \(w =
a_M\). Take two arbitrary disjoint proper elements \(u\) and \(v\) such that \(w
= \BooV{u}{v}\). Now the assertion follows from \LEM{PN-aux1} (applied with
\((b_M,c_M,\beta_M,\gamma_M,\delta_M) = (u,v,\alpha,\beta,\gamma)\);
\(\delta_M\) is involved only when \(w \neq \unit\)).\par
Now assume \(k > 0\) and that the result is valid when \(Q\) has \(k-1\)
elements. Let \(a_{M-1}' \df \BooV{a_{M-1}}{a_M}\) and denote by \(\aaA'\)
the Boolean subalgebra of \(\bbB\) generated by \(a_0,\ldots,a_{M-2},a_{M-1}'\).
Then \(w \in \aaA' \subset \aaA\). It follows from the induction hypothesis that
there exist
\begin{itemize}
\item an \(S\)-valued measure \(\rho'\) defined on a finite Boolean subalgebra
 \(\ddD' \supset \aaA'\) of \(\bbB\) that is compatible with \(\Sigma\) and
 extends \(\rho\restriction{\aaA'}\); and
\item two disjoint proper elements \(u\) and \(v\) of \(\ddD'\) such that \(w
 = \BooV{u}{v}\) and \((\rho'(u),\rho'(v)) = (\alpha,\beta)\).
\end{itemize}
Reducing \(\ddD'\) if necessary, we may and do assume that:
\begin{itemize}
\item[\((\star)\)] \(\ddD'\) is generated by \(a_0,\ldots,a_{M-2},a_{M-1}',u,
 v\).
\end{itemize}
Now set \(u' \df \BooW{u}{a_{M-1}'}\), \(v' \df \BooW{v}{a_{M-1}'}\), \(p \df
\rho'(u')\) and \(q \df \rho'(v')\). Then \(\BooV{u'}{v'} = \BooW{w}{a_{M-1}'}
= a_{M-1}'\) (where the last equality follows from the inclusion \(\{M-1,M\}
\subset Q\)). Additionally, set \(r \df \rho(\BooD{\unit}{a_{M-1}'})\).\par
First assume that \(u' = \zero\) (resp. \(v' = \zero\)). Then \(v' = a_{M-1}'\)
(resp. \(u' = a_{M-1}'\)) and it follows from \((\star)\) that \(a_{M-1}'\) is
a minimal non-zero element of \(\ddD'\) and therefore \(a_{M-1}'\) belongs to
a (unique) partition \(\{d_0,\ldots,d_s\} \subset \ddD'\) of \(\bbB\) that
generates \(\ddD'\), say \(d_s = a_{M-1}'\). In particular, \LEM{PN-aux1} can be
applied with the following data: \((\eeE,\nu) = (\ddD',\rho')\),
\((e_0,\ldots,e_N) = (d_0,\ldots,d_s)\), \(J = \{s\}\) and \((b_s,c_s,\beta_s,
\gamma_s,\delta_s) = (a_{M-1},a_M,\rho(a_{M-1}),\rho(a_M),r)\) (\(\delta_s\) is
involved only when \(a_{M-1}' \neq \unit\)). In this way we obtain an extension
\(\tilde{\rho}\) of \(\rho'\) that extends also \(\rho\), which finishes
the proof in that case.\par
Finally, we assume that both \(u'\) and \(v'\) are proper. We infer that
\begin{equation}\label{eqn:PN-aux21}\begin{cases}
p+q = \rho(a_{M-1}') = \rho(a_{M-1})+\rho(a_M)\\
(p,q,r),(\rho(a_{M-1}),\rho(a_M),r) \in \Sigma & \UP{if } a_{M-1}' \neq \unit\\
(p,q),(\rho(a_{M-1}),\rho(a_M)) \in \Sigma' \UP{ and } r = 0_S & \UP{if }
a_{M-1}' = \unit
\end{cases}.\end{equation}
Since \(\BooV{u'}{v'} = a_{M-1}'\), it follows from \((\star)\) that \(u'\) and
\(v'\) belong to a (unique) partition \(\{d_0,\ldots,d_s\}\) of \(\bbB\) that
generates \(\ddD'\), say \(d_0 = u'\) and \(d_1 = v'\) (so, \(s > 0\)).
A combination of \eqref{eqn:PN-aux21} and (Sp4b) yields two triples \((\beta_0,
\delta_0,\gamma_0)\) and \((\beta_1,\delta_1,\gamma_1)\) from \(\Sigma \cup
(\{0_S\} \times \Sigma') \cup (\Sigma' \times \{0_S\})\) such that one of these
triples is from \(\Sigma \cup (\{0_S\} \times \Sigma')\) and the other from
\(\Sigma \cup (\Sigma' \times \{0_S\})\) and
\begin{equation}\label{eqn:PN-aux20}\begin{cases}
\beta_0+\gamma_0 = p\\
\beta_1+\gamma_1 = q\\
\beta_0+\beta_1 = \rho(a_{M-1})\\
\gamma_0+\gamma_1 = \rho(a_M)\\
\delta_0 = q+r\\
\delta_1 = p+r
\end{cases}.\end{equation}
Without loss of generality, we may and do assume that
\begin{equation}\label{eqn:PN-aux22}
(\beta_0,\delta_0,\gamma_0) \in \Sigma \cup (\Sigma' \times \{0_S\})
\qquad \UP{and} \qquad
(\beta_1,\delta_1,\gamma_1) \in \Sigma \cup (\{0_S\} \times \Sigma').
\end{equation}
Now if we take two arbitrary disjoint proper elements \(u''\) and \(v''\) of
\(\bbB\) such that \(\BooV{u''}{v''} = a_{M-1}'\), then the assignments
\[\begin{cases}
d_j \mapsto d_j & \UP{for } 1 < j \leq s\\
u' \mapsto u''\\
v' \mapsto v''
\end{cases}\]
extend to a Boolean algebra isomorphism \(\Phi\) from \(\ddD'\) onto some
\(\ddD''\) such that \(\Phi(x) = x\) for all \(x \in \aaA'\). Hence a measure
\(\rho''\dd \ddD'' \ni y \mapsto \rho'(\Phi^{-1}(y)) \in S\) is compatible with
\(\Sigma\) and extends \(\rho\restriction{\aaA'}\) as well. So, we may work with
\(\rho''\) instead of \(\rho'\). In this way we may and do assume (as
\(\BooV{u'}{v'} = \BooV{a_{M-1}}{a_M}\)) that the four elements
\[b_0 = \BooW{u'}{a_{M-1}}, \quad c_0 = \BooW{u'}{a_M}, \quad
b_1 = \BooW{v'}{a_{M-1}}, \qquad c_1 = \BooW{v'}{a_M}\]
satisfy all the respective conditions:
\begin{itemize}
\item if \((\beta_0,\delta_0,\gamma_0) \in \Sigma\), then
 \(b_0\) and \(c_0\) are non-zero; and
\item similarly, if \((\beta_1,\delta_1,\gamma_1) \in \Sigma\), then \(b_1\) and
 \(c_1\) are non-zero; and
\item if \((\beta_0,\delta_0,\gamma_0) \notin \Sigma\), then
 \(b_0 = u'\) and \(c_0 = \zero\) (cf. \eqref{eqn:PN-aux22}); and
\item similarly, if \((\beta_1,\delta_1,\gamma_1) \notin \Sigma\), then \(b_1 =
 \zero\) and \(c_1 = v'\) (cf. \eqref{eqn:PN-aux22}).
\end{itemize}
Now let \(I\) consist of all \(j \in \{0,1\}\) such that both \(b_j\) and
\(c_j\) are non-zero. If \(I = \varempty\), then \(u' = a_{M-1}\), \(v' = a_M\)
and \(\rho'(u') = \rho(a_{M-1})\) and \(\rho'(v') = \rho(a_M)\). Thus, in such
a case the proof is finished by setting \(\ddD = \ddD'\) and \(\tilde{\rho} =
\rho'\). Therefore it remains to prove the lemma when \(I \neq \varempty\). In
that case we apply \LEM{PN-aux1} with the following data: \((\eeE,\nu) = (\ddD',
\rho')\), \((e_0,\ldots,e_N) = (d_0,\ldots,d_s)\), \(J = I\), and for any \(j
\in I\), \(b_j\), \(c_j\), \(\beta_j\), \(\gamma_j\) and \(\delta_j\) are as
specified above (note that \(\BooV{b_j}{c_j} = d_j\) and \((\beta_j,\gamma_j,
\delta_j) \in \Sigma\) for \(j \in I\)). In this way we obtain a measure
\(\tilde{\rho}\dd \ddD \to S\) compatible with \(\Sigma\) that extends
\(\rho'\). Observe that, whatever \(I\) is, each of \(b_0,c_0,b_1,c_1\) belongs
to \(\ddD\) and \(\tilde{\rho}(b_j) = \beta_j\) and \(\tilde{\rho}(c_j) =
\gamma_j\) for \(j=0,1\) (thanks to \eqref{eqn:PN-aux20}). Consequently,
\(a_{M-1} = \BooV{b_0}{b_1}\) and \(a_M = \BooV{c_0}{c_1}\) are in \(\ddD\) as
well, which yields that \(\aaA \subset \ddD\). Finally, \(\tilde{\rho}(a_{M-1})
= \tilde{\rho}(b_0)+\tilde{\rho}(b_1) = \beta_0+\beta_1 = \rho(a_{M-1})\), by
\eqref{eqn:PN-aux20}. Similarly, \(\tilde{\rho}(a_M) = \rho(a_M)\), and we are
done.
\end{proof}

We are now ready to give

\begin{proof}[Proof of \THM{PN-ext-spec}]
To shorten statements, we will call (here in this proof) an \(S\)-valued measure
\emph{\(\Sigma\)-finite} if it is compatible with \(\Sigma\) and defined on
a finite Boolean subalgebra of \(\bbB\).\par
We only need to show that (i) is implied by (ii). To this end, we fix
a \(\Sigma\)-finite measure \(\mu\) and arrange:
\begin{itemize}
\item all proper elements of \(\bbB\) into a sequence \((a_n)_{n=1}^{\infty}\);
\item all triples from \(\Sigma\) into \((\alpha_n,\beta_n,
 \gamma_n)_{n=1}^{\infty}\);
\item all pairs from \(\Sigma'\) into \((\eta_n,\nu_n)_{n=1}^{\infty}\)
\end{itemize}
(we do not require that these sequences are one-to-one). Starting from
\(\aaA_0 \df \aaA\) and \(\mu_0 \df \mu\) we will now construct two sequences
\(\aaA_1,\aaA_2,\ldots\) and \(\mu_1,\mu_2,\ldots\) such that all the following
conditions are fulfilled:
\begin{enumerate}[(1\({}_n\))]
\item \(\aaA_n\) is a finite Boolean subalgebra of \(\bbB\) such that \(\aaA_n
 \supset \aaA_{n-1} \cup \{a_1,\ldots,a_n\}\);
\item \(\mu_n\) is a \(\Sigma\)-finite measure extending \(\mu_{n-1}\) whose
 domain coincides with \(\aaA_n\);
\item for any \(k=1,\ldots,n\) there is a proper element \(f \in \aaA_n\) such
 that \(\mu(f) = \eta_k\) and \(\mu(\BooD{\unit}{f}) = \nu_k\);
\item for any \(j,k=1,\ldots,n\), if \(\mu_n(a_j) = \alpha_k+\beta_k\) and
 \(\mu_n(\BooD{\unit}{a_j}) = \gamma_k\), then there are disjoint proper
 elements \(c,d \in \aaA_n\) for which \(\BooV{c}{d} = a_j\), \(\mu(c) =
 \alpha_k\) and \(\mu(d) = \beta_k\).
\end{enumerate}
So, assume \(n > 0\) and \(\aaA_{n-1}\) and \(\mu_{n-1}\) have already been
defined. To construct \(\aaA_n\) and \(\mu_n\), it is sufficient to describe how
to extend a \(\Sigma\)-finite measure to a \(\Sigma\)-finite measure:
\begin{enumerate}[(ext1)]
\item whose domain contains an arbitrarily given element of \(\bbB\) (in this
 way we will first take care of \((1_n)\)); or
\item such that \((3_n)\) holds for a fixed \(k\) (in this way, after a finite
 number of steps, we will take care of \((3_n)\)); or
\item such that \((4_n)\) holds for fixed \(j\) and \(k\) (similarly---in this
 way, after a finite number of steps we will take care of \((4_n)\) and we will
 call our final Boolean subalgebra and \(\Sigma\)-finite measure \(\aaA_n\) and
 \(\mu_n\), respectively).
\end{enumerate}
To this end, assume \(\rho\) is a \(\Sigma\)-finite measure whose domain is
\(\eeE\). Fix a partition \(\{e_0,\ldots,e_N\} \subset \eeE\) of \(\bbB\) which
generates \(\eeE\) (\(N \geq 0\)). Note that both (ext2) and (ext3) are covered
by \LEM{PN-aux2}. So, we only need to show (ext1), which we do below.\par
Fix arbitrary \(w \in \bbB \setminus \eeE\) and let \(J\) consist of all \(j \in
\{0,\ldots,N\}\) such that \(b_j \df \BooW{w}{e_j}\) is non-zero and differs
from \(e_j\). (\(J\) is non-empty, as \(w \notin \eeE\).) Observe that
\begin{equation}\label{eqn:PN-aux30}
\BooD{w}{(\bigvee_{j \in J} b_j)} \in \eeE.
\end{equation}
For \(j \in J\) set \(c_j \df \BooD{e_j}{b_j}\) and note that both \(b_j\) and
\(c_j\) are proper and \(\BooV{b_j}{c_j} = e_j\). If \(e_j = \unit\) (that is,
if \(N = 0\)), we use (Sp0) to conclude that there exists a pair \((\beta_j,
\gamma_j)\) in \(\Sigma'\). Otherwise (that is, if \(e_j \neq \unit\)), we apply
\LEM{PN-aux0} to infer that \((\rho(\BooD{\unit}{e_j}),\rho(e_j)) \in \Sigma'\)
and, consequently, there is a triple \((\delta_j,\beta_j,\gamma_j) \in \Sigma\)
such that \(\rho(\BooD{\unit}{e_j}) = \delta_j\) and \(\rho(e_j) = \beta_j +
\gamma_j\). Then also \((\beta_j,\gamma_j,\delta_j) \in \Sigma\), by (Sp1).\par
Now we apply \LEM{PN-aux1} to get a \(\Sigma\)-finite measure \(\rho'\) that
extends \(\rho\) and whose domain contains all \(b_j\) and \(c_j\) where \(j \in
S\). (More precisely, \(\rho'(b_j) = \beta_j\) and \(\rho'(c_j) = \gamma_j\).)
Then also \(w\) is in this domain (thanks to \eqref{eqn:PN-aux30}), which
finishes the proof of (ext1).\par
In this way we have constructed the sequences \((\aaA_n)_{n=0}^{\infty}\) and
\((\mu_n)_{n=0}^{\infty}\) such that conditions \((1_n)\)--\((4_n)\) hold for
all \(n > 0\). Properties \((1_n)\)--\((2_n)\) allow us to define correctly
an \(S\)-valued measure \(\lambda\) on \(\bbB\) that extends all \(\mu_n\). (In
particular, \(\lambda\) extends \(\mu\).) It is clear that \(\lambda\) is
compatible with \(\Sigma\), which means that \(\Sigma(\lambda) \subset \Sigma\).
So, it remains to verify the reverse inclusion and that \(\lambda\) has (4EP)
(by \THM{PN-4EP}).\par
Let \((p,q,r)\) be an arbitrary element of \(\Sigma\). Then there exist indices
\(k\) and \(n\) such that \((p,q,r) = (\alpha_k,\beta_k,\gamma_k)\) and
\((p+q,r) = (\eta_n,\nu_n)\) (cf. (Sp1) and \LEM{PN-aux0}). An application of
\((3_n)\) gives us a proper element \(f \in \bbB\) such that \(\lambda(f) =
p+q\) and \(\lambda(\BooD{\unit}{f}) = r\). Then there is \(j\) such that \(f =
a_j\). Finally, an application of \((4_p)\) with \(p = \max(j,k)\) yields two
disjoint proper elements \(c, d \in \bbB\) such that \(\BooV{c}{d} = f\) and
\((\lambda(c),\lambda(d)) = (p,q)\). We conclude that \(\{c,d,
\BooD{\unit}{f}\}\) is a partition of \(\bbB\) such that \((\lambda(c),
\lambda(d),\lambda(\BooD{\unit}{f})) = (p,q,r)\), which shows that \(\Sigma =
\Sigma(\lambda)\).\par
Finally, let \(\{a,b,c\}\) be a partition of \(\bbB\) and \(d\) be a proper
element of \(\bbB\) such that \eqref{eqn:PN-4EP} holds. Then \((\lambda(a),
\lambda(b),\lambda(c)) \in \Sigma\) and there are two indices \(j\) and \(k\)
such that \(d = a_j\) and \((\lambda(a),\lambda(b),\lambda(c)) = (\alpha_k,
\beta_k,\gamma_k)\). Condition \eqref{eqn:PN-4EP} enables us to make use of
\((4_n)\) with \(n = \max(j,k)\), from which it follows that there are two
disjoint proper elements \(p\) and \(q\) such that \(\BooV{p}{q} = d\) and
\((\lambda(p),\lambda(q)) = (\alpha_k,\beta_k)\ (= (\lambda(a),\lambda(b)))\),
which proves (4EP) and finishes the whole proof of the theorem.
\end{proof}

As a consequence, we obtain the following generalisation of
\cite[Proposition~3.7]{DKMN1}.

\begin{thm}{PN-generic}
Let \(S\) be a semigroup, equipped with the discrete topology, and \(\Sigma\) be
a \homspec{3} in \(S\). Further, let \(M(\Sigma)\) be the set of all
\(S\)-valued measures on \(\bbB\) that are compatible with \(\Sigma\). Finally,
let \(\lambda \in M(\Sigma)\) be a fixed homogeneous measure with
\(\Sigma(\lambda) = \Sigma\) \UP(whose existence was established in
\THM{PN-ext-spec}\UP). Then:
\begin{enumerate}[\upshape(A)]
\item The space \(S^{\bbB}\) of all functions from \(\bbB\) into \(S\), when
 equipped with the pointwise convergence topology, is completely metrizable
 and its covering dimension is 0.
\item The set \(M(\Sigma)\) is closed in \(S^{\bbB}\) and \(M(\Sigma)\) is
 a Polish space.
\item \(M(\Sigma)\) is invariant under any Boolean algebra automorphism \(\Phi\)
 of \(\bbB\); that is, if \(\mu \in M(\Sigma)\), then \(\mu \circ \Phi \in
 M(\Sigma)\) as well.
\item All measures of the form \(\lambda \circ \Phi\) where \(\Phi\) runs over
 all Boolean automorphisms of \(\bbB\) form a dense \(\ggG_{\delta}\) set \(H\)
 in \(M(\Sigma)\). Moreover, \(H\) consists precisely of all homogeneous
 measures \(\mu\) such that \(\Sigma(\mu) = \Sigma\).
\end{enumerate}
In particular, for any \(\mu \in M(\Sigma)\), the conjugacy class \(\{\mu \circ
\Phi\dd\ \Phi \in \OPN{Aut}(\bbB)\}\) of \(\mu\) is a dense \(\ggG_{\delta}\)
set in \(M(\Sigma)\) iff \(\mu\) is homogeneous and \(\Sigma(\mu) = \Sigma\).
\end{thm}
\begin{proof}
Parts (A)--(C) are obvious (recall that both \(\bbB\) and \(\Sigma\) are
countable, by (Sp0)). To prove (D), first use Theorems~\ref{thm:PN-4EP} and
\ref{thm:PN-ext-4EP} to conclude that \(H\) consists precisely of all
homogeneous measures whose trinary spectrum coincides with \(\Sigma\). Then
it is a direct consequence of \THM{PN-ext-spec} that \(H\) is dense in
\(X \df M(\Sigma)\). Finally, to convince oneself that \(H\) is
\(\ggG_{\delta}\), observe that
\begin{itemize}
\item all measures from \(X\) that have (4EP) form a \(\ggG_{\delta}\) set \(Y\)
 in \(X\); and
\item similarly, all measures \(\mu \in X\) with \(\Sigma \subset \Sigma(\mu)\)
 form a \(\ggG_{\delta}\) set \(Z\) in \(X\); and
\item \(H = Y \cap Z\).
\end{itemize}
The additional claim of the theorem easily follows from (D) (as there can be
only one conjugacy class that is both dense and \(\ggG_{\delta}\)).
\end{proof}

\begin{exm}{PN-delta-k}
Let \(\mu\) be a homogeneous \(S\)-valued measure and \(\Sigma \df
\Sigma(\mu)\). For each \(k > 0\) let \(\Delta_k(\mu) = \Delta_k(\Sigma)\)
consist of all \(k\)-tuples of the form
\[(\mu(a_1),\ldots,\mu(a_k))\]
where \(\{a_1,\ldots,a_k\}\) is a partition of \(\bbB\). It is clear that:
\begin{itemize}
\item \(\Delta_k(\mu)\) is a symmetric subset of \(S^k\) (that is, membership
 of a \(k\)-tuple to \(\Delta_k(\mu)\) is independent of the order of
 the entries of this \(k\)-tuple);
\item \(\Delta_1(\mu) = \{e(\Sigma)\}\), \(\Delta_2(\mu) = \Sigma'\) and
 \(\Delta_3(\mu) = \Sigma\).
\end{itemize}
We claim that for \(N > 3\) an \(N\)-tuple \((b_1,\ldots,b_N) \in S^N\) belongs
to \(\Delta_N(\mu)\) iff
\[\Bigl(\sum_{j=1}^{k-1} b_j,b_k,\sum_{j=k+1}^N b_j\Bigr) \in \Sigma\]
for \(k=2,\ldots,N-1\). Indeed, if \(\{a_1,\ldots,a_N\}\) is a partition of
\(\bbB\), then so are
\[\Bigl\{\bigvee_{j=1}^{k-1} a_j,a_k,\bigvee_{j=k+1}^N a_j\Bigr\} \qquad
(k=2,\ldots,N-1),\]
which shows the `only if' part. To see the `if' one, we use (4EP) and induction.
There exists a partition \(\{a_1,a_2,c_1\}\) of \(\bbB\) such that \((\mu(a_1),
\mu(a_2),\mu(c_1)) = (b_1,b_2,\sum_{j=3}^N b_j)\). Now assume we have already
defined a partition
\[\{a_1,a_2,\ldots,a_k,c_{k-1}\}\] such that \(\mu(a_j) =
b_j\) for \(j=1,\ldots,k\) and \(\mu(c_{k-1}) = \sum_{j=k+1}^N b_j\) (where \(2
\leq k < N-1\); if \(k = N-1\), it is sufficient to set \(a_N = c_{N-2}\) to
finish the proof). Since \((x,y,z) \df (\sum_{j=1}^k b_j,b_{k+1},\sum_{j=k+2}^N
b_j)\) belongs to \(\Sigma(\mu)\), we infer that there is a partition \(\{d_1,
d_2,d_3\}\) of \(\bbB\) such that \((x,y,z) = (\mu(d_1),\mu(d_2),\mu(d_3))\).
Observe that \(\mu(d_1) = \mu(\bigvee_{j \leq k} a_j) =
\mu(\BooD{\unit}{c_{k-1}})\) and \(\mu(c_{k-1}) = \mu(d_2)+\mu(d_3)\). So, (4EP)
implies that there are two disjoint proper elements \(a_{k+1}\) and \(c_k\) for
which \(c_{k-1} = \BooV{a_{k+1}}{c_k}\) and
\(\mu(a_{k+1}) = b_{k+1}\) and \(\mu(c_k) = \sum_{j=k+2}^N b_j\), and we may
continue this procedure to get a final partition \(\{a_1,\ldots,a_N\}\).\par
The above notation (that is, \(\Delta_k(\Sigma)\)) and formulas for
\(\Delta_k(\Sigma)\) will be used in further results of this paper.
\end{exm}

For any \(S\)-valued measure \(\mu\) on \(\bbB\) and a non-zero element \(j \in
\bbB\), we will use the following notation:
\begin{itemize}
\item \(\bbB\restriction{j} \df \{\BooW{x}{j}\dd\ x \in \bbB\}\);
 \(\bbB\restriction{j}\) is a Boolean algebra isomorphic to \(\bbB\) (with
 \(j\) playing role of \(\unit\));
\item \(\mu\restriction{j}\) is the restriction of \(\mu\) to
 \(\bbB\restriction{j}\); \(\mu\restriction{j}\) is an \(S\)-valued measure as
 well.
\end{itemize}

\begin{pro}{PN-homo-inherit}
Let \(\mu\) and \(j\) be a homogeneous \(S\)-valued measure on \(\bbB\) and,
respectively, a proper element of \(\bbB\), and let \(\Sigma \df \Sigma(\mu)\)
and \(d \df \mu(\BooD{\unit}{j})\). Then \(\mu\restriction{j}\) is homogeneous
as well and for \(\Sigma_j \df \Sigma(\mu\restriction{j})\) we have:
\[\begin{cases}
\Sigma_j = \{(a,b,c)\dd\ a+b+c = \mu(j),\ (a,b,c+d) \in \Sigma,\ (a+b,c,d) \in
\Sigma\}\\
(\Sigma_j)' = \{(a,b)\dd\ a+b = \mu(j),\ (a,b,d) \in \Sigma\}\\
e(\Sigma_j) = \mu(j)
\end{cases}.\]
\end{pro}
\begin{proof}
Homogeneity of \(\nu \df \mu\restriction{j}\) simply follows from
ultrahomogeneity of \(\mu\). Indeed, if \(\aaA_0\) is a finite Boolean
subalgebra of \(\bbB_0 \df \bbB\restriction{j}\) and \(\phi_0\dd \aaA_0 \to
\bbB_0\) is a partial \(\nu\)-isomorphism, then it extends to a partial
\(\mu\)-isomorphism \(\phi\dd \aaA \to \bbB\) defined on the Boolean subalgebra
\(\aaA\) of \(\bbB\) generated by \(\aaA_0\) by setting \(\phi(\BooD{\unit}{j})
= \BooD{\unit}{j}\). Since \(\aaA\) is finite, it follows from \THM{PN-4EP} that
\(\phi\) extends to \(\Phi \in \OPN{Aut}(\mu)\). Then \(\Phi\restriction{\bbB_0}
\in \OPN{Aut}(\nu)\) extends \(\phi_0\) and therefore \(\nu\) is
homogeneous.\par
Now if \(\{f,g,h\}\) is a partition of \(\bbB_0\), then \(\{f,g,h,
\BooD{\unit}{j}\}\) is a partition of \(\bbB\) and thus
\[(\mu(f),\mu(g),\mu(h),\mu(\BooD{\unit}{j})) \in \Delta_4(\Sigma)\]
(cf. \EXM{PN-delta-k}). Additionally, \(\mu(f)+\mu(g)+\mu(h) = \mu(j)\), which
shows that \(\Sigma_j\) is contained in the set appearing on the right-hand side
of the postulated formula for \(\Sigma_j\) (thanks to the characterisation given
in the example cited above). Conversely, if \((a,b,c)\) satisfies all conditions
specified in this formula, then \((a,b,c,d) \in \Delta_4(\mu)\) and it follows
from (4EP) that there exists a partition \(\{p,q,r,s\}\) of \(\bbB\) such that
\((\mu(p),\mu(q),\mu(r),\mu(s)) = (a,b,c,d)\) and \(s = \BooD{\unit}{j}\)
(because \(a+b+c = \mu(j)\)). Consequently, \(\{p,q,r\}\) is a partition of
\(\bbB_0\) and therefore \((a,b,c) \in \Sigma_j\). The remaining formulas easily
follow.
\end{proof}

The following example is taken from \cite{DKMN1}.

\begin{exm}{PN-homo-bad}
Although (ultra)homogeneity introduced in \DEF{PN-homo-meas} has a quite natural
definition which may lead to a supposition that the automorphism group of
a homogeneous measure is rich, actually this property does not even imply that
this group is nontrivial. To convince oneself of that, observe that there exists
a full non-atomic probability Borel measure \(\mu\) on the \Cantor\ such that on
\(\bbB\) it is both rational-valued and one-to-one. Then each partial
\(\mu\)-isomorphism coincides with the identity map and hence \(\mu\) is
homogeneous. However, \(\OPN{Aut}(\mu\restriction{j})\) is trivial for any
non-zero \(j \in \bbB\). Such a pathology motivates us to distinguish certain
subclass of homogeneous measures that we introduce in the next section. As we
will see, to get quite a rich automorphism group, it is enough to make a small
modification in axiom (Sp4b).
\end{exm}

\begin{exm}{PN-homo-bad2}
Denote by \(\mu\) the measure from the previous example. So, \(\mu\) is
a rational-valued measure on \(\bbB\) that is a one-to-one function, and
\(\mu(\unit) = 1\). It is not difficult to construct, in a similar manner,
a measure \(\nu\dd \bbB \to [0,1]\) such that \(\nu(\unit) = 1\) and \(\nu(j)-
\nu(k) \notin \QQQ\) for any two distinct non-zero elements \(j, k \in \bbB\).
Now take any partition \(\{a,b\}\) of \(\bbB\) and define \(\rho\dd \bbB \to
[0,1]\) as follows: \(\rho(x) = \frac12\mu(\BooW{x}{a})+ \frac12\nu(\BooW{x}
{b})\). It is not difficult to check that for two \textbf{different} elements
\(x\) and \(y\) of \(\bbB\) the following equivalence holds:
\begin{equation}\label{eqn:aux170}
\rho(x) = \rho(y) \iff \{x,y\} = \{a,b\}.
\end{equation}
Now let \(j\) be a non-zero element in \(\bbB\), \(\aaA\) be an arbitrary
Boolean subalgebra of \(\bbB_j \df \bbB\restriction{j}\) and \(\phi\dd \aaA \to
\bbB_j\) a partial \(\rho'\)-isomorphism where \(\rho' \df
\rho\restriction{j}\). Observe that:
\begin{itemize}
\item If \(j \neq \unit\), then \(\rho'\) is one-to-one and hence \(\phi(x) =
 x\) for each \(x \in \aaA\) (thanks to \eqref{eqn:aux170}). In particular,
 \(\rho'\) is homogeneous unless \(j \neq \unit\).
\item If \(j = \unit\), then---similarly, by \eqref{eqn:aux170}---\(\phi(x) =
 x\) for all \(x \in \aaA \setminus \{a,b\}\). It follows that if \(\aaA \neq
 \aaA_o \df \{\zero,\unit,a,b\}\), then \(\phi\) is the identity on \(\aaA\).
\item Finally, a function \(\psi\dd \aaA_o \to \aaA_o\) such that \(\psi(a) =
 b\), \(\psi(b) = a\) and \(\psi(x) = x\) for \(x \in \{\zero,\unit\}\) is
 a partial \(\rho\)-isomorphism.
\end{itemize}
The above remarks yield that \(\rho\restriction{j}\) is homogeneous for any
proper element \(j\) of \(\bbB\) and the above \(\psi\) is a unique partial
\(\rho\)-isomorphism that does not coincide with the identity map. In
particular, \(\OPN{Aut}(\rho\restriction{j})\) is the trivial group for any
non-zero element \(j\) of \(\bbB\), and \(\rho\) is not homogeneous.\par
The above example shows that the property that \(\mu\restriction{j}\) is
homogeneous for each proper \(j \in \bbB\) is, in general, insufficient for
a measure \(\mu\) to be homogeneous as well.
\end{exm}

\section{Solid measures}

Topological transitivity (sometimes called \emph{almost transitivity}) in
dynamical systems belongs to classical notions responsible for, roughly
speaking, ubiquity. A most common definition says that a group action (on
a topological space) is (topologically) transitive if the orbit of at least one
point is dense in the space. If the action is continuous and the underlying
topological space is Polish, this condition may equivalently be formulated in
terms of open sets, which for the \Cantor\ may easily be translated to
the context of its Boolean algebra as we do below.

\begin{dfn}{PN-transitive}
We say a subgroup \(G\) of the Boolean algebra automorphism group of \(\bbB\)
acts \emph{transitively} on \(\bbB\) (or that the action of \(G\) is
\emph{transitive} on \(\bbB\)) if for any two proper elements \(a\) and \(b\) of
\(\bbB\) there exists \(\phi \in G\) such that \(\BooW{\phi(a)}{b} \neq
\zero\).\par
An \(S\)-valued measure \(\mu\) on \(\bbB\) is said to be \emph{solid} if
\(\mu\) is homogeneous and the group \(\OPN{Aut}(\mu)\) acts transitively on
\(\bbB\).\par
A \homspec{3} \(\Sigma\) in \(S\) is said to be \emph{solid} if a homogenous
measure whose trinary spectrum coincides with \(\Sigma\) is solid (we have
already known that there exists such a measure and it is unique up to Boolean
algebra automorphism---so, the property of being a solid \homspec{3} is
independent of the choice of such a measure).
\end{dfn}

For further applications, we also distinguish a subclass of solid \homspec{3}s
that is related to minimal group actions. Recall that the (natural) action of
a homeomorphism group \(G\) of the \Cantor\ \(\Can\) is minimal if each orbit
is dense; that is, if for any \(a \in \Can\) the set \(\{g(a)\dd\ g \in G\}\) is
dense in \(\Can\). Also this property can be characterised in terms of clopen
sets and thus can be translated to the context of \(\bbB\), as we do below.

\begin{dfn}{PN-mini}
We say the (natural) action of a subgroup \(G\) of the Boolean algebra
automorphism group of \(\bbB\) is \emph{minimal} if for any proper element
\(a \in \bbB\) there are a finite number of automorphisms \(\phi_1,\ldots,
\phi_n \in G\) such that \(\bigvee_{k=1}^n \phi_k(a) = \unit\).\par
We call an \(S\)-valued measure \(\mu\) on \(\bbB\) \emph{supersolid} if \(\mu\)
is homogeneous and the action of \(\OPN{Aut}(\mu)\) is minimal (on
\(\bbB\)).\par
A \homspec{3} \(\Sigma\) in \(S\) is said to be \emph{supersolid} if
a homogenous measure whose trinary spectrum coincides with \(\Sigma\) is
supersolid.
\end{dfn}

As for homogeneity (for measures), the property of being solid or supersolid is
hereditary, as shown by

\begin{pro}{PN-solid-hered}
For a semigroup-valued measure \(\mu\dd \bbB \to S\) \tfcae
\begin{enumerate}[\upshape(i)]
\item the action of \(H \df \OPN{Aut}(\mu)\) is transitive \UP(resp.
 minimal\UP) on \(\bbB\);
\item for any proper \(j \in \bbB\), the action of \(H_j \df
 \OPN{Aut}(\mu\restriction{j})\) is transitive \UP(resp. minimal\UP) on \(\bbB_j
 \df \bbB\restriction{j}\).
\end{enumerate}
\end{pro}
\begin{proof}
We start from the following useful observation:
\begin{itemize}
\item[\((\star)\)] if \(f \in \bbB\) and \(\phi \in H\) are such that \(f\) and
 \(\phi(f)\) are disjoint, then there is \(\psi \in H\) satisfying \(\psi(f) =
 \phi(f)\) and \(\psi(x) = x\) for any \(x \in \bbB\) disjoint from both \(f\)
 and \(\phi(f)\).
\end{itemize}
(Indeed, in the above situation, define \(\psi\dd \bbB \to \bbB\) by:
\[\psi(x) = \BooV{\phi(\BooW{x}{f})}{\BooV{\phi^{-1}(\BooW{x}{\phi(f)})}{(
\BooD{x}{(\BooV{f}{\phi(f)})})}}.)\]
First assume the action of \(H\) is transitive, and fix proper \(j \in \bbB\)
as well as two non-zero elements \(a\) and \(b\) of \(\bbB_j\). We may and do
assume that \(a\) and \(b\) are disjoint. It follows from the assumption that
there is \(\phi \in H\) such that \(\BooW{\phi(a)}{b} \neq \zero\).
Equivalently, there is non-zero \(c\) such that \(\BooW{c}{a} = c\) and
\(\BooW{\phi(c)}{b} = \phi(c)\). In particular, \(c\) and \(\phi(c)\) are
disjoint. So, we infer from \((\star)\) that there exists \(\psi \in H\) for
which \(\psi(c) = \phi(c)\) and \(\psi(\BooD{\unit}{j}) = \BooD{\unit}{j}\).
Consequently, \(\psi' \df \psi\restriction{\bbB_j}\) belongs to \(H_j\) and
\(\BooW{\psi'(a)}{b} \neq \zero\).\par
Now assume the group \(H_j\) acts transitively on \(\bbB_j\) for any proper \(j
\in \bbB\), and fix arbitrary proper elements \(a\) and \(b\) of \(\bbB\). To
show that there exists \(\phi \in \OPN{Aut}(\mu)\) with \(\BooW{\phi(a)}{b} \neq
\zero\), we may reduce \(a\) and \(b\) and thus we may (and do) assume that \(c
\df \BooV{a}{b}\) differs from \(\unit\). Then there is non-zero \(d \in \bbB\)
that is disjoint from \(c\) and satisfies \(j \df \BooV{c} {d} \neq \unit\).
Note that \(j\) is proper in \(\bbB\) and both \(a\) and \(b\) are proper
elements of \(\bbB_j\). So, it follows from our assumption that there is \(\psi
\in H_j\) such that \(\BooW{\psi(a)}{b} \neq \zero\). A note that \(\psi\)
extends to \(\phi \in H\) (by declaring that \(\phi(x) = x\) for any \(x \in
\bbB\) that is disjoint from \(j\)) concludes the proof of that part.\par
Now assume the action of \(H\) is minimal and fix proper \(j \in \bbB\) as well
as non-zero \(a \in \bbB_j\). We may and do assume that \(a' \df \BooD{j}{a}\)
is non-zero as well. It follows from the assumption that there exists
a partition \(\{c_1,\ldots,c_N\}\) of \(\bbB_{a'}\), non-zero elements \(b_1,
\ldots,b_N \in \bbB_a\) and automorphisms \(\phi_1,\ldots,\phi_N \in H\) such
that \(\phi_k(b_k) = c_k\ (k=1,\ldots,N)\). Using \((\star)\) separately for
each \(k=1,\ldots,N\), we get automorphisms \(\psi_1,\ldots,\psi_N \in H\) such
that \(\psi_k(b_k) = c_k\) and \(\psi_k(\BooD{\unit}{j}) = \BooD{\unit}{j}\). In
particular, the restriction \(\psi_k'\) of \(\psi_k\) to \(\bbB_j\) belongs to
\(H_j\), and \(j = \bigvee_{k=0}^N \psi_k'(a)\) where \(\psi_0'\) is
the identity on \(\bbB_j\).\par
Finally, assume the action of \(H_j\) is minimal on \(\bbB_j\) for any proper
\(j \in \bbB\), and fix a proper element \(a\) of \(\bbB\). Take a partition
\(\{p,q,r,s\}\) of \(\bbB\) such that \(a = \BooV{p}{q}\) and set \(j \df
\BooV{p}{r}\) and \(j' \df \BooD{\unit}{j}\). Observe that both \(j\) and \(j'\)
are proper in \(\bbB\), \(p\) is proper in \(\bbB_j\) and \(q\) is proper in
\(\bbB_{j'}\). So, we conclude that there are \(\phi_1,\ldots,\phi_N \in H_j\)
and \(\psi_1,\ldots,\psi_M \in H_{j'}\) such that \(j = \bigvee_{k=1}^N
\phi_k(p)\) as well as \(j' = \bigvee_{\ell=1}^M \psi_{\ell}(q)\). Now for any
\(k = 1,\ldots,N\) and \(\ell = 1,\ldots,M\) let \(\xi_{k\ell} \in H\) denote
a unique automorphism that extends both \(\phi_k\) and \(\psi_{\ell}\). Then
\(\bigvee_{k,\ell} \xi_{k,\ell}(a) = \unit\), as \(a = \BooV{p}{q}\).
\end{proof}

The above result, combined with \PRO{PN-homo-inherit}, yields

\begin{cor}{PN-solid-hered}
If \(\mu\dd \bbB \to S\) is solid \UP(resp. supersolid\UP), then for any
non-zero \(j \in \bbB\) the measure \(\mu\restriction{j}\) is solid \UP(resp.
supersolid\UP) as well.
\end{cor}

Our nearest aim is to give characterisations of solid and supersolid
\homspec{3}s.

\begin{thm}{PN-solid}
A \homspec{3} \(\Sigma\) in \(S\) is solid iff condition \UP{(Sp4b)} can always
be fulfilled with some \((x,y,z)\) and \((u,v,w)\) from \(\Sigma\). In other
words, \(\Sigma\) is solid iff
\begin{enumerate}[\upshape(Sp1)]\addtocounter{enumi}{3}
\item for any \(a,b,p,q,r \in S\) such that \(a+b = p+q\), and both \((a,b,r)\)
 and \((p,q,r)\) belong to \(\Sigma\) or both belong to \(\Sigma' \times
 \{0_S\}\) there are \(x,y,u,v \in S\) for which
 \begin{equation}\label{eqn:PN-Sp4}
 \begin{cases}
 (x,y,q+r) \in \Sigma, & (u,v,p+r) \in \Sigma,\\
 p = x+y, & q = u+v,\\
 a = x+u, & b = y+v.
 \end{cases}
 \end{equation}
\end{enumerate}
\end{thm}

We draw the reader's attention that, in comparison to (Sp4b), in (Sp4) we have
changed the role of the letters \(x,y,u\) and \(v\) (and, in addition, \(z\) and
\(w\) are not present in (Sp4)).\par
We precede the proof of the above theorem by a key lemma.

\begin{lem}{PN-aux10}
If \(\mu\) is a solid \(S\)-valued measure on \(\bbB\), then for any two proper
elements \(a\) and \(b\) of \(\bbB\) there exists \(\phi \in \OPN{Aut}(\mu)\)
such that each of the elements \(\BooW{\phi(a)}{b}\), \(\BooD{\phi(a)}{b}\),
\(\BooD{b}{\phi(a)}\) and \(\BooV{\phi(a)}{b}\) is proper.
\end{lem}
\begin{proof}
First we claim that there is \(\xi \in \OPN{Aut}(\mu)\) such that both
\(\BooW{\xi(a)}{b}\) and \(\BooD{\xi(a)}{b}\) are non-zero. To prove this
property, we take two automorphisms \(\phi, \psi \in \OPN{Aut}(\mu)\) such that
\(\BooW{\phi(a)}{b}\) and \(\BooD{\psi(a)}{b}\) are non-zero (to find \(\psi\),
substitute \(\BooD{\unit}{b}\) in place of \(b\)). If one of
\(\BooD{\phi(a)}{b}\) or \(\BooW{\psi(a)}{b}\) is non-zero, the claim is proved.
Thus, assume \(a' \df \phi(a)\) is disjoint from \(\BooD{\unit}{b}\) and \(a''
\df \psi(a)\) is disjoint from \(b\). Express \(a\) in the form \(a =
\BooV{p}{q}\) where \(p\) and \(q\) are two disjoint proper elements, and
additionally set \((p',q',p'',q'') \df (\phi(p),\phi(q),\psi(p),\psi(q))\),
\(z \df \BooD{\unit}{(\BooV{a'}{a''})}\) and \(a_o \df \BooV{p'}{q''}\). Note
that \(\BooW{a_o}{b} = p'\) and \(\BooD{a_o}{b} = q''\) are non-zero, and
\(\mu(a_o) = \mu(a)\) and
\begin{multline*}
\mu(\BooD{\unit}{a_o}) = \mu(q')+\mu(p'')+\mu(z) =
\mu(q)+\mu(p)+\mu(z) = \mu(a)+\mu(z) = \mu(a')+\mu(z) \\= \mu(\BooV{a'}{z}) =
\mu(\BooD{\unit}{a''}) = \mu(\psi(\BooD{\unit}{a})) = \mu(\BooD{\unit}{a}).
\end{multline*}
So, we infer from homogeneity of \(\mu\) that there is \(\xi \in
\OPN{Aut}(\mu)\) such that \(\xi(a) = a_o\), which shows the first claim of this
proof.\par
Going further, we may replace \(a\) by \(\xi(a)\) and therefore we do assume
that \(\xi(a) = a\). For simplicity, set \(b' \df \BooD{\unit}{b}\). We have
already known that \(\BooW{b}{a}\) and \(\BooW{b'}{a}\) are non-zero. If also
\(\BooD{b}{a}\) and \(\BooD{b'}{a}\) are non-zero, we have nothing to do. Thus,
we assume one of these elements coincides with \(\zero\). By symmetry, we may
and do assume that \(\BooD{b}{a} = \zero\). Set \(c \df \BooD{a}{b}\) (then
\(c\) is non-zero and \(a = \BooV{b}{c}\)). Now applying the first claim of
the proof with switched \(a\) and \(b\), we get an automorphism \(\eta \in
\OPN{Aut}(\mu)\) such that \(\BooW{\eta^{-1}(b)}{a} \neq \zero \neq
\BooD{\eta^{-1}(b)}{a}\). Equivalently, \(\BooW{b}{\eta(a)} \neq \zero \neq
\BooD{b}{\eta(a)}\). If also both \(\BooW{b'}{\eta(a)}\) and
\(\BooD{b'}{\eta(a)}\) are non-zero, then the proof is finished. So, again, we
assume one of these elements coincides with \(\zero\). We distinguish two
separate cases.\par
First assume \(\BooW{b'}{\eta(a)} = \zero\). Then we set \(d \df
\BooD{b}{\eta(a)}\) (so, \(d\) is non-zero and \(b = \BooV{\eta(a)}{d}\)).
Observe that \(a = \BooV{\eta(a)}{\BooV{c}{d}} =
\BooV{\eta(b)}{\BooV{c}{\BooV{d}{\eta(c)}}}\) and \(\{\eta(b),c,d,\eta(c),
\BooD{\unit}{a}\}\) is a partition of \(\bbB\). So, for \(a' \df
\BooV{\eta(b)}{c}\) we have \(\mu(a') = \mu(\eta(b))+\mu(c) = \mu(b)+\mu(c) =
\mu(a)\) and
\begin{multline*}
\mu(\BooD{\unit}{a'}) = \mu(\BooV{d}{\BooV{\eta(c)}{(\BooD{\unit}{a})}}) =
\mu(d)+\mu(c)+\mu(\BooD{\unit}{a}) = \mu(\BooV{d}{\BooV{c}{(\BooD{\unit}{a})}})
\\= \mu(\BooD{\unit}{(\BooV{\eta(b)}{\eta(c)})}) = \mu(\eta(\BooD{\unit}{a})) =
\mu(\BooD{\unit}{a}).
\end{multline*}
We infer that there exists \(\tau \in \OPN{Aut}(\mu)\) such that \(\tau(a) =
a'\). But \(\BooW{b}{a'} = \eta(b)\), \(\BooW{b'}{a'} = c\), \(\BooD{b}{a'} =
\BooV{d}{\eta(c)}\) and \(\BooD{b'}{a'} = \BooD{\unit}{a}\), which finishes
the proof in that case.\par
Finally, assume \(\BooD{b'}{\eta(a)} = \zero\) and set \(d \df
\BooW{\eta(a)}{b}\) (then \(d\) is non-zero and \(\eta(a) = \BooV{b'}{d}\)). We
claim that \(\{c,d,\BooD{\unit}{a},\eta(\BooD{\unit}{a})\}\) is a partition of
\(\bbB\) such that
\begin{equation}\label{eqn:PN-aux40}
b = \BooV{\eta(\BooD{\unit}{a})}{d} \qquad \UP{and} \qquad
b' = \BooV{(\BooD{\unit}{a})}{c}.
\end{equation}
Indeed:
\begin{itemize}
\item \(\eta(\BooD{\unit}{a}) = \BooW{b}{(\BooD{\unit}{d})}\) and hence
 \(\eta(\BooD{\unit}{a})\) is disjoint from \(c\), \(d\) and \(\BooD{\unit}{a}\)
 (as \(a = \BooV{b}{c}\) and \(\BooW{b}{c} = \zero\));
\item \(\BooD{\unit}{a}\) is disjoint from both \(c\) and \(d\) (as
 \(\BooW{c}{a} = c\) and \(\BooW{d}{a} = d\)) and \(c\) and \(d\) are disjoint,
 since \(\BooW{b}{c} = \zero\);
\item \(\BooV{c}{\BooV{d}{\BooV{(\BooD{\unit}{a})}{\eta(\BooD{\unit}{a})}}} =
 \BooW{(\BooV{c}{\BooV{d}{\BooV{(\BooD{\unit}{a})}{b}}})}{(\BooV{c}{\BooV{d}%
 {\BooV{(\BooD{\unit}{a})}{(\BooD{\unit}{d})}}})} = \unit\) (as \(a =
 \BooV{b}{c}\));
\item the latter formula in \eqref{eqn:PN-aux40} is a direct consequence of
 the following two relations: \(a = \BooV{b}{c}\) and \(\BooD{b}{c} = \zero\);
 the former of them now follows, since \(\{c,d,\BooD{\unit}{a},
 \eta(\BooD{\unit}{a})\}\) is a partition.
\end{itemize}
Now we express \(\BooD{\unit}{a}\) as \(\BooV{p}{q}\) where \(p\) and \(q\) are
disjoint and proper. Set \((p',q') \df (\eta(p),\eta(q))\) and \(a' \df
\BooV{c}{\BooV{p}{\BooV{q'}{d}}}\). Since \(\{c,d,p,q,p',q'\}\) is a partition,
we conclude that \(\mu(a') = \mu(a)\) (as \(\mu(p) = \mu(p')\) and \(a =
\BooV{p'}{\BooV{q'}{\BooV{c}{d}}}\), by \eqref{eqn:PN-aux40}). and, similarly,
\(\mu(\BooD{\unit}{a'}) = \mu(\BooD{\unit}{a})\). So, there is \(\tau \in
\OPN{Aut}(\mu)\) such that \(\tau(a) = a'\). But \eqref{eqn:PN-aux40} implies
that \(\BooW{b}{a'} = \BooV{q'}{d}\), \(\BooW{b'}{a'} = \BooV{p}{c}\),
\(\BooD{b}{a'} = p'\) and \(\BooD{b'}{a'} = q\) are all non-zero, and we are
done.
\end{proof}

\begin{proof}[Proof of \THM{PN-solid}]
First assume (Sp4) holds. Fix two proper elements \(s\) and \(t\) of \(\bbB\).
This means that both \(s' \df \BooD{\unit}{s}\) and \(t' \df \BooD{\unit}{t}\)
are non-zero as well. Note that both \((p,q) \df (\mu(s),\mu(s'))\) and \((a,b)
\df (\mu(t),\mu(t'))\) belong to \(\Sigma'\). So, we infer from (Sp4) that there
are \(x,y,u,v \in S\) such that \eqref{eqn:PN-Sp4} holds with \(r = 0_S\). Note
that then \((x,y,q),(u,v,p) \in \Sigma\). Now since \(\mu\) is homogeneous, it
follows from (4EP) that there are two pairs \((d,f),(g,h)\) of disjoint proper
elements such that \(s = \BooV{d}{f}\), \(s' = \BooV{g}{h}\) and
\((\mu(d),\mu(f)) = (x,y)\) and \((\mu(g),\mu(h)) = (u,v)\). Then \(\mu(t) =
a = x+u = \mu(d)+\mu(g) = \mu(\BooV{d}{g})\) and similarly \(\mu(t') = b = y+v =
\mu(\BooV{f}{h}) = \mu(\BooD{j}{(\BooV{d}{g})})\). Since \(\mu\) is homogeneous,
there exists \(\phi \in \OPN{Aut}(\mu)\) for which \(\phi(t) = \BooV{d}{g}\).
But then \(\BooW{\phi(t)}{s} \neq \zero\) and hence \(\OPN{Aut}(\mu)\) acts
transitively on \(\bbB\).\par
Now assume \(\mu\) is solid and take arbitrary \(p,q,r,a,b \in S\) that satisfy
all the conditions specified in (Sp4). For a while we distinguish two separate
cases:
\begin{itemize}
\item If both \((p,q,r)\) and \((a,b,r)\) belong to \(\Sigma\), we take
 a partition \(\{s,t,w\}\) of \(\bbB\) such that \((\mu(s),\mu(t),\mu(w)) = (p,
 q,r)\). Then we use (4EP) (for the element \(\BooV{s}{t}\) and a partition of
 \(\bbB\) corresponding to \((a,b,r)\)) to find two disjoint proper elements
 \(f\) and \(g\) such that \(\BooV{f}{g} = \BooV{s}{t}\) and \((\mu(f),\mu(g)) =
 (a,b)\).
\item When \((p,q,r),(a,b,r) \in \Sigma' \times \{0_S\}\), there is a partition
 \(\{s,t\}\) of \(\bbB\) for which \((\mu(s),\mu(t)) = (p,q)\). We additionally
 set \(w \df \zero\) (note that \(\mu(w) = r\) and \(\BooV{\BooV{p}{q}}{w} =
 \unit\)). Since \((a,b) \in \Sigma'\), there is a partition \(\{f,g\}\) of
 \(\bbB\) such that \((\mu(f),\mu(g)) = (a,b)\). Note that \(\BooV{f}{g} =
 \BooV{s}{t}\).
\end{itemize}
Now we set \(j \df \BooV{s}{t}\) and use \LEM{PN-aux10} for the measure \(\nu
\df \mu\restriction{j}\) (\(\nu\) is solid by \COR{PN-solid-hered}). We conclude
that there is \(\phi \in \OPN{Aut}(\nu)\) such that \(\alpha \df
\BooW{\phi(s)}{f}\), \(\beta \df \BooW{\phi(s)}{g}\),\(\gamma \df
\BooW{\phi(t)}{f}\) and \(\delta \df \BooW{\phi(t)}{g}\) are all proper. Now we
set \((x,y,u,v) \df (\mu(\alpha),\mu(\beta),\mu(\gamma),\mu(\delta))\). We only
need to check that \((x,y,q+r)\) and \((u,v,p+r)\) belong to \(\Sigma\) (all
the remaining formulas from \eqref{eqn:PN-Sp4} are immediate consequences of
the invariance of \(\mu\) under \(\phi\)), which is immediate as these two
triples correspond to the partitions \(\{\alpha,\beta,\BooV{\phi(t)}{w}\}\) and,
respectively, \(\{\gamma,\delta,\BooV{\phi(s)}{w}\}\).
\end{proof}

\begin{rem}{PN-Sp4}
Observe that (Sp4) covers both (Sp4a) and (Sp4b). Indeed, it is readily seen
that (Sp4) is stronger than (Sp4b). And to conclude (Sp4a) from (Sp4), take
arbitrary \((a,b,c) \in \Sigma\), set \(p = a, q = b\) and \(r = c\) and apply
(Sp4) to get \(x, y \in S\) such that \(a = p = x+y\) and \((x,y,b+c) \in
\Sigma\). So, a set \(\Sigma \subset S^3\) is a solid \homspec{3} iff it
satisfies (Sp0)--(Sp3) and (Sp4).
\end{rem}

\begin{exm}{PN-good-super}
Good measures on the \Cantor\ are supersolid. More specifically, if \(\mu\) is
such a measure (represented here as a function on \(\bbB\)) and \(I = \mu(\bbB
\setminus \{\zero,\unit\})\), then \(\Sigma(\mu) = \{(a,b,c) \in I^3\dd\ a+b+c =
1\}\). It is not difficult to check that the set appearing on the right-hand
side of this equation is a supersolid \homspec{3} (cf. \THM{PN-supersolid}
below). In \THM{PN-super-no-good} (see also \REM{PN-uncount}) we will give
an example of a probability Borel measure on the \Cantor\ that is supersolid but
not good. (As shown in \PRO{PN-super-full-nonatom} below, all such measures are
automatically full and non-atomic).
\end{exm}

\begin{pro}{PN-super-full-nonatom}
Let \(\mu\) be a positive Borel measure on the \Cantor. If the action of
\(\OPN{Aut}(\mu)\) is minimal, then \(\mu\) is both full and non-atomic.\par
In particular, supersolid positive Borel measures are full and non-atomic.
\end{pro}
\begin{proof}
Let \(Q\) stand for the (closed) support of \(\mu\). It is well-known that \(Q\)
is fully invariant for any homeomorphism \(h\) preserving \(\mu\) (that is,
\(h(Q) = Q\)). In particular, the set \(\{g(a)\dd\ g \in \OPN{Aut}(\mu)\}\) is
contained in \(Q\) for any \(a \in Q\). So, we conclude from the assumption of
the proposition that \(Q = \Can\). Similarly, if \(a \in \Can\) was an atom for
\(\mu\) (that is, if \(\mu(\{a\}) > 0\)), then the set \(F\) of all \(x \in
\Can\) satisfying \(\mu(\{x\}) = \mu(\{a\})\) would be finite and fully
invariant for any \(h \in \OPN{Aut}(\mu)\), which would contradict minimality
of the action.
\end{proof}

\begin{rem}{PN-solid-without-full-nonatom}
The previous proof works perfectly also for supersolid signed Borel measures on
the \Cantor; that is, if \(\mu\) is such a measure, then the variation \(|\mu|\)
of \(\mu\) is both full and non-atomic, as atoms of \(|\mu|\) coincide with
atoms of \(\mu\) and the (closed) support of \(|\mu|\) is the least closed set
\(F\) such that \(\mu\) vanishes on each clopen set that is disjoint from
\(F\).\par
As an arbitrary Dirac's measure (that is, a probability measure that is
supported on a single point) shows, analogous properties does not hold, in
general, for solid probability Borel measures. However, it may easily be shown
that the (closed) support \(K\) of each such a (last mentioned) measure \(\mu\)
is either nowhere dense or equal to \(\Can\). (Indeed, a point with dense orbit
belongs either to \(\Can \setminus K\) or to \(K\) and both these sets are
invariant for \(\OPN{Aut}(\mu)\).) It is also possible to
construct such a measure that is full and has atoms. We leave the details to
interested readers.
\end{rem}

Since minimal group actions are transitive, it follows that supersolid measures
(as well as supersolid \homspec{3}s) are solid. Our nearest aim is to give
a characterisation of supersolid \homspec{3}s. We precede it by

\begin{lem}{PN-mini}
Let \(\mu\) be a homogeneous \(S\)-valued measure and \(\Sigma = \Sigma(\mu)\).
For two proper elements \(a\) and \(b\) of \(\bbB\) \tfcae
\begin{enumerate}[\upshape(i)]
\item there exists \(\phi \in \OPN{Aut}(\mu)\) such that \(\BooW{\phi(a)}{b} =
 b \neq \phi(a)\);
\item there exists \(\psi \in \OPN{Aut}(\mu)\) for which \(\BooV{\psi(b)}{a} =
 a \neq \psi(b)\);
\item there exists \(x \in S\) such that \((\mu(b),x,\mu(\BooD{\unit}{a})) \in
 \Sigma\) and \(x+\mu(b) = \mu(a)\) and \(x+\mu(\BooD{\unit}{a}) =
 \mu(\BooD{\unit}{b})\).
\end{enumerate}
\end{lem}
\begin{proof}
Conditions (i) and (ii) are simply equivalent (just set \(\psi = \phi^{-1}\) and
conversely: \(\phi = \psi^{-1}\)). Further, if \(\psi\) is as specified in (ii),
set \(c \df \BooD{a}{\psi(b)}\) and note that \(c\) is proper and \(\{\psi(b),c,
\BooD{\unit}{a}\}\) is a partition of \(\bbB\) such that \(a =
\BooV{\psi(b)}{c}\). So, it suffices to set \(x \df \mu(c)\) to get (iii).\par
Conversely, if \(x\) satisfies (iii), then there is a partition \(\{u,v,w\}\)
such that
\[((\mu(u),\mu(v),\mu(w)) = (\mu(b),x,\mu(\BooD{\unit}{a})).\]
Then it follows from (4EP) that there exist two disjoint proper elements \(c\)
and \(d\) such that \(a = \BooV{c}{d}\) and \((\mu(c),\mu(d)) = (\mu(u),\mu(v))\
(=(\mu(b),x))\). Finally, since \(\mu(c) = \mu(b)\) and \(\mu(\BooD{\unit}{b}) =
x+\mu(\BooD{\unit}{a}) = \mu(\BooD{\unit}{c})\), we conclude that there exists
\(\psi \in \OPN{Aut}(\mu)\) with \(\psi(b) = c\), and we are done.
\end{proof}

\begin{thm}{PN-supersolid}
A \homspec{3} \(\Sigma\) in \(S\) is supersolid iff the following condition is
fulfilled:
\begin{enumerate}[\upshape(Sp1)]\addtocounter{enumi}{4}
\item For any \((x,y) \in \Sigma'\) there exist \(b_1,\ldots,b_N \in S\) \UP(for
 some \(N \geq 3\) or, equivalently, for all sufficiently large \(N\)\UP) and
 \(w_1,\ldots,w_N \in S\) such that all the following conditions are fulfilled:
 \begin{enumerate}[\upshape(Sp5a)]
 \item \((\sum_{s=1}^{k-1} b_s,b_k,\sum_{s=k+1}^N b_s) \in \Sigma\) whenever
  \(1 < k < N\);
 \item \(b_k+w_k = x\), \(w_k+y = \sum_{s \neq k} b_s\) and \((b_k,w_k,y) \in
  \Sigma\) for all \(k=1,\ldots,N\).
 \end{enumerate}
\end{enumerate}
\end{thm}
\begin{proof}
Fix a homogeneous measure \(\mu\) such that \(\Sigma(\mu) = \Sigma\). First
assume (Sp5) is fulfilled and fix a proper element \(a\) of \(\bbB\). Set
\((x,y) \df (\mu(a),\mu(\BooD{\unit}{a}))\). Then \((x,y) \in \Sigma'\). So, let
\(b_1,\ldots,b_N\) and \(w_1,\ldots,w_N\) be as specified in (Sp5). We infer
from (Sp5a) that \((b_1,\ldots,b_N) \in \Delta_N(\Sigma)\) (cf.
\EXM{PN-delta-k}). So, there exists a partition \(\{c_1,\ldots,c_N\}\) of
\(\bbB\) such that \(\mu(c_k) = b_k\) for \(k=1,\ldots,N\). Further, (Sp5b) says
that
\begin{equation}\label{eqn:PN-aux41}
(\mu(c_k),w_k,\mu(\BooD{\unit}{a})) \in \Sigma
\end{equation}
for \(k=1,\ldots,N\). The above relation, combined with the first two equations
in (Sp5b), shows that condition (iii) of \LEM{PN-mini} is fulfilled for \(b =
c_k\) and \(x = w_k\) (note that \(w_k+\mu(\BooD{\unit}{a}) = w_k+y =
\sum_{s \neq k} \mu(c_s) = \mu(\BooD{\unit}{c_k})\)). Thus, we conclude that
there exists \(\phi_k \in \OPN{Aut}(\mu)\) such that \(\BooW{\phi_k(a)}{c_k} =
c_k\). Since \(\{c_1,\ldots,c_N\}\) is a partition of \(\bbB\), we get
\(\bigvee_{k=1}^M \phi_k(a) = \unit\), which proves that \(\mu\) is
supersolid.\par
Now assume \(\mu\) is supersolid and fix \((x,y) \in \Sigma'\). There exists
a partition \(\{a,b\}\) of \(\bbB\) such that \((\mu(a),\mu(b)) = (x,y)\). It
follows from our assumption that there is a finite number of automorphisms
\(\phi_1,\ldots,\phi_m \in \OPN{Aut}(\mu)\) such that \(\bigvee_{k=1}^m
\phi_k(a) = \unit\). It is now clear that for all sufficiently large numbers \(N
(\geq 3)\) there is a partition \(\{p_1,\ldots,p_N\}\) of \(\bbB\) such that for
any \(k \leq N\) there exists \(s \in \{1,\ldots,m\}\) satisfying
\(\BooW{p_k}{\phi_s(a)} = p_k \neq \phi_s(a)\). So, it follows from
\LEM{PN-mini} that for each such \(k\) there is \(w_k \in S\) such that
\begin{equation}\label{eqn:PN-aux43}
\begin{cases}
(\mu(p_k),w_k,\mu(\BooD{\unit}{a})) \in \Sigma,\\
w_k+\mu(p_k) = \mu(a), \\ w_k+\mu(\BooD{\unit}{a}) = \mu(\BooD{\unit}{p_k}).
\end{cases}
\end{equation}
Finally, we set \(b_k = \mu(p_k)\) for \(k=1,\ldots,N\). Observe that (Sp5a)
follows from the characterisation of \(\Delta_N(\Sigma)\) described in
\EXM{PN-delta-k}, whereas (Sp5b) is a reformulation of \eqref{eqn:PN-aux43}.
\end{proof}

Now it is a good moment to state the following result whose simple proof is left
to the reader.

\begin{pro}{PN-homspec-chain}
The union of an increasing sequence of \homspec{3}s \UP(resp. solid or
supersolid \homspec{3}s\UP) in a semigroup \(S\) is a \homspec{3} \UP(resp.
a solid or supersolid \homspec{3}\UP) in \(S\) as well.
\end{pro}

To formulate our next result, we introduce an auxiliary notion. Let
\(\{S_t\}_{t \in T}\) be a non-empty collection of semigroups and let \(Q \df
\prod_{t \in T} S_t\) stand for the full Cartesian product of these semigroups.
A subset \(\Sigma\) of \(Q^3\) is said to have the \emph{replacement property}
if for any two its elements \(((x_t)_{t \in T}, (y_t)_{t \in T},
(z_t)_{t \in T})\) and \(((p_t)_{t \in T},(q_t)_{t \in T},(r_t)_{t \in T})\) and
any index \(s \in T\) also \(((u_t)_{t \in T},(v_t)_{t \in T},(w_t)_{t \in T})\)
belongs to \(\Sigma\) where \((u_t,v_t,w_t) = \begin{cases}(x_t,y_t,z_t) &
\UP{if } t \neq s\\(p_s,q_s,r_s) & \UP{if } t = s\end{cases}\). Additionally, in
the above situation, for any \(s \in T\), \(\pi_s\dd Q \to S_s\) and
\(\pi_s^{[3]}\dd Q^3 \to S_s^3\) are homomorphisms given by:
\[\pi_s((x_t)_{t \in T}) = x_s, \qquad
\pi_s^{[3]}((x_t)_{t \in T},(y_t)_{t \in T},(z_t)_{t \in T}) = (x_s,y_s,z_s).\]

\begin{pro}{PN-coord}
Let \(\{S_t\}_{t \in T}\) be a non-empty collection of semigroups and \(\Sigma\)
be a \homspec{3} in \(Q \df \prod_{t \in T} S_t\) that has the replacement
property. For each \(s \in T\) set \(\Sigma_s = \pi_s^{[3]}(\Sigma)\). Further,
let \(\mu\) be a \(Q\)-valued homogeneous measure on \(\bbB\) with \(\Sigma(\mu)
= \Sigma\) and for \(s \in S\) let \(\mu_s \df \pi_s \circ \mu\dd \bbB \to S_s\)
\UP(that is, \(\mu(x) = (\mu_t(x))_{t \in T}\) for \(x \in \bbB\)\UP). Then for
each \(s \in T\):
\begin{enumerate}[\upshape(a)]
\item \(\Sigma_s\) is a \homspec{3}; and \(\Sigma_s\) is solid \UP(resp.
 supersolid\UP) provided so is \(\Sigma\);
\item \(\mu_s\) is a homogeneous measure such that \(\Sigma(\mu_s) = \Sigma_s\).
\end{enumerate}
\end{pro}
\begin{proof}
First we will show item (b), which will cover the statement that \(\Sigma_s\) is
a \homspec{3} and in fact is a stronger claim. It is clear that \(\Sigma(\mu_s)
= \Sigma_s\). So, it is sufficient to show that \(\mu_s\) has (4EP) (as it is
trivial that \(\mu_s\) is an \(S_s\)-valued measure). To this end, we fix
a partition \(\{a,b,c\}\) of \(\bbB\) and a proper element \(d\) of \(\bbB\)
such that \eqref{eqn:PN-4EP} holds for \(\mu_s\). Express \(d\) in the form \(d
= \BooV{f}{g}\) where \(f\) and \(g\) are disjoint and proper, and set \(d' \df
\BooD{\unit}{d}\). Denote \((\mu(f),\mu(g),\mu(d')) = ((x_t)_{t \in T},
(y_t)_{t \in T},(z_t)_{t \in T})\). Since \((\mu(f),\mu(g),\mu(d'))\) and
\((\mu(a),\mu(b),\mu(c))\) belong to \(\Sigma\), the replacement property
implies that also \(((u_t)_{t \in T},(v_t)_{t \in T},(w_t)_{t \in T})\) belongs
to \(\Sigma\) where \((u_t,v_t,w_t) = (x_t,y_t,z_t)\) for \(t \neq s\) and
\((u_s,v_s,w_s) = (\mu_s(a),\mu_s(b),\mu_s(c))\). So, there is a partition
\(\{k,\ell,m\}\) of \(\bbB\) such that \((\mu(k),\mu(\ell),\mu(m)) =
((u_t)_{t \in T},(v_t)_{t \in T},(w_t)_{t \in T})\). Now the equalities
\(\mu_s(a)+\mu_s(b) = \mu_s(d)\) and \(\mu(f)+\mu(g) = \mu(d)\) imply that
\(\mu(k)+\mu(\ell) = \mu(d)\). Moreover, \(\mu(m) = (z_t)_{t \in T} = \mu(d')\).
We conclude from (4EP) of \(\mu\) that \(d\) may be written in the form \(d =
\BooV{p}{q}\) where \(p\) and \(q\) are two disjoint proper elements of \(\bbB\)
such that \((\mu(p),\mu(q)) = (\mu(k),\mu(\ell))\). Then \((\mu_s(p),\mu_s(q)) =
(\mu_s(a),\mu_s(b))\) and we are done.\par
Finally, to show that (super)solidity is preserved in this context, instead of
checking (Sp4)--(Sp5) axioms we will verify the defining conditions. To this
end, assume \(\mu\) is (super)solid and note that \(\OPN{Aut}(\mu) \subset
\OPN{Aut}(\mu_s)\) (for each \(s \in S\)). So, the action of the latter group is
transitive (resp. minimal) if so is the action of the former group. This
completes the proof.
\end{proof}

The above theorem says that, roughly speaking, if the trinary spectrum of
a homogeneous measure that takes values in the product of semigroups is
sufficiently rich, then all its coordinate functions are also homogeneous
measures. In the next section we will establish certain results in the reverse
direction---that is, we will discuss the question of when the direct (or
similar) product of solid \homspec{3}s is a \homspec{3} as well.

\section{Products of [super]solid \homspec{3}s}\label{sec:products}

\begin{dfn}{PN-product}
Everywhere below \(\Sigma_s\) (whatever \(s\) is) denotes a subset of \(S_s^3\)
(where \(S_s\) is a semigroup).\par
A \emph{skew product} of \(\Sigma_1,\ldots,\Sigma_N\) is the subset
\(\overline{\oplus}_{k \leq N} \Sigma_k\) of \((S_1 \times \ldots \times
S_N)^3\) defined as follows:
\[\overline{\bigoplus}_{k \leq N} \Sigma_k \df \{((x_k)_{k=1}^N,(y_k)_{k=1}^N,
(z_k)_{k=1}^N)\dd\ (x_k,y_k,z_k) \in \Sigma_k\ (k=1,\ldots,N)\}.\]
So, the skew product consists only of rearranging the coordinates of
the standard Cartesian product.\par
We call a sequence \(\Sigma_1,\Sigma_2,\ldots\) \emph{null} if \((0_{S_n},
0_{S_n},0_{S_n}) \in \Sigma_n\) for all sufficiently large \(n\). If \(\Sigma_1,
\Sigma_2,\ldots\) is null, we define its \emph{skew direct product} as
the subset \(\overline{\oplus}_{n \geq 1} \Sigma_n\) of \((\prod_{n=1}^{\infty}
S_n)^3\) consisting of all triples of the form
\begin{equation}\label{eqn:PN-aux80}
((x_n)_{n=1}^{\infty},(y_n)_{n=1}^{\infty},(z_n)_{n=1}^{\infty})
\end{equation}
where \((x_n,y_n,z_n) \in
\Sigma_n\) for all \(n\) and \((x_n,y_n,z_n) = (0_{S_n},0_{S_n},0_{S_n})\) for
all sufficiently large \(n\).\par
Finally, for any \(\Sigma\) we define its \emph{skew countable power} as
the subset \(\Sigma^{\bar{\omega}}\) of \(S^{\omega}\) consisting of all
eventually constant sequences of the form \eqref{eqn:PN-aux80} where
\((x_n,y_n,z_n) \in \Sigma\)
for all \(n\) (`eventually constant' means that \((x_n,y_n,z_n) = (x_N,y_N,
z_N)\) for some \(N\), that depends on these three sequences, and all \(n \geq
N\)).
\end{dfn}

In Section~\ref{sec:0} we will give a full classification/characterisation of
\homspec{3}s \(\Sigma\) in \(S\) such that \((0_S,0_S,0_S) \in \Sigma\) (consult
\COR{PN-zero-h3s} below).

\begin{rem}{PN-coord-finite}
If \(\Sigma\) is a subset of \(Q^3\) where \(Q\) is the full product of
a \textbf{finite} collection of semigroups \(S_1,\ldots,S_N\) and \(\Sigma_k =
\pi_k^{[3]}(\Sigma)\), then \(\Sigma\) has the replacement property iff \(\Sigma
= \overline{\oplus}_{k \leq N} \Sigma_k\).
\end{rem}

\begin{rem}{PN-prod-non-homspec}
The next result says that all kinds of `product-like' constructions
introduced in \DEF{PN-product} preserve the property of being solid (resp.
supersolid) \homspec{3}s. It is an intriguing phenomenon that a counterpart of
this result for all \homspec{3}s is false in general. Namely, there exists
a full non-atomic probability Borel measure \(\mu\) on the \Cantor\ that is
homogeneous---so, its trinary spectrum \(\Sigma\) is a \homspec{3}---but
\(\Sigma \bar{\oplus} \Sigma\) is not a \homspec{3}. (Indeed, as \(\mu\) one may
take the measure described in \EXM{PN-homo-bad}---we leave the details to
interested readers.) So, a slight difference between (Sp4) and (Sp4b) makes
a big difference.
\end{rem}

\begin{pro}{PN-prod-solid}
Each of the following sets is a \UP(super\UP)solid \homspec{3} with
the replacement property:
\begin{enumerate}[\upshape(a)]
\item the skew product of a finite number of \UP(super\UP)solid \homspec{3}s;
\item the skew product of a null sequence of \UP(super\UP)solid \homspec{3}s;
\item the skew countable power of a \UP(super\UP)solid \homspec{3}.
\end{enumerate}
\end{pro}
\begin{proof}
A verification of each of (Sp0)--(Sp5) can be made coordinatewise (when checking
(Sp5), one needs to use the version with ``all sufficiently large \(N\)'').
The details are left as an exercise.
\end{proof}

\begin{dfn}{PN-independent}
A finite number of semigroup-valued measures \(\mu_1,\ldots,\mu_N\) on \(\bbB\)
(each of which can take values in a different semigroup) are called
\emph{independent} (in the sense of homogeneity) if their `diagonal' \(a \mapsto
(\mu_1(a),\ldots,\mu_N(a))\) is a homogenous measure and \(\Sigma(\mu_1,\ldots,
\mu_N) = \bar{\oplus}_{k \leq N} \Sigma(\mu_k)\).\par
We call a sequence \(\mu_1,\mu_2,\ldots\) of measures on \(\bbB\) \emph{null} if
the corresponding sequence
\[\Sigma(\mu_1),\Sigma(\mu_2),\ldots\]
of \homspec{3}s is null. A countable number of measures \(\mu_1,\mu_2,\ldots\)
that form a null sequence are called \emph{independent} if their diagonal \(a
\mapsto (\mu_1(a),\mu_2(a),\ldots)\) is a homogeneous measure and
\(\Sigma(\mu_1,\mu_2,\ldots) = \bar{\oplus}_{n \geq 1} \Sigma(\mu_n)\).
\end{dfn}

The following is partially a consequence of \PRO{PN-coord}. We skip its easy
proof.

\begin{pro}{PN-independ-WK}
Let \(\mu_n\) be an \(S_n\)-valued measure on \(\bbB\) and let \(\lambda\) stand
for the `diagonal' measure of all \(\mu_n\) \UP(that is, \(\lambda(a) =
(\mu_n(a))_n\) for \(a \in \bbB\)\UP).
\begin{enumerate}[\upshape(A)]
\item If \(\mu_1,\ldots,\mu_N\) are independent, then all these measures are
 homogeneous. If \(\lambda\) is solid or supersolid, so are all \(\mu_n\).
\item If \(\mu_1,\mu_2,\ldots\) form a null sequence and are independent, then:
 \begin{itemize}
 \item for each \(a \in \bbB\) there exists \(m\) such that \(\mu_n(a) =
  0_{S_n}\) for all \(n \geq m\);
 \item all these measures are homogeneous;
 \item if \(\lambda\) is solid or supersolid, so are all \(\mu_n\).
 \end{itemize}
\end{enumerate}
\end{pro}

The following may be considered as the converse (in a sense) to the above
result.

\begin{pro}{PN-made-independ}
Below \(\lambda\) and \(\mu_n\) denote solid measures on \(\bbB\).
\begin{enumerate}[\upshape(A)]
\item For a finite number of measures \(\mu_1,\ldots,\mu_N\) there exist
 Boolean algebra automorphisms \(\phi_k\) of \(\bbB\) \((k=1,\ldots,N)\) such
 that the measures \(\mu_1 \circ \phi_1,\ldots,\mu_N \circ \phi_N\) are
 independent.
\item If \(\mu_1,\mu_2,\ldots\) form a null sequence, then there are Boolean
 algebra automorphisms \(\phi_n\) of \(\bbB\) \((n \geq 1)\) such that \(\mu_1
 \circ \phi_1,\mu_2 \circ \phi_2,\ldots\) are independent.
\item There exists a sequence \((\phi_n)_{n=1}^{\infty}\) of Boolean algebra
 automorphisms of \(\bbB\) such that the measure \(\mu(a) \df
 (\lambda(\phi_n(a)))_{n=1}^{\infty}\) is solid and \(\Sigma(\mu) =
 (\Sigma(\lambda))^{\bar{\omega}}\).
\end{enumerate}
\end{pro}
\begin{proof}
We carry out a simultaneous proof of all the three items. In (A) and (B) we take
as \(\Sigma\) the skew product of all \(\Sigma(\mu_n)\), whereas in (C) let
\(\Sigma\) denote the skew countable power of \(\Sigma(\lambda)\)---in this
case, we set \(\mu_n \df \lambda\) for all \(n \geq 1\). We infer from
\PRO{PN-prod-solid} that \(\Sigma\) is a solid \homspec{3} with the replacement
property. So, there exists a solid measure \(\rho\) on \(\bbB\) such that
\(\Sigma(\rho) = \Sigma\). Denote \(\rho = (\rho_n)_n\). It follows from
\PRO{PN-coord} that each of \(\rho_n\) is solid and \(\Sigma(\rho_n) =
\Sigma(\mu_n)\). So, an application of \THM{PN-ext-4EP} yields a Boolean algebra
automorphism \(\phi_n\) of \(\bbB\) such that \(\rho_n = \mu_n \circ \phi_n\),
and the conclusion of the respective item follows.
\end{proof}

\begin{exm}{PN-intersection}
Assume \(\mu_1,\ldots,\mu_N\) (or \(\mu_1,\mu_2,\ldots\) form a null sequence
and) are independent (super)solid measures. Then their `diagonal' measure \(\mu
= (\mu_n)_n\) is (super)solid as well. However, as it is readily seen,
\(\OPN{Aut}(\mu) = \bigcap_n \OPN{Aut}(\mu_n)\). So, not only the action of each
of \(\OPN{Aut}(\mu_n)\) on \(\bbB\) is transitive (resp. minimal), but also
the action of the intersection of these groups is so.\par
Now if \(\mu_1,\ldots,\mu_N\) (or \(\mu_1,\mu_2,\ldots\) form a null sequence
and) are (super)solid measures, it follows from \PRO{PN-made-independ} that
there are Boolean algebra automorphisms \(\phi_n\) of \(\bbB\) such that
the measures \(\mu_n \circ \phi_n\) are independent. Observe that
\(\OPN{Aut}(\mu_n \circ \phi_n) = \phi_n^{-1} \OPN{Aut}(\mu_n) \phi_n\). So, it
follows from the first paragraph that the action of \(\bigcap_n \phi_n^{-1}
\OPN{Aut}(\mu_n) \phi_n\) is transitive (resp. minimal) on \(\bbB\).\par
The above remarks about finite collections apply, in particular, to good
measures.
\end{exm}

The following basic result appears to be very useful. We leave its proof to
the reader.

\begin{lem}{PN-1-1}
Let \(\mu\) be an \(S\)-valued \UP(super\UP)solid measure on \(\bbB\) and let
\(u\dd S \to Q\) be a one-to-one homomorphism. Then \(u \circ \mu\) is
a \(Q\)-valued measure that is \UP(super\UP)solid as well. Moreover,
\(\OPN{Aut}(u \circ \mu) = \OPN{Aut}(\mu)\).
\end{lem}

As a consequence, we obtain

\begin{cor}{PN-new-solid}
Let \(\mu_1,\ldots,\mu_N\) be independent \UP(super\UP)solid \(\QQQ\)-valued
measures on \(\bbB\) and let \(s_1,\ldots,s_N \in \RRR\) be linearly independent
over \(\QQQ\). Then \(\sum_{k=1}^N s_k \mu_k\) is \UP(super\UP)solid as well.
\end{cor}
\begin{proof}
Let \(\lambda\dd \bbB \ni a \mapsto (\mu_1(a),\ldots,\mu_N(a)) \in \QQQ^N\) and
\(u\dd \QQQ^N \ni (w_1,\ldots,w_N) \mapsto \sum_{k=1}^N s_k w_k \in \RRR\). It
follows that \(\lambda\) is (super)solid and \(u\) is a one-to-one group
homomorphism. Since \(\sum_{k=1}^N s_k \mu_k = u \circ \lambda\), the conclusion
of the result follows from \LEM{PN-1-1}.
\end{proof}

\begin{thm}{PN-super-no-good}
There exist probability Borel measures on the \Cantor\ that are supersolid and
not good. There also exist supersolid signed Borel measures on the \Cantor\
\UP(that are neither nonnegative nor nonpositive\UP).
\end{thm}
\begin{proof}
Let \(\mu\) be a probability Borel measure that is good and rational-valued on
\(\bbB\). We conclude from \PRO{PN-made-independ} that there exists a Boolean
algebra automorphism \(\phi\) such that \(\mu\) and \(\mu \circ \phi\) are
independent. Since both these measures are supersolid and rational-valued on
\(\bbB\), it follows from \COR{PN-new-solid} that \(\nu_a \df (1-a) \mu + a \mu
\circ \phi\) is supersolid for any irrational \(a\). Observe that \(\nu_a\) is
probability if \(a \in (0,1) \setminus \QQQ\). What is more, for any such \(a\),
\(\nu_a\) is not a good measure, as the group generated by \(\nu_a(\bbB)\)
contains numbers from \((0,1)\) that are not in \(\nu_a(\bbB)\) (indeed, there
are rational numbers \(p < 0\) and \(q > 0\) such that \(r \df (1-a)p+aq\)
belongs to \((0,1)\) and then \(r\) is not in \(\nu_a(\bbB)\)). Finally, if
\(a \in (1,2) \setminus \QQQ\), then \(\nu_a\) is neither nonnegative nor
nonpositive.
\end{proof}

\begin{rem}{PN-uncount}
The above proof can be improved to show that the number of measures specified
in \THM{PN-super-no-good} is uncountable. Indeed, take an arbitrary good measure
\(\mu\) on \(\Can\). Then a simple set-theoretic argument shows that there
exists a real number \(a \in (0,1)\) such that the line \((1-a)x+ay = 0\)
intersects the additive subgroup of \(\RRR^2\) generated by \(I \times I\) only
at the origin where \(I \df \mu(\bbB)\). Then \(\hat{\mu} \df (1-a)\mu+a\mu\)
is supersolid and not good. Moreover, it is easy to show that for uncountably
many choices of \(\mu\) the corresponding measures \(\hat{\mu}\) are not
conjugate to each other. The details are left to the reader.
\end{rem}

\section{Homogeneous measures attaining zero value}\label{sec:0}

This section is devoted to homogeneous measures \(\mu\dd \bbB \to S\) for which
\(0_S \in \mu(\bbB \setminus \{\zero,\unit\})\). We will show that each such
a measure is automatically solid and, roughly speaking, ``almost supersolid''
(consult \PRO{PN-zero-solid} below). The main goal of this section is
\THM{PN-zero-super} below which gives a full classification of all such
supersolid measures.\par
In what follows, we will denote by \(\OPN{grp}(S)\) the group of all invertible
elements of \(S\); that is, \(x \in S\) belongs to \(\OPN{grp}(S)\) if there is
\(y \in S\) such that \(x+y = 0_S\). This element \(y\) is unique and we will
denote it by \(-x\) (for \(x \in \OPN{grp}(S)\)).

\begin{dfn}{PN-G-filling}
Let \(G\) be a countable (or finite) group. A measure \(\rho\dd \bbB \to G\) is
said to be \emph{\(G\)-filling} if for all \(g \in G\) and any proper non-zero
\(a \in \bbB\) there is non-zero \(b \in \bbB\) such that \(\BooW{a}{b} = b \neq
a\) and \(\rho(b) = g\). A \(G\)-filling measure \(\nu\) with \(\nu(\unit) = E\)
is also called \emph{\(G\)-filling \(E\)-measure}.
\end{dfn}

Since \(G\)-filling measures satisfy (4EP), it follows that all of them are
homogeneous. Moreover, as it may be easily checked, all they are supersolid.
One shows (using, e.g., the results of the previous sections) that for any
countable (or finite) group \(G\) and each \(E \in G\) there exists
a \(G\)-filling \(E\)-measure. The main aim of this section is to prove
the following

\begin{thm}{PN-zero-super}
For a measure \(\mu\dd \bbB \to S\) \tfcae
\begin{enumerate}[\upshape(i)]
\item \(G = \mu(\bbB)\) is a subgroup of \(\OPN{grp}(S)\) and \(\mu\) is
 \(G\)-filling;
\item \(\mu\) is supersolid and \(0_S \in \mu(\bbB \setminus \{\zero,\unit\})\).
\end{enumerate}
\end{thm}

From now on to the end of this section, \(\mu\dd \bbB \to S\) is a fixed
homogeneous measure and \(\zZ \in \bbB\) is a fixed proper element such that
\(\mu(\zZ) = 0_S\). Additionally, we use \(H\) to denote the automorphism group
of \(\mu\) and we set
\[\OPN{grp}(\mu) \df \{\mu(a)\dd\ a \in \bbB\restriction{\zZ}\}.\]
We will see in the sequel that \(\OPN{grp}(\mu)\) is a subgroup of
\(\OPN{grp}(S)\). The next result implies that \(\OPN{grp}(S)\) does not depend
on the choice of \(\zZ\).

\begin{lem}{PN-zero-1}
Let \(a\) and \(b\) be two proper elements of \(\bbB\). In each of the following
cases there exists \(h \in H\) such that \(h(a) = b\):
\begin{enumerate}[\upshape(a)]
\item \(\mu(a) = \mu(b) = x \in \OPN{grp}(S)\); or
\item \(\mu(\BooD{a}{b}) = \mu(\BooD{b}{a}) = y \in \OPN{grp}(S)\).
\end{enumerate}
\end{lem}
\begin{proof}
To show (a), just observe that then \(\mu(\BooD{\unit}{a}) = \mu(\unit)+(-x) =
\mu(\BooD{\unit}{b})\) and apply homogeneity of \(\mu\). Similarly, if (b)
holds, then \(\mu(a) = \mu(\BooW{a}{b})+y = \mu(b)\) as well as
\(\mu(\BooD{\unit}{a}) = \mu(\BooD{\unit}{(\BooV{a}{b}}))+y =
\mu(\BooD{\unit}{b})\), and homogeneity of \(\mu\) finishes the proof.
\end{proof}

\begin{lem}{PN-zero-2}
There exists a proper element \(\zZ' \in \bbB\) such that \(\BooW{\zZ}{\zZ'} =
\zero\), \(\BooV{\zZ}{\zZ'} \neq \unit\) and \(\mu(\zZ') = 0_S\).
\end{lem}
\begin{proof}
Take an arbitrary proper element \(a \in \bbB\) such that \(\BooW{\zZ}{a} =
\zero\) and \(b \df \BooV{\zZ}{a}\) differs from \(\unit\). It follows from item
(b) of \LEM{PN-zero-1} that there is \(h \in H\) such that \(h(b) = a\). To end
the proof, set \(\zZ' = h(\zZ)\).
\end{proof}

\begin{lem}{PN-zero-3}
The set \(G \df \OPN{grp}(\mu)\) is a subgroup of \(\OPN{grp}(S)\) and for any
non-zero \(j \in \bbB\) such that \(\mu(j) \in G\) the measure
\(\mu\restriction{j}\) is \(G\)-filling.
\end{lem}
\begin{proof}
Let \(\zZ'\) be as specified in \LEM{PN-zero-2}. It follows from \LEM{PN-zero-1}
that
\begin{equation}\label{eqn:PN-zero-1}
\OPN{grp}(\mu) = \{\mu(b)\dd\ b \in \bbB\restriction{\zZ'}\}
= \{\mu(x)\dd\ x \in \bbB\restriction{(\BooV{\zZ}{\zZ'})}\}.
\end{equation}
If \(g \in G\), then \(g = \mu(a)\) for some \(g \in \bbB_0 \df
\bbB\restriction{\zZ}\). Then \(0_S = \mu(a)+\mu(\BooD{\zZ}{a})\) and therefore
\(g \in \OPN{grp}(G)\) and \(-g = \mu(\BooD{\zZ}{a})\) belongs to \(G\).
Further, if \(p, q \in G\), then \(p = \mu(a)\) and \(q = \mu(b)\) for some \(a
\in \bbB\restriction{\zZ}\) and \(b \in \bbB\restriction{\zZ'}\) (by
\eqref{eqn:PN-zero-1}). Then \(p+q = \mu(\BooV{a}{b})\) and \(\BooV{a}{b} \in
\bbB\restriction{(\BooV{\zZ}{\zZ'})}\). Consequently, \(p+q \in G\) (again
thanks to \eqref{eqn:PN-zero-1}) and thus \(G\) is a group.\par
Now let \(j \in \bbB\) be non-zero such that \(\mu(j) \in G\). First assume
\(j\) differs from \(\unit\). Take non-zero \(a \in \bbB\restriction{\zZ}\) such
that \(\mu(j) = \mu(a)\) (such \(a\) exists even if \(\mu(j) = 0_S\)). We infer
from \LEM{PN-zero-1} that \(j = h(a)\) for some \(h \in H\). In particular,
\(\mu(\bbB\restriction{j}) \subset G\). Fix a non-zero element \(j_0\) of
\(\bbB\restriction{j}\). There exists \(u \in H\) such that
\(u(\BooV{h^{-1}(j_0)}{\zZ'}) = h^{-1}(j_0)\) (by the same lemma). Then \(b \df
h(u(\zZ'))\) belongs to \(\bbB\restriction{j_0}\) and satisfies \(\mu(b) =
0_S\). So, there is \(k \in H\) such that \(k(\zZ) = b\). Now for any \(g \in
G\) there is \(x \in \bbB\restriction{\zZ}\) such that \(\mu(x) = g\) and then
\(k(x) \in \bbB\restriction{j_0}\) is such that \(\mu(k(x)) = g\), which shows
that \(\mu\restriction{j}\) is \(G\)-filling.\par
Finally, assume \(j = \unit\). Then \(\mu(\BooD{\unit}{\zZ}) = \mu(\unit)\)
belongs to \(G\) as well. So, it follows from the previous paragraph that both
\(\mu\restriction{\zZ}\) and \(\mu\restriction{(\BooD{\unit}{\zZ})}\) are
\(G\)-filling. Hence \(\mu(\bbB) \subset G\) and, by the same part of the proof,
\(\mu\restriction{w}\) is \(G\)-filling for each proper \(w \in \bbB\).
Consequently, \(\mu\) is \(G\)-filling, and we are done.
\end{proof}

\begin{lem}{PN-zero-4}
For any non-zero element \(j\) of \(\bbB\) there is non-zero \(x \in
\bbB\restriction{j}\) such that \(\mu(x) = 0_S\).
\end{lem}
\begin{proof}
If \(w \df \BooW{\zZ}{j}\) is non-zero, then there exists non-zero \(x \in
\bbB\restriction{w}\) such that \(\mu(x) = 0_S\), as \(\mu\restriction{\zZ}\) is
\(G\)-filling (by \LEM{PN-zero-3}). On the other hand, if \(w = \zero\), then
there is \(h \in H\) such that \(h(\BooV{\zZ}{j}) = j\) (thanks to
\LEM{PN-zero-1}). Then \(\mu(x) = 0_S\) for \(x = h(\zZ)\).
\end{proof}

Below we summarize all that have been established in this section. We recall
that a subset \(\eeE\) of \(\bbB\) is an \emph{ideal} if it is non-empty and
\(\BooV{a}{b} \in \eeE\) as well as \(\bbB\restriction{a} \subset \eeE\) for all
non-zero \(a, b \in \eeE\).

\begin{pro}{PN-zero-solid}
The measure \(\mu\) is solid and the set \(G = \OPN{grp}(\mu)\) is a subgroup of
\(\OPN{grp}(S)\). The collection \(\ddD\) of all \(a \in \bbB\) such that
\(\mu(a) \in G\) is an ideal in \(\bbB\) such that every non-zero member of
\(\bbB\) contains a non-zero element of \(\ddD\). Moreover, for any non-zero \(j
\in \ddD\) the measure \(\mu\restriction{j}\) is \(G\)-filling \UP(in
particular, it is supersolid\UP).
\end{pro}
\begin{proof}
We only need to show that \(\mu\) is solid and \(\BooV{a}{b} \in \ddD\) for all
\(a, b \in \ddD\). The former of these properties follows from \LEM{PN-zero-4}
and item (a) of \LEM{PN-zero-1}, whereas the latter simply follows from
\LEM{PN-zero-3}, as \(\mu(\BooV{a}{b}) = \mu(a)+\mu(\BooD{b}{a})\).
\end{proof}

\begin{rem}{PN-zero-sigma}
It easily follows from \LEM{PN-zero-4} that \(\Sigma' \times \{0_S\} \subset
\Sigma\) for \(\Sigma \df \Sigma(\mu)\). Thus, in this case (Sp4) is equivalent
to (Sp4b), which yields an indirect proof that \(\mu\) is solid.
\end{rem}

\begin{proof}[Proof of \THM{PN-zero-super}]
Thanks to \PRO{PN-zero-solid}, we only need to show that \(\mu(\unit) \in
\OPN{grp}(\mu)\) for any measure satisfying (ii); or, equivalently, that \(\unit
\in \ddD\) where \(\ddD\) is as specified in the last cited result. To this end,
we keep the notation of this section and take a finite number of automorphisms
\(\phi_1,\ldots,\phi_N \in \OPN{Aut}(\mu)\) such that \(\unit = \bigvee_{k=1}^N
\phi_k(\zZ)\). Then \(\phi_k(\zZ) \in \ddD\) and also \(\unit \in \ddD\) (by
\PRO{PN-zero-solid}), and we are done.
\end{proof}

\begin{cor}{PN-zero-h3s}
A \homspec{3} \(\Sigma\) in \(S\) contains \((0_S,0_S,0_S)\) iff
\begin{equation}\label{eqn:PN-zero-h3s}
\Sigma = \{(x,y,z) \in G^3\dd\ x+y+z = 0_S\}
\end{equation}
for some subgroup \(G\) of \(\OPN{grp}(S)\). Moreover, if
\eqref{eqn:PN-zero-h3s} holds and \(\mu\dd \bbB \to S\) is a homogeneous measure
such that \(\Sigma(\mu) = \Sigma\), then \(\mu\) is a \(G\)-filling
\(0_G\)-measure.
\end{cor}
\begin{proof}
The `if' part is trivial. To prove the `only if' one, assume \((0_S,0_S,0_S) \in
\Sigma\) and take a homogeneous measure \(\mu\) whose trinary spectrum coincides
with \(\Sigma\). Observe that then \(0_S \in \mu(\bbB \setminus
\{\zero,\unit\})\) and apply \PRO{PN-zero-solid} to get \eqref{eqn:PN-zero-h3s}
and the additional claim of the result.
\end{proof}

\begin{rem}{PN-zero-E}
We wish to emphasize that the results of this section do not apply to
homogeneous measures \(\mu\) satisfying \(\mu(\unit) = 0_S\). Although in such
cases \(\mu(\bbB) \subset \OPN{grp}(S)\), it may happen that \(0_S \notin
\mu(\bbB \setminus \{\zero,\unit\})\)---as the following general example shows:
If \(\rho\dd \bbB \to [0,1]\) is a good measure on \(\Can\), then \(\mu \df \pi
\circ \rho\) is a supersolid measure such that \(\mu(\unit) = 0_{\RRR/\ZZZ}\)
and \(\mu\) is one-to-one on \(\bbB \setminus \{\zero\}\) where \(\pi\dd \RRR
\to \RRR/\ZZZ\) is the quotient epimorphism. (In particular, \(\OPN{Aut}(\mu) =
\OPN{Aut}(\rho)\).)
\end{rem}

\section{Group-valued homogeneous and solid measures}\label{sec:group}

When \(S = G\) is a group (or, more generally, when \(S\) is a cancellative
semigroup and one embeds \(S\) into a group \(G\)), axioms (Sp0)--(Sp5) may
be simplified thanks to (Sp2). In this section we describe how to define
\textbf{binary} spectra of group-valued homogeneous measures and formulate
simpler axioms than the aforementioned. The fact that the situation is becoming
simpler is evidenced by the result below.\par
From now on to the end of this section \(G\) denotes an Abelian group (again, we
use additive notation and denote the neutral element of \(G\) by \(0_G\)).\par
Let \(\mu\) be a \(G\)-valued measure on \(\bbB\). Observe that the function
\[\Sigma(\mu) \ni (x,y,z) \mapsto (x,y) \in G \times G\]
is one-to-one, as \(\mu(c) = \mu(\unit)-\mu(a)-\mu(b)\) for any partition
\(\{a,b,c\}\) of \(\bbB\). This remark motivates us to work with a binary
spectrum of \(\mu\) instead of \(\Sigma(\mu)\). However, \(e(\Sigma(\mu))\) is
an ``invariant'' of \(\mu\) that we also need to remember. Therefore we
introduce

\begin{dfn}{PN-binary-spec}
For any \(G\)-valued measure \(\mu\) we use \(E(\mu)\) to denote \(\mu(\unit)\)
and \(\Lambda(\mu)\) is defined as follows:
\[\Lambda(\mu) \df \{(x,y)\dd\ (x,y,z) \in \Sigma(\mu)\}.\]
We call \(\Lambda(\mu)\) the \emph{binary spectrum} of \(\mu\). Additionally, we
set \(\Lambda'(\mu) \df \{a\dd\ (a,b) \in \Lambda(\mu)\}\). Note that:
\begin{itemize}
\item \(\Lambda(\mu)\) consists precisely of all pairs of the form
 \((\mu(a),\mu(b))\) where \(a\) and \(b\) are arbitrary disjoint proper
 elements of \(G\) such that \(\BooV{a}{b} \neq \unit\);
\item \(\Lambda'(\mu) = \mu(\bbB \setminus \{\zero,\unit\})\).
\end{itemize}
A pair \((E,\Lambda)\) is said to be a \emph{\homspec{2}} (in \(G\)) if all
the following conditions are fulfilled:
\begin{enumerate}[(GrSp1a)]
\item[(GrSp0)] \(\Lambda \subset G \times G\) is non-empty and at most
 countable, and \(E \in G\).
\item[(GrSp1)] If \((a,b) \in \Lambda\), then both \((b,a)\) and \((a,E-a-b)\)
 belong to \(\Lambda\) as well.
\item[(GrSp2)] If \(a,b,c \in G\) are such that \((a,b),(a+b,c) \in \Lambda\),
 then also \((a,b+c) \in \Lambda\).
\item[(GrSp3a)] For any \((a,b) \in \Lambda\) there exists \((x,y) \in \Lambda\)
 such that \(a = x+y\).
\item[(GrSp3b)] For any
 \begin{itemize}
 \item \(a, b, c \in G\) such that \((a,b),(c,a+b-c) \in \Lambda\); or
 \item \(a, c \in \Lambda' \df \{p\dd\ (p,q) \in \Lambda\}\) and \(b \df E-a\)
 \end{itemize}
 there exists \(x \in G\) such that one of \((x,a-x)\) and \((c-x,b-c+x)\)
 belongs to \(\Lambda \cup (\Lambda' \times \{0_G\})\) and the other to
 \(\Lambda \cup (\{0_G\} \times \Lambda')\).
\end{enumerate}
For any \homspec{2} \((E,\Lambda)\) the set \(\Lambda'\) is defined as specified
in (GrSp3b). A \emph{solid \homspec{2}} is a pair \((E,\Lambda)\) that satisfies
(GrSp0)--(GrSp2) and (GrSp3) formulated below.
\begin{enumerate}[(GrSp1)]\addtocounter{enumi}{2}
\item For any
 \begin{itemize}
 \item \(a, b, c \in G\) such that \((a,b),(c,a+b-c) \in \Lambda\); or
 \item \(a, c \in \Lambda'\) and \(b \df E-a\)
 \end{itemize}
 there exists \(x \in G\) such that \((x,a-x),(c-x,b-c+x) \in \Lambda\).
\end{enumerate}
Finally, a \emph{supersolid \homspec{2}} is a solid \homspec{2} \((E,\Lambda)\)
which satisfies the following condition:
\begin{enumerate}[(GrSp1)]\addtocounter{enumi}{3}
\item For any \(x \in \Lambda'\) there exist \(b_1,\ldots,b_N \in G\) (for
 some \(N \geq 3\) or, equivalently, for all sufficiently large \(N\)) such that
 \(\sum_{k=1}^N b_k = E\) and all the following conditions are fulfilled:
 \begin{enumerate}[\upshape(GrSp4a)]
 \item \((\sum_{s=1}^{k-1} b_s,b_k) \in \Lambda\) whenever \(1 < k < N\);
 \item \((b_k,x-b_k) \in \Lambda\) for all \(k=1,\ldots,N\).
 \end{enumerate}
\end{enumerate}
\end{dfn}

The proof of the following technical result is left to the reader.

\begin{pro}{PN-hom2=hom3}
Let \(G\) be a group and \(E \in G\).
\begin{itemize}
\item A set \(\Sigma \subset G^3\) such that \(a+b+c = E\) for all \((a,b,c) \in
 \Sigma\) is a \homspec{3} \UP(resp. a solid or supersolid \homspec{3}\UP) iff
 the pair \((E,\Lambda)\) is a \homspec{2} \UP(resp. a solid or supersolid
 \homspec{2}\UP) where \(\Lambda \df \{(a,b)\dd\ (a,b,c) \in \Sigma\}\).
\item Conversely, if \(\Lambda \subset G^2\), then \((E,\Lambda)\) is
 a \homspec{2} \UP(resp. a solid or a supersolid \homspec{2}\UP) iff the set
 \(\Sigma \df \{(a,b,E-a-b)\dd\ (a,b) \in \Lambda\}\) is a \homspec{3} \UP(resp.
 a solid or supersolid \homspec{3}\UP).
\end{itemize}
\end{pro}

\begin{exm}{PN-G-E}
\begin{enumerate}[\upshape(A)]
\item Let \(G\) be at most countable and \(E \in G\) be arbitrary. The pair
 \((E,G^2)\) is a supersolid \homspec{2}. We call it \emph{\(E\)-full
 \homspec{2}}. A \(G\)-valued measure \(\mu\) on \(\bbB\) is homogeneous and
 satisfies \((E(\mu),\Lambda(\mu)) = (E,G^2)\) iff \(\mu\) is a \(G\)-filling
 \(E\)-measure.
\item Let \(G\) be a two-point group and let \(E\) denote the unique element
 of \(G\) different from \(0_G\). We leave it as a simple exercise that the only
 \homspec{2}s in \(G\) are: \((0_G,\{(0_G,0_G)\})\) (supersolid, a homogeneous
 measure corresponding to it is constantly equal to \(0_G\)); \(0_G\)-full
 \homspec{2} and \(E\)-full \homspec{2}.
\item Let \(\mu\) be a good measure on the \Cantor\ and let \(I = \mu(\bbB
 \setminus \{\zero,\unit\})\). Then \(\mu\) corresponds to the \homspec{2}
 \((1,\{(a,b) \in I^2\dd\ a+b<1\})\).
\end{enumerate}
\end{exm}

Although axioms from the family (GrSp*) look much simpler than from the family
(Sp*), it is still hard to work with them in practice. Actually, they serve as
a `minimal' set of conditions that guarantee the existence of a homogeneous
measure with a prescribed trinary/binary spectrum. We could now deduce further
consequences of the defining axioms, but instead we will introduce new
structures that induce group-valued homogeneous measures and, in our opinion,
are more handy. As measures that attain a zero value on a proper element of
\(\bbB\) have been well studied in the previous section, and also for
simplicity, here we restrict our investigations only to group-valued solid
measures that do not have the last mentioned property.\par
In what follows, partial orders are understood as strict (sharp); that is, they
are (by definition) anti-symmetric (which means that if \(x < y\), then
necessarily \(y \nless x\)) and transitive (that is, \(x < y\) and \(y < z\)
imply \(x < z\)).\par
As we do not demand homogeneous measures to assume non-zero value on \(\unit\),
we introduce the following concept to simplify the presentation of our results.
We assume that \(\varempty\) represents an element that does not belong to any
(semi)group and we add it to all of them as a new neutral element (in this way
a group is no longer a group). More specifically, if \(Q\) is a (semi)group,
then \(\varempty \notin Q\), \(Q^{\circ} = Q \cup \{\varempty\}\) and
\(g+\varempty = \varempty+g = g-\varempty = g\) for any \(g \in Q^{\circ}\).
When \(Q\) is a group, subtraction of distinct elements of \(Q^{\circ}\) is
possible and does not lead to a confusion. (From now on to the end of
the section, we will always subtract only different elements of a group).

\begin{dfn}{PN-pospec}
Let \(G\) be a group. A \emph{solid partially ordered spectrum} (in short,
a \emph{solid pospec}) in \(G\) is a set \(J \subset G^{\circ}\) equipped with
a strict partial order `\(\prec\)' such that all the following conditions are
fulfilled:
\begin{enumerate}[(PoSp1)]\addtocounter{enumi}{-1}
\item \(J\) is countably infinite and there is \(\Omega \in G\) such that
 \(\varempty\) and \(\Omega\) belong to \(J\) and satisfy \(\varempty \prec x
 \prec \Omega\) for all \(x \in I \df J \setminus \{\varempty,\Omega\}\). (So,
 \(\varempty\) and \(\Omega\) are, respectively, the least and the greatest
 elements of the poset \((J,\prec)\).)
\item If \(a, b \in J\), \(a \prec b\) and \((a,b) \neq (\varempty,\Omega)\),
 then \(b-a \in I\).
\item If \(a, b, c \in J\) satisfy \(a \prec b \prec c\), then \(c-b \prec c-a\)
 and \(b-a \prec c-a\).
\item If \(a, b, c \in J\) are such that \(a \prec c\) and \(\varempty \neq b
 \prec c-a\), then \(a+b \in I\) and \(a \prec a+b \prec c\).
\item For all \(a, b, c \in J\) satisfying \(\varempty \prec a \prec c\) and
 \(\varempty \prec b \prec c\) there is \(x \in I\) such that \(x \prec a\) and
 \(x \prec b\) and \(a-x \prec c-b\).
\end{enumerate}
\end{dfn}

\begin{rem}{PN-pospec}
Let us list some further consequences of axioms (PoSp1)--(PoSp4). Continuing
the notation introduced in the previous definition, we have:
\begin{enumerate}[(PoSp1')]
\item \(\Omega-a \in I\) for each \(a \in I\); \(0_G \notin I\); and \(b-a \in
 J\) whenever \(a, b \in J\) satisfy \(a \prec b\). (However, it may happen that
 \(\Omega = 0_G\).)
\item If \(\varempty \prec b \prec c\), then \(c-b \prec c\) (because it is
 sufficient to apply the former property from (PoSp2) with \(a = \varempty\)).
\addtocounter{enumi}{1}
\item The set \(I\) is both downward and upward directed; that is, for all \(x,
 y \in I\) there are \(a, b \in I\) such that \(a \prec x \prec b\) and \(a
 \prec y \prec b\). In particular, \(I\) has no minimal nor maximal elements.
\end{enumerate}
Indeed, to convince oneself that \(I\) is upward directed, take any two elements
\(a\) and \(b\) from \(I\) and note that both \(a' \df \Omega-a\) and \(b' \df
\Omega-b\) belong to \(I\) as well (by (PoSp1')). Since \(a' \prec \Omega\) and
\(b' \prec \Omega\), it follows from (PoSp4) that there is \(x' \in I\) for
which \(x' \prec a'\) and \(x' \prec b'\). Then \(x \df \Omega-x'\) belongs to
\(I\) and satisfies \(a \prec x\) and \(b \prec x\), thanks to (PoSp2).\par
We also remark that the latter claim of (PoSp2) is a consequence of the former
one. Indeed, if we assume that \(a \prec b \prec c\) implies that \(c-b \prec
c-a\), then (PoSp2') is valid. Hence \(c-b \prec c-a\) is followed by \(b-a =
(c-a)-(c-b) \prec c-a\).
\end{rem}

The role of solid pospecs is explained by the next result, where we keep
the notation introduced in \DEF{PN-pospec}.

\begin{thm}{PN-pospec}
Let \(G\) be a group.
\begin{enumerate}[\upshape(A)]
\item If \((E,\Lambda)\) is a solid \homspec{2} such that \(0_G \notin
 \Lambda'\), then \(E \notin \Lambda'\) and the set \(J \df \{\varempty,E\} \cup
 \Lambda'\) is a solid pospec when `\(\prec\)' is defined as follows:
 \[a \prec b \iff a = \varempty\ \lor\ b = E\ \lor\ (a,b-a) \in \Lambda \qquad
 (a,b \in J,\ a \neq b).\]
 \UP(Moreover, in the above situation, \(\Omega = E\) and \(I = \Lambda'\).\UP)
\item Conversely, if \((J,\prec)\) is a solid pospec, then \((E,\Lambda)\) is
 a solid \homspec{2} such that \(0_G \notin \Lambda'\) where \(E \df \Omega\)
 and \(\Lambda\) consists precisely of all those pairs \((a,b)\) from \(I \times
 I\) such that \(a+b \in I\) and \(a \prec a+b\). \UP(Moreover, \(\Lambda' =
 I\).\UP)
\end{enumerate}
The constructions described in items \UP{(A)--(B)} above are mutually inverse to
each other.
\end{thm}
\begin{proof}
First we will prove (A). So, assume \((E,\Lambda)\) is a solid \homspec{2} such
that \(0_G \notin \Lambda'\). Then also \(E \notin \Lambda'\), thanks to (GrSp1)
and (GrSp3a). Indeed, if \(a \in \Lambda'\), then it follows from (GrSp3a) that
\(a = x+y\) for some \((x,y) \in \Lambda\), and then---by (GrSp1)---also
\((x,E-a)\) as well as \((E-a,x)\) belong to \(\Lambda\), which gives
\begin{equation}\label{eqn:PN-aux45}
E-a \in \Lambda' \qquad (a \in \Lambda').
\end{equation}
We will now show that `\(\prec\)' is a strict partial order on \(J\). So, assume
\(a \prec b\) and \(b \prec c\) (where \(a,b,c \in J\)). If \(a = \varempty\) or
\(c = E\), then surely \(a \prec c\). Hence, we further assume that \(a \neq
\varempty\) and \(c \neq E\). Then, since \(a \prec b\) and \(b \prec c\), we
infer that both the pairs \((a,b-a)\) and \((b,c-b)\) belong to \(\Lambda\).
Then (GrSp2) implies that \((a,c-a) \in \Lambda\) (indeed, setting \(x \df a\),
\(y \df b-a\) and \(z \df c-b\), we have \((x,y),(x+y,z) \in \Lambda\) and
therefore \((x,y+z) \in \Lambda\), by (GrSp2)). This implies that \(a \neq c\)
(as \(0_G \notin \Lambda\)) and \(a \prec c\). Transitivity of `\(\prec\)'
immediately implies that it is anti-symmetric (because if \(a \prec b\) and \(b
\prec a\), then \(a \prec a\), which does not hold, by the very definition of
`\(\prec\)'). Now it is clear that \(\varempty\) and \(\Omega \df E\) are,
respectively, the least and the greatest elements of \((J,\prec)\); and that \(I
= \Lambda'\). So, (PoSp0) holds (it will follow from (PoSp4) that \(J\) is
infinite). (PoSp1) is a direct consequence of \eqref{eqn:PN-aux45}: if \(a \prec
b\) for \((a,b) \in (J \times J) \setminus \{(\varempty,\Omega)\}\), then either
\(a = \varempty\) (and then \(b-a = b \in I\)) or \(b = \Omega\) (and then \(b-a
= E-a \in I\), by \eqref{eqn:PN-aux45}), or \((a,b-a) \in \Lambda\) (and then
\(b-a \in I\)). To establish the remaining axioms, we now take a solid measure
\(\mu\dd \bbB \to G\) such that \(\mu(\unit) = E\) and \(\Lambda(\mu) =
\Lambda\). Observe that for any \(a, b \in J\):
\begin{itemize}
\item[\((*)\)] \(\varempty \neq a \prec b\) iff there are two disjoint non-zero
 elements \(\alpha, \beta \in \bbB\) such that \((a,b) =
 (\mu(\alpha),\mu(\alpha)+\mu(\beta))\).
\end{itemize}
(The above equivalence can simply be verified separately when \(b = E\) and when
\(b \neq E\). We leave it to the reader.) Moreover, it follows from homogeneity
of \(\mu\) that if \(\alpha \in \bbB\) is fixed so that \(\mu(\alpha) = a\),
then we can always find (under the assumption that \(a \prec b\)) non-zero
\(\beta \in \bbB\) disjoint from \(\alpha\) such that \(b = \mu(\alpha)+
\mu(\beta)\). With the aid of \((*)\), axioms (PoSp2)--(PoSp4) may briefly be
explained as follows. If \(a \prec b \prec c\), then there are disjoint
\(\alpha, \beta, \gamma \in \bbB\) such that \(\beta\) and \(\gamma\) are
non-zero (and \(\alpha\) is non-zero iff \(a \neq \varempty\)), and \(a+0_G =
\mu(\alpha)\) (in this way we cover the case when \(a = \varempty\)), \(b =
\mu(\alpha)+\mu(\beta)\) and \(c = \mu(\alpha)+\mu(\beta)+\mu(\gamma)\). Then
\(c-b = \mu(\gamma)\), \(c-a = \mu(\beta)+\mu(\gamma)\) and \(b-a =
\mu(\beta)\), which implies (thanks to \((*)\)) that \(c-b \prec c-a\) as well
as \(b-a \prec c-a\) (this shows (PoSp2)). Now we pass to (PoSp3). So, assume
\(a \prec c\) and \(\varempty \neq b \prec c-a\). If \(a = \varempty\), then
trivially
\begin{equation}\label{eqn:PN-aux46}
a+b \in I \qquad \UP{and} \qquad a \prec a+b \prec c.
\end{equation}
So, we assume \(a \neq \varempty\). Then there are two disjoint non-zero
elements \(\alpha, \beta \in \bbB\) such that \((a,c) = (\mu(\alpha),
\mu(\alpha)+\mu(\beta))\) (by \((*)\). Further, since \(\varempty \neq b \prec
c-a = \mu(\beta)\), homogeneity of \(\mu\) implies that there are non-zero
\(\gamma, \delta \in \bbB\) for which \(b = \mu(\gamma)\) and \(\beta =
\BooV{\alpha}{\BooV{\gamma}{\delta}}\) (again by \((*)\)). So, \(a =
\mu(\alpha)\), \(a+b = \mu(\alpha)+\mu(\gamma) = \mu(\BooV{\alpha}{\gamma})\)
and \(c = \mu(\alpha)+\mu(\beta)+\mu(\gamma)\), which yield \eqref{eqn:PN-aux46}
and prove (PoSp3). Finally, we pass to (PoSp4). (This is the only part where we
will use the property that \(\mu\) is solid.) As this is the hardest part of
the theorem, we divide the proof of (PoSp4) into a few steps.\par
Firstly, we will show that \((I,\prec)\) is downward directed. To this end, fix
\(a, b \in I\) and choose proper elements \(\alpha, \beta \in \bbB\) such that
\((a,b) = (\mu(\alpha),\mu(\beta))\). Since \(\mu\) is solid, we conclude that
there exists a non-zero \(\delta \in \bbB\) and an automorphism \(\phi \in
\OPN{Aut}(\mu)\) such that
\begin{equation}\label{eqn:PN-aux47}
\BooW{\delta}{\alpha} = \delta \neq \alpha \qquad \UP{and} \qquad
\BooW{\phi(\delta)}{\beta} = \phi(\delta) \neq \beta.
\end{equation}
Then \(d \df \mu(\delta)\) belongs to \(I\) and coincides with
\(\mu(\phi(\delta))\). Consequently, \(d \prec a\) (by the former relation in
\eqref{eqn:PN-aux47}) as well as \(d \prec b\) (by the latter one therein),
which finishes the proof that \(I\) is downward directed.\par
Now, as specified in (PoSp4), we fix \(a,b,c \in J \setminus \{\varempty\}\)
such that \(a \prec c\) and \(b \prec c\). As usual, homogeneity of \(\mu\)
combined with \((*)\) gives us three non-zero elements \(\alpha, \beta, \gamma
\in \bbB\) such that \((a,b,c) = (\mu(\alpha),\mu(\beta),\mu(\gamma))\) and
\(\BooW{\alpha}{\gamma} = \alpha \neq \gamma\) and \(\BooW{\beta}{\gamma} =
\beta \neq \gamma\). It follows from (PoSp1) and (PoSp2') that both \(c-a\) and
\(c-b\) belong to \(I\) and satisfy \(c-a \prec c\) and \(c-b \prec c\). Set
\(\delta \df \BooW{\alpha}{\beta}\). We distinguish seven distinct cases:
\begin{itemize}
\item Assume \(\delta = \zero\) and \(\BooV{\alpha}{\beta} = \gamma\). Then
 \(c-b = a\) and, since \(I\) is downward directed, there is \(x \in I\) such
 that \(x \prec a\) and \(x \prec b\). We infer from (PoSp2') that \(a-x \prec
 a = c-b\), and we are done.
\item Now assume \(\delta = \zero\) and \(\BooV{\alpha}{\beta} \neq \gamma\).
 Then \(\BooW{\alpha}{(\BooD{\gamma}{\beta})} = \alpha \neq
 \BooD{\gamma}{\beta}\) and hence \(a = \mu(a) \prec \mu(\BooD{\gamma}{\beta}) =
 c - b\). Choose \(x \in I\) such that \(x \prec a\) and \(x \prec b\). Then
 \(a-x \prec a\) (by (PoSp2')) and, consequently, \(a-x \prec c-b\), as we
 wished.
\item This time assume \(\delta = \alpha = \beta\). Then \(a = b\) and there is
 \(y \in I\) such that \(y \prec a\) and \(y \prec c-b\). Then \(x \df a-y\)
 belongs to \(I\) (by (PoSp1)) and \(x \prec a = b\) (by (PoSp2')) and \(a-x = y
 \prec c-b\), and we are done.
\item Now assume \(\delta = \alpha \neq \beta\). In particular, \(a =
 \mu(\delta) \prec \mu(\beta) = b\). Choose \(y \in I\) such that \(y \prec a\)
 and \(y \prec c-b\), and set \(x \df a-y\). Then \(x \in I\) and \(x \prec a
 \prec b\) and \(a-x = y \prec c-b\), as we wished.
\item This time assume \(\delta = \beta \neq \alpha\). In particular, \(b =
 \mu(\delta) \prec \mu(\alpha) = a\). It follows from the previous case that
 there is \(x \in I\) such that \(x \prec b \prec a\) and \(b-x \prec c-a\).
 Since \(b \prec a \prec c\), we conclude from (PoSp2) that \(a-b \prec c-b\).
 Finally, (PoSp3) applied to the case \((a-b) \prec (c-b)\) and \((b-x) \prec
 (c-b)-(a-b)\) yields \((a-x =)\ (a-b)+(b-x) \in I\) and \(a-x \prec c-b\).
\item Now assume \(\delta \neq \zero\) and \(\BooV{\alpha}{\beta} = \gamma\).
 Then \(\eta \df \BooD{\gamma}{\alpha}\) is non-zero and satisfies
 \(\BooW{\eta}{\beta} = \eta \neq \beta\), which implies that \((c-a =)\
 \mu(\eta) \prec \mu(\beta)\ (= b)\). By symmetry, \(c-b \prec a\). Both these
 relations imply that \(d \df a+b-c\) belongs to \(I\) and satisfies \(d \prec
 a\) and \(d \prec b\), by (PoSp2'). Since \(a-d = c-b\) and \(b-d = c-a\)
 belong to \(I\), there is \(y \in I\) such that \(y \prec a-d\) and \(y \prec
 b-d\). Then \(d \prec a\) and \(y \prec a-d\), so (thanks to (PoSp3)) \(x \df
 d+y\) belongs to \(I\) and satisfies \(d \prec x \prec a\). Similarly, \(d
 \prec b\) and \(y \prec b-d\) give \(x \prec b\). Finally, (PoSp2) applied to
 \(d \prec x \prec a\) yields \(a-x \prec a-d = c-b\), and we are done.
\item Finally, assume none of the above cases holds. This means that \(\delta\)
 is non-zero and differs from both \(\alpha\) and \(\beta\), and that
 \(\BooV{\alpha}{\beta} \neq \gamma\). Then \(x \df \mu(\delta) \prec
 \mu(\alpha) = a\) as well as \(x \prec \mu(\beta) = b\). Finally, \(\xi \df
 \BooD{\alpha}{\delta}\) and \(\rho \df \BooD{\gamma}{\beta}\) satisfy
 \(\BooW{\xi}{\rho} = \xi \neq \rho\), which yields \(a-x = \mu(\xi) \prec
 \mu(\rho) = c-b\), as we wished.
\end{itemize}
So, (PoSp4) is fulfilled.\par
Now we pass to (B). So, assume \((J,\prec)\) is a solid pospec and define
\(\Lambda\) and \(E\) as specified therein. Fix for a while \(a \in I\). It
follows from (PoSp4') that there is \(c \in I\) such that \(a \prec c\). Then,
by (PoSp1), \(b \df c-a\) belongs to \(I\). Consequently, \((a,b) \in \Lambda\)
and thus \(\Lambda \neq \varempty\) and \(\Lambda' = I\) (in particular, \(0_G
\notin \Lambda'\), by (PoSp1')). So, it is clear that (GrSp0) is satisfied.
Further, if \((a,b) \in \Lambda\), then clearly \((b,a) \in \Lambda\) as well.
Moreover, we infer from (PoSp1') that both \(\Omega-(a+b)\) and \(\Omega-b\)
belong to \(I\) (as \(b, a+b \in I\)). So, \((a,E-a-b) \in \Lambda\), which
yields (GrSp1).\par
Now assume \((a,b)\) and \((a+b,c) \in \Lambda\). Then \(a, b, c, a+b, a+b+c\)
are all in \(I\) and \(a \prec a+b \prec a+b+c\). So, \(a \prec a+b+c\) and
therefore \((a,b+c) \in \Lambda\), which proves (GrSp2). It remains to check
(GrSp3). To this end, we fix \(a,b,c \in G\) such that both \((a,b),(c,a+b-c)\)
belong to \(\Lambda\), or \(a,c \in \Lambda'\) and \(b = E-a\). Then \(d \df
a+b\) belongs to \(J\), \(b \in I\) and \(a \prec d\) and \(c \prec d\) (these
properties hold even if \(b = E-a\), as then \(d = E = \Omega\), and \(b \in I\)
thanks to (PoSp1')). Now it follows from (PoSp4) that there is \(x \in I\) such
that \(x \prec c\), \(x \prec a\) and \(c-x \prec d-a\). By applying (PoSp1), we
obtain that \(a-x\), \(c-x\) and \((d-a)-(c-x) = b-c+x\) belong to \(I\).
Consequently, \((x,a-x) \in \Lambda\) (as \(x, a-x, x+(a-x) \in I\)) and,
similarly, \((c-x,b-c+x) \in \Lambda\), and we are done.\par
A verification of the additional claim of the theorem is left to the reader.
\end{proof}

\begin{rem}{PN-pospec-measure}
It follows from the above result that for every solid pospec \((J,\prec)\) in
a group \(G\) there exists a solid measure \(\mu\dd \bbB \to G\) such that
\(\mu(\unit) = \Omega\) and \(\Lambda(\mu) = \{(a,b-a)\dd\ a, b \in I,\ a \prec
b\}\). The latter property can be formulated equivalently as follows: \(\mu(\bbB
\setminus \{\zero,\unit\}) = I\) and for any proper element \(a \in \bbB\) and
an element \(g \in I\) there exists a proper element \(b\) of
\(\bbB\restriction{a}\) such that \(\mu(b) = g\) if and only if \(g \prec
\mu(a)\). We call each such a measure \(\mu\) \emph{induced by \((J,\prec)\)}.
\end{rem}

\begin{dfn}{PN-super-pospec}
We call a solid pospec \emph{supersolid} if a measure induced by this pospec is
supersolid.
\end{dfn}

\begin{thm}{PN-super-pospec}
A solid pospec \((J,\prec)\) in a group \(G\) is supersolid iff the following
condition is fulfilled.
\begin{enumerate}[\upshape(PoSp1)]\addtocounter{enumi}{4}
\item For any \(x \in J \setminus \{\varempty\}\) there are a finite number of
 elements \(w_0,\ldots,w_N \in J\) \UP(where \(N > 0\)\UP) such that \(\varempty
 = w_0 \prec w_1 \prec \ldots \prec w_N = \Omega\) and \(w_k-w_{k-1} \prec x\)
 for \(k = 1,\ldots,N\).
\end{enumerate}
\end{thm}
\begin{proof}
Let \((E,\Lambda)\) be a solid \homspec{2} corresponding to \((J,\prec)\) as
stated by \THM{PN-pospec}. Then the pospec is supersolid iff the \homspec{2} is
so. First assume the latter structure is supersolid, and fix arbitrary \(x \in
J \setminus \{\varempty\}\). Since each element of \(J \setminus \{\Omega\}\) is
less that \(\Omega\), we may and do assume that \(x \neq \Omega\). Then \(x \in
\Lambda'\) and it follows from (GrSp4) that there are \(b_1,\ldots,b_N \in G\)
whose sum equals \(E = \Omega\) and (GrSp4a) and (GrSp4b) are satisfied.
The former of these two conditions implies that \(w_k \df \sum_{s=1}^k b_s\
(k=1,\ldots,N)\) belongs to \(J\), and that \(w_{k-1} \prec w_k\) where \(w_0
\df \varempty\) (cf. item (B) of \THM{PN-pospec}); whereas the latter shows that
\(w_k-w_{k-1} = b_k \prec x\), which proves (PoSp5). Conversely, if (PoSp5)
holds and \(x \in \Lambda'\ (\subset J \setminus \{\varempty\})\), then there
are \(w_0,\ldots,w_N \in J\) witnessing (PoSp5). If \(N < 3\), we extend
the system \(w_0,\ldots,w_N\) by taking arbitrary \(z \in J \setminus
\{\varempty\}\) such that \(z \prec w_1\) (thanks to (PoSp4')) and inserting
\(z\) between \(w_0\) and \(w_1\) (we may continue this process as long as we
wish; here we use (PoSp2') to get the desired relation of the form
``\(w_k-w_{k-1} \prec x\)''). So, we may and do assume that \(N \geq 3\). It
follows from (PoSp1) that \(b_k \df w_k-w_{k-1}\) belongs to \(\Lambda'\). Note
that \(\sum_{s=1}^k b_s = w_k\) for any \(k = 1,\ldots,N\). Consequently, all
\(b_k\) sum up to \(\Omega = E\) and (GrSp4a) is fulfilled (as \(w_{k-1} \prec
w_k\)). Finally, (GrSp4b) also holds, since it is covered by the relations
\(w_k-w_{k-1} \prec x\), and we are done.
\end{proof}

\begin{exm}{PN-pospec-good}
A good measure \(\mu\) is induced by a supersolid pospec \(J \df \mu(\bbB)\)
with the natural strict order between reals as `\(\prec\)', when \(0\) is
identified with \(\varempty\).
\end{exm}

In case of linearly ordered solid pospecs, some of axioms (PoSp0)--(PoSp4) can
be simplified, as shown by the next result, where we keep the notation of
\DEF{PN-pospec}.

\begin{pro}{PN-pospec-lin}
Let \(J \subset G^{\circ}\) be equipped with a strict linear order `\(\prec\)'
such that \UP{(PoSp0)--(PoSp1)} and \UP{(PoSp3)} hold. Then \((J,\prec)\) is
a solid pospec iff \(I\) has no least element.
\end{pro}
\begin{proof}
Thanks to (PoSp4'), we only need to show that if \(I\) has no least element,
then \((J,\prec)\) satisfies (PoSp2) and (PoSp4). We start from the remark that
\(0_G \notin I\), by (PoSp1') (which is a consequence of (PoSp1)). Now we will
show (PoSp2'). So, assume \(\varempty \neq a \prec b\). If \(b = \Omega\), then
\(\Omega-a \in I\) (by (PoSp1)) and \(\Omega-a \prec b\) (thanks to (PoSp0)).
So, we may assume \(b \neq \Omega\). We argue by a contradiction. If \(b-a
\not\prec b\), then \(b \prec b-a\) (as \(a \neq 0_G\) and the order is linear).
So, we infer from (PoSp1) that \(-a \in I\). Note that \(b-a \neq -a\) (as \(b
\in I\)) and therefore \(b-a \prec -a\) (because otherwise \(-a \prec b-a\),
which combined with \(a \prec b\) would yield that \(a-a = 0_G \in I\), by
(PoSp3)). So, it follows (again from (PoSp1)) that \(-b \in I\). Moreover, as
\(-a \neq -b\), we must have \(-a \prec -b\) (because otherwise \(-b \prec -a\),
which combined with \(b \prec (-a)-(-b)\) would again yield that \(0_G \in I\),
thanks to (PoSp3)). Consequently, \(a-b \in I\). Moreover, \(-b \prec a-b\)
(because otherwise \(a-b \prec -b\), which combined with \(b-a \prec
(-b)-(a-b)\) would give \(0_G \in I\)). Finally, we get \(a \prec b \prec b-a
\prec -a \prec -b \prec a-b\) and hence \(a \prec a-b\) which combined with \(-a
\prec (a-b)-a\) gives \(0_G \in I\), which is a contradiction. Thus \(b-a \prec
b\), as we have claimed.\par
According to the last paragraph of \REM{PN-pospec}, to verify (PoSp2) we only
need to check its former claim. To this end, assume \(a \prec b \prec c\) and
note that then \(a \neq b\). If \(a = \varempty\), then our conclusion is
covered by (PoSp2') (that has already been shown). So, we assume that \(a \in
I\). If \(c-b \not\prec c-a\), then (by linearity of the order) \(c-a \prec
c-b\) (note also that both \(c-a\) and \(c-b\) belong to \(I\), by (PoSp1)). But
then \(a-b = (c-b)-(c-a)\) belongs to \(I\) and
\begin{equation}\label{eqn:PN-aux48}
a-b \prec c-b,
\end{equation}
by (PoSp2'). Since \(b \in I\), we infer that \(a \prec a-b\) (because otherwise
\(a-b \prec a = b-(b-a)\) and also \(b-a \prec b\), by (PoSp2'), which would
yield that \(0_G = (a-b)+(b-a)\) belongs to \(I\), thanks to (PoSp3)).
Consequently, \(-b \in I\) and, moreover, \(-b \prec a-b\), by (PoSp2').
The last relation, combined with \eqref{eqn:PN-aux48}, yields that \(c \neq
0_G\) (which was possible as \(c\) can coincide with \(\Omega\)). But then \(c-b
\prec -b\) (as otherwise \(-b \prec c-b\) which combined with \(b \prec c\)
would give, thanks to (PoSp3), \(0_G \in I\)), which, together with the previous
``\(-b \prec a-b\)'', leads to \(c-b \prec a-b\), contradictory to
\eqref{eqn:PN-aux48}. So, \(c-b \prec c-a\), and thus (PoSp2) is fulfilled.\par
Now we pass to (PoSp4). So, assume \(a, b \in I\) and \(c \in J\) are such that
\(a \prec c\) and \(b \prec c\). First assume \(a \prec b\) or \(a = b\), and
take arbitrary \(y \in I\) such that \(y \prec a\) and \(y \prec c-b\) (recall
that \(c-b \in I\) and the order is linear and \(I\) has no least element). Set
\(x \df a-y\) and observe that \(a \in I\) (by (PoSp1)) and \(x \prec a\)
(thanks to (PoSp2')). In particular, \(x \prec b\) as well. Finally, \(a-x = y
\prec c-b\), which finishes the proof of (PoSp4) in that case.\par
It remains to consider the case when \(b \prec a\). Then take arbitrary \(y \in
I\) such that \(y \prec b\) and \(y \prec c-a\). Set \(x \df b-y\) and note that
(similarly as in the preious case) \(x \in I\) and \(x \prec b\). In particular,
\(x \prec a\). Finally, since \(b \prec a \prec c\), we infer from (PoSp2) that
\(a-b \prec c-b\). But also \(b-x = y \prec (c-b)-(a-b)\). So, we may apply
(PoSp3) to get that \((a-x =)\ (a-b)+(b-x) \prec c-b\), and we are done.
\end{proof}

With the aid of the above result, we can now give many examples of solid pospecs
that are linearly ordered.

\begin{dfn}{PN-G-order}
A strict linear order `\(<\)' on a group \(G\) is said to be \emph{compatible}
(with \(G\)) if \(a < b\) implies \(a+c < b+c\) (for all \(a,b,c \in G\)). We
call an element \(a\) of \(G\) \emph{positive} if \(0_G < a\).\par
A group \(G\) is said to be \emph{densely ordered by `\(<\)'} if `\(<\)' is
a compatible linear order on \(G\) such that for any \(a > 0_G\) there is \(b >
0_G\) satisfying \(b < a\).
\end{dfn}

It may easily be shown that each group admitting a compatible linear order is
torsion-free. A little bit harder is to show the converse (that is, that any
torsion-free group admits a compatible linear order---consult the example given
below).\par
A construction introduced in the example below is common in Lie algebra theory
(when working with simple roots of a simple Lie algebra).

\begin{exm}{PN-lin-ord}
Let \(V\) be a non-trivial vector group over \(\QQQ\). Fix any (Hamel) basis
\(\{e_t\}_{t \in T}\) of \(V\) and a linear order `\(\prec\)' on the set \(T\)
of indices. So, for any vector \(v \in V\) there exists a unique finitely
supported function \(\alpha_v\dd T \to \QQQ\) such that \(v = \sum_{t \in T}
\alpha_v(t) e_t\). If \(v\) is non-zero, we will denote by \(\qQ(v)\) the first
(w.r.t. `\(\prec\)') non-zero coefficient of \(v\); that is, \(\qQ(v) =
\alpha_v(s)\) where \(s = \min_{\prec}\{t \in T\dd\ \alpha_v(t) \neq 0\}\).
Finally, we define a strict order `\(<\)' on \(V\) as follows:
\[u < w \iff u \neq w\ \land\ \qQ(u-w) < 0 \qquad (u,w \in V).\]
It is straightforward to check that `\(<\)' is a compatible strict linear order
on \(V\). Moreover, \(V\) is densely ordered by `\(<\)' (as there is no least
positive rational number).\par
When \(V = \QQQ^d\) (where \(d\) is finite), we may take as the above basis
the canonical one and as the order on the indices \(\{1,\ldots,d\}\) the natural
order. Then the compatible linear order on \(\QQQ^d\) introduced above coincides
with the so-called lexicographic order; that is, a non-zero vector \(v = (v_1,
\ldots,v_d)\) is positive iff \(v_j > 0\) where \(j \in \{1,\ldots,d\}\) is
the first index such that \(v_j \neq 0\). In this case, we call this order
\emph{canonical}.\par
Finally, as each torsion-free group admits a one-to-one homomorphism into
a vector space over \(\QQQ\), we see that each such a group admits a compatible
linear order. Note also that a countable torsion-free group can be densely
ordered iff it is non-isomorphic to \(\ZZZ\). (The assumption of countability of
the group may be skipped, but the proof in a general case is more subtle. As in
this paper we deal only with countable groups, we omit the proof for the general
case.) To convince oneself about that, note that each such a (countable!) group
\(G\) admits a one-to-one homomorphism \(\phi\) into the reals and the image of
\(\phi\) is a dense subgroup of \(\RRR\), provided \(G\) is non-cyclic.
\end{exm}

\begin{pro}{PN-lin-ord}
\begin{enumerate}[\upshape(A)]
\item Let `\(<\)' be a compatible linear order on a countable group \(G\) and
 let \(\Omega\) be a positive element of \(G\). Set \(J \df \{\varempty,\Omega\}
 \cup I\) where \(I\) consists of all positive elements \(g \in G\) such that
 \(g < \Omega\). Let `\(\prec\)' be a strict linear order on \(J\) induced by
 `\(<\)'; that is, for two distinct elements \(x, y \in J\), \(x \prec y\) iff
 \(x = \varempty\) or both these elements differ from \(\varempty\) and satisfy
 \(x < y\). Then \((J,\prec)\) is a solid pospec iff \(G\) is densely ordered by
 `\(<\)'.
\item If \(G = \QQQ^d\) \UP(where \(d > 1\) is finite\UP), `\(<\)' is
 the canonical order on \(G\) and \(\Omega = (1,0,\ldots,0) \in G\), then
 the poset \((J,\prec)\) defined in this case in a way specified in \UP{(A)} is
 a solid pospec that is not supersolid.
\end{enumerate}
\end{pro}
\begin{proof}
We start from (A). A verification that `\(\prec\)' is a strict linear order on
\(J\) as well as that axioms (PoSp0)--(PoSp3) hold (with a single exception that
\(J\) may not be infinite) is left to the reader. So, according to
\PRO{PN-pospec-lin}, \((J,\prec)\) is a solid pospec iff \(I\) has no least
element (note that the latter condition implies that \(J\) is infinite). But
\(I\) has no least element iff \(G\) is densely ordered by `\(<\)', which proves
(A). To show (B), recall that \(\QQQ^d\) is densely ordered by the canonical
order, and hence \((J,\prec)\) (defined in (B)) is a solid pospec (by (A)). To
convince oneself that it is not supersolid, it is enough (thanks to
\THM{PN-super-pospec}) to explain why (PoSp5) does not hold. But this is
immediate, as \(x \df (0,\ldots,0,1)\) belongs to \(I\) and any vector \(v\)
from \(I\) such that \(v \prec x\) has first coordinate equal to \(0\). So, if
\(\varempty = w_0 \prec w_1 \prec \ldots \prec w_N\) satisfy \(w_k-w_{k-1} \prec
x\) (cf. (PoSp5) in \THM{PN-super-pospec}), then the first coordinate of \(w_N\)
is \(0\) and thus \(w_N \neq \Omega\).
\end{proof}

\begin{dfn}{PN-good-like}
Let \(G\) be a countable group that is densely ordered by `\(<\)', let
\(\Omega\) be a positive element in \(G\), and let \((J,\prec)\) be a solid
pospec defined in this case in a way specified in item (A) of \PRO{PN-lin-ord}.
A homogeneous measure induced by \((J,\prec)\) is called a \emph{good-like}
\(G\)-valued measure.
\end{dfn}

\begin{rem}{PN-lin-ord-super}
Instead of giving an `abstract' example of a supersolid linearly ordered pospec,
we inform the reader that for each such a pospec \((J,\prec)\) (in a group
\(G\)) a homogeneous measure \(\mu\) induced by \((J,\prec)\) is equivalent to
a good measure in the following sense: there exist a (unique up to conjugacy)
good measure \(\rho\) on \(\Can\) and a unique homomorphism \(\phi\) from
the subgroup of \(\RRR\) generated by \(V \df \rho(\bbB)\) into \(G\) such that
\(\mu(a) = (\phi \circ \rho)(a)\) for any \(a \in \bbB\), and \(\rho\) is
one-to-one on \(V \setminus \{0\}\). So, roughly speaking, supersolid good-like
measures are precisely good measures. This result is a topic of our subsequent
paper.
\end{rem}

\section{Universal invariant measures}

In this section we identify \(\bbB\) with the algebra of all clopen subsets of
\(\Can\), and for any group-valued measure \(\mu\) on \(\bbB\) we identify
\(\OPN{Aut}(\mu)\) with a subgroup of the homeomorphism group
\(\OPN{Homeo}(\Can)\) of \(\Can\).

\begin{dfn}{PN-univ-meas}
Let \(H\) be any subset of \(\OPN{Homeo}(\Can)\). A group-valued measure
\(\mu\dd \bbB \to G\) is said to be a \emph{universal \(H\)-invariant
group-valued measure} if both the following conditions are met:
\begin{itemize}
\item \(\mu\) is \(H\)-invariant; that is, \(H \subset \OPN{Aut}(\mu)\) or,
 equivalently, each \(h \in H\) preserves \(\mu\)---that is, \(\mu(h[a]) =
 \mu(a)\) for any \(a \in \bbB\) (where \(h[a]\) denotes the image of \(a\) via
 \(h\)); and
\item for any \(H\)-invariant group-valued measure \(\rho\dd \bbB \to Q\) there
 exists a unique homomorphism \(u\dd G \to Q\) such that \(\rho(a) = (u \circ
 \mu)(a)\) for all \(a \in \bbB\).
\end{itemize}
(In particular, if \(\mu\) is a universal \(H\)-invariant group-valued measure,
then for any group \(Q\) there is a one-to-one correspondence between
\(Q\)-valued \(H\)-invariant measures and homomorphisms from \(G\) into
\(Q\).)\par
Similarly, we call a probability Borel measure \(\lambda\) on \(\Can\)
a \emph{universal \(H\)-invariant probability measure} if
\begin{itemize}
\item \(\mu\) is \(H\)-invariant; and, denoting by \(W_{\lambda}\)
 the \(\QQQ\)-linear span of \(\lambda(\bbB)\) in \(\RRR\),
\item for any \(H\)-invariant probability Borel measure \(\nu\) on \(\Can\)
 there exists a unique \(\QQQ\)-linear operator \(v\dd W_{\lambda} \to \RRR\)
 such that \(\nu(a) = (v \circ \lambda)(a)\) for all \(a \in \bbB\).
\end{itemize}
(In particular, if \(\lambda\) is a universal \(H\)-invariant probability
measure, then there is a one-to-one correspondence between \(H\)-invariant
probability measures and \(\QQQ\)-linear operators from \(W_{\lambda}\) into
\(\RRR\) that are non-negative on \(\lambda(\bbB)\) and fix \(1\).)\par
It is worth underlying here that the property of being universal is not
an `intrinsic' property of a group-valued measure; that is, if \(\mu\dd \bbB \to
G\) is such a measure and \(G \subsetneq G'\) (as a subgroup), then \(\mu\),
considered as a \(G'\)-valued measure is \textbf{not} universal. (So, formally
we should define universal measures as pairs \((\mu,G)\).)\par
It is also worth noticing that if \(\mu\) is a probability Borel measure on
the \Cantor\ and \(G\) stands for the subgroup of \(\RRR\) generated by
\(\mu(\bbB)\), then \(\mu\) is a universal \(H\)-invariant probability measure
unless it is a universal \(H\)-invariant group-valued measure (as each
homomorphism from a subgroup of \(\RRR\) into \(\RRR\) admits a unique extension
to a \(\QQQ\)-linear operator from the \(\QQQ\)-linear span of this subgroup
into \(\RRR\)).
\end{dfn}

We are mostly interested in cases when universal measures defined above are
homogeneous. However, we begin with

\begin{thm}{PN-univ-exist}
Let \(H\) be a subset of \(\OPN{Homeo}(\Can)\).
\begin{enumerate}[\upshape(A)]
\item There exists a universal \(H\)-invariant group-valued measure.
\item If \(\mu\dd \bbB \to G\) is a universal \(H\)-invariant group-valued
 measure, then a group-valued measure \(\nu\dd \bbB \to Q\) is a universal
 \(H\)-invariant group-valued measure iff there exists an isomorphism \(\Phi\dd
 G \to Q\) such that \(\nu = \Phi \circ \mu\). \UP(Moreover, if \(\Phi\) exists,
 it is unique.\UP) In particular, \(\OPN{Aut}(\mu) = \OPN{Aut}(\rho)\) for any
 universal \(H\)-invariant group-valued measure \(\rho\).
\item There exists a universal \(H\)-invariant probability measure iff there
 exists an \(H\)-invariant probability Borel measure on \(\Can\).
\item If \(\mu\) is a universal \(H\)-invariant probability measure and
 \(W_{\mu}\) is the \(\QQQ\)-linear span of \(\mu(\bbB)\), then a probability
 Borel measure \(\nu\) on the \Cantor\ is a universal \(H\)-invariant
 probability measure iff there exists a \(\QQQ\)-linear isomorphism \(\Phi\dd
 W_{\mu} \to W_{\nu}\) such that \(\nu(a) = (\Phi \circ \mu)(a)\) for all \(a
 \in \bbB\). \UP(Moreover, if \(\Phi\) exists, it is unique.\UP) In particular,
 \(\OPN{Aut}(\mu) = \OPN{Aut}(\rho)\) for any universal \(H\)-invariant
 probability measure \(\rho\).
\end{enumerate}
\end{thm}

As we will see, the `group-valued' part of \THM{PN-univ-exist} is an easy
observation.\par
We precede the proof of the above result by the following lemma, in which
\(\ell_{\infty}\) denotes the Banach space of all bounded real-valued sequences
whose indices start from 1. This result is surely well-known. For the reader's
convenience, we give its short proof.

\begin{lem}{PN-l-infty}
Let \(Q\) be a countable subset of \(\ell_{\infty}\). There exists a sequence
\(a_1,a_2,\ldots\) of positive real numbers such that \(\sum_{n=1}^{\infty} a_n
= 1\) and the function
\begin{equation}\label{eqn:PN-aux49}
Q \ni (x_n)_{n=1}^{\infty} \mapsto \sum_{n=1}^{\infty} a_n x_n \in \RRR
\end{equation}
is one-to-one.
\end{lem}
\begin{proof}
Replacing \(Q\) by its \(\QQQ\)-linear span, we may and do assume that \(Q\) is
a non-trivial countable \(\QQQ\)-linear subspace of \(\ell_{\infty}\).\par
Consider the space \(X \df \prod_{n=1}^{\infty} [3^{-n},2^{-n}]\) with
the product topology. It is a compact metrisable space and for every \(z =
(z_n)_{n=1}^{\infty} \in \ell_{\infty}\) the function \(\sigma_z\dd X \to \RRR\)
given by \(\sigma_z(x) = \sum_{n=1}^{\infty} z_n x_n\) is continuous where \(x =
(x_n)_{n=1}^{\infty} \in X\). Now for any non-zero \(q \in Q\) let \(U_q\) stand
for the set of all \(x \in X\) such that \(\sigma_q(x) \neq 0\). It is clear
that \(U_q\) is an open subset of \(X\). It is also dense in \(X\), as if
\(\sigma_q(x) = 0\) for some \(x = (x_n)_{n=1}^{\infty} \in X\) and \(k > 0\) is
an index such that \(q_k \neq 0\), then \(x_k = -\frac{1}{q_k} \sum_{n \neq k}
q_n x_n\) and we can approximate \(x\) by elements of \(U_q\) by changing only
the \(k\)th coordinate of \(x\).\par
Now it follows from the Baire's category theorem that there is
\((b_n)_{n=1}^{\infty} \in X\) that belongs to
\(\bigcap_{q \in Q \setminus \{0\}} U_q\). Then setting \(B \df
\sum_{n=1}^{\infty} b_n\) and \(a_n \df \frac{b_n}{B}\), we obtain the sequence
we sarched for (as for this specific sequence the kernel of \eqref{eqn:PN-aux49}
is trivial).
\end{proof}

\begin{proof}[Proof of \THM{PN-univ-exist}]
We start from (A). As there are at most \(2^{\aleph_0}\) (at most) countable
Abelian groups, there exists a \underline{set} \(\Gamma =
\{\mu_t\dd \bbB \to G_t\}_{t \in T}\) of group-valued \(H\)-invariant measures
such that for any \(t \in T\), \(G_t\) is generated (as a group) by
\(\mu_t(\bbB)\) and
\begin{itemize}
\item[\((\dag)\)] for any group-valued \(H\)-invariant measure \(\nu\dd \bbB \to
 Q\) there exists \(t \in T\) and a homomorphism \(\xi\dd G_t \to Q\) such that
 \(\nu = \xi \circ \mu_t\).
\end{itemize}
The set \(T\) is non-empty as the measure constantly equal to \(0_{\RRR}\) is
\(H\)-invariant. Now define \(\mu\dd \bbB \to \tilde{G} \df \prod_{t \in T}
G_t\) as the diagonal of \(\Gamma\); that is, \(\mu(a) = (\mu_t(a))_{t \in T}\).
Let \(G\) be the subgroup of \(\tilde{G}\) generated by \(\mu(\bbB)\). It is now
easy to see that \(\mu\dd \bbB \to G\) is a universal \(H\)-invariant
group-valued measure---to show the existence of a suitable homomorphism (to
represent a given \(H\)-invariant measure) one uses projections \(\pi_s\dd G \ni
(x_t)_{t \in T} \mapsto x_s \in G_s\ (s \in T)\) and \((\dag)\), and its
uniqueness follows from the definition of \(G\). The details are left to
the reader.\par
We now pass to a more subtle item (C). We only need to show the `if' part.
Denote by \(\Delta\) the set of all \(H\)-invariant probability Borel measures
on \(\Can\). By our assumption in (C), \(\Delta\) is non-empty. Recall that
the weak* topology of the set \(P\) of all probability Borel measures on
\(\Can\) is a metrisable topology in which measures \(\mu_1,\mu_2,\ldots\) from
\(P\) tend to \(\rho \in P\) iff \(\lim_{n\to\infty} \mu_n(a) = \rho(a)\) for
any \(a \in \bbB\). In this topology \(P\) is compact and \(\Delta\) is closed.
In particular, \(\Delta\) is separable. So, fix a sequence \(\mu_1,\mu_2,
\ldots\) of measures from \(\Delta\) that form a dense set in \(\Delta\). Now
set \(M\dd \bbB \ni a \mapsto (\mu_n(a))_{n=1}^{\infty} \in \ell_{\infty}\) and
denote by \(Q\) the \(\QQQ\)-linear span of \(M(\bbB)\) in \(\ell_{\infty}\).
Further, denote by \(\Lambda\) the set of all \(\QQQ\)-linear operators \(F\dd
Q \to \RRR\) such that \(F(1,1,1,\ldots) = 1\) and \(F(M(\bbB)) \subset [0,1]\).
Since \(Q\) is generated (as a vector space over \(\QQQ\)) by \(M(\bbB)\), we
infer that \(\Lambda\) is compact when equipped with the pointwise convergence
topology. Observe that \(M\) is \(H\)-invariant and hence \(F \circ M\dd \bbB
\to [0,1]\) is a non-negative \(H\)-invariant (abstract) measure that sends
\(\unit\) to \(1\). As such, it uniquely extends to a probability Borel measure
from \(\Delta\). In this way we obtain a continuous function \(\Psi\dd \Lambda
\to \Delta\) (that assigns to each \(F\) from \(\Lambda\) a unique measure
\(\nu\) from \(\Delta\) such that \(\nu(a) = (F \circ M)(a)\) for all \(a \in
\bbB\)). Note that all projections \(\pi_m\dd Q \ni (x_n)_{n=1}^{\infty} \mapsto
x_m \in \RRR\) are in \(\Lambda\) and therefore \(\mu_m = \Psi(\pi_m) \in
\Psi(\Lambda)\). So, we conclude that \(\Psi(\Lambda) = \Delta\). Finally, we
apply \LEM{PN-l-infty} to obtain a sequence \(a_1,a_2,\ldots\) of positive reals
that sum up to \(1\) and for which the function \(\Theta\dd Q \to \RRR\) given
by \eqref{eqn:PN-aux49} is one-to-one. Note that \(\Theta \in \Lambda\) and,
consequently, \(\mu \df \Psi(\Theta)\) is an \(H\)-invariant probability Borel
measure on \(\Can\). We can now easily show that it is universal: if \(\nu \in
\Delta\), then there is \(F \in \Lambda\) such that \(\nu = \Psi(F)\), which
means that \(\nu(a) = F(M(a)) = (F \circ \Theta^{-1})(\Theta(M(a))) = (F \circ
\Theta^{-1})(\mu(a))\) for all \(a \in \bbB\) (and the domain of \(\Theta^{-1}\)
coincides with the \(\QQQ\)-linear span of \(\mu(\bbB)\)). As in
the group-valued case, uniqueness of \(u = F \circ \Theta^{-1}\) is
automatic.\par
As parts (B) and (D) are straightforward, we leave their proofs to the reader.
\end{proof}

The following is a direct consequence of parts (B) and (D) of
\THM{PN-univ-exist} (cf. \LEM{PN-1-1}). We skip its simple proof.

\begin{cor}{PN-univ-homo}
A given universal \(H\)-invariant group-valued \UP(resp. probability\UP) measure
is homogeneous iff all such measures are homogeneous.
\end{cor}

It is straightforward to check that the measure \(\nu\dd \bbB \to C(\Can,\ZZZ)\)
that assigns to a clopen set its characteristic function is a group-valued
homogeneous measure such that any group-valued measure on \(\bbB\) has the form
\(\phi \circ \nu\) where \(\phi\) is a homomorphism from \(C(\Can,\ZZZ)\) into
a group. A less trivial counterpart of this observation for probability Borel
measures on the \Cantor\ is formulated below.

\begin{cor}{PN-univ-meas}
There exists a homogeneous probability Borel measure \(\mu\) on \(\Can\) such
that all probability Borel measures on \(\Can\) are precisely unique extensions
\UP(to measures defined the \(\sigma\)-algebra of all Borel subsets of
the \Cantor\UP) of all set functions of the form \(\phi \circ
(\mu\restriction{\bbB})\) where \(\phi\) is a \(\QQQ\)-linear operator from
the \(\QQQ\)-linear span of \(\mu(\bbB)\) into \(\RRR\) that fixes \(1\) and is
non-negative on \(\mu(\bbB)\).
\end{cor}
\begin{proof}
Take as \(\mu\) a universal \(H\)-invariant probability measure where \(H\) is
the trivial group. It follows from \EXM{PN-homo-bad} and \COR{PN-univ-homo} that
\(\mu\) is homogeneous, whereas the other property of \(\mu\) (formulated in
the result) follows from its universality.
\end{proof}

\begin{rem}{PN-intersec-aut-gr}
The arguments involved in the proofs of parts (A) and (C) of \THM{PN-univ-exist}
can be used to show the following interesting properties:
\begin{itemize}
\item Let \(\aaA\) consist of all subgroups of \(\OPN{Homeo}(\Can)\) of the form
 \(\OPN{Aut}(\mu)\) where \(\mu\) is a group-valued measure on \(\bbB\) (resp.
 where \(\mu\) is a probability Borel measure on \(\Can\)). Then for any
 non-empty collection \(\{G_s\}_{s \in S}\) of groups from \(\aaA\) its
 intersection \(\bigcap_{s \in S} G_s\) belongs to \(\aaA\) as well.
\item For any \(G \in \aaA\) there exists a universal \(G\)-invariant
 group-valued (resp. probability) measure \(\mu\) and it satisfies
 \(\OPN{Aut}(\mu) = G\).
\end{itemize}
We leave the details to interested readers.
\end{rem}

\begin{rem}{PN-signed-univ}
If \(H\) is a subset of \(\OPN{Homeo}(\Can)\) such that there exists a non-zero
signed Borel measure \(\nu\) on \(\Can\) that is \(H\)-invariant, then there
exists a universal \(H\)-invariant probability measure \(\mu\) and for any
\(H\)-invariant signed Borel measure \(\rho\) there exists a unique
\(\QQQ\)-linear operator \(u\dd W_{\mu} \to \RRR\) such that \(\rho(a) =
(u \circ \mu)(a)\) for all \(a \in \bbB\). To convince oneself of that, note
that if \(\nu\) is \(H\)-invariant, then the variation \(|\nu|\) of \(\nu\) is
\(H\)-invariant as well (as \(H\) consists of homeomorphisms). So,
\(\frac{|\nu|}{|\nu|(\Can)}\) is a probability Borel measure that is
\(H\)-invariant. Consequently, the above \(\mu\) exists. Moreover, \(\nu\) is
a linear combination of two \(H\)-invariant probability Borel measures and hence
it is of the form \(u \circ \mu\) on \(\bbB\) for some \(\QQQ\)-linear \(u\).
\end{rem}

\begin{pro}{PN-univ-homo}
Let \(H\) be the automorphism group of a homogeneous group-valued measure
\(\mu\) on \(\bbB\) \UP(resp. of a homogeneous probability Borel measure \(\mu\)
on \(\Can\)\UP). Then a universal \(H\)-invariant group-valued measure \UP(resp.
a universal \(H\)-invariant probability measure\UP) \(\lambda\) is homogeneous
as well and satisfies \(\OPN{Aut}(\lambda) = H\).
\end{pro}
\begin{proof}
It follows from \THM{PN-univ-exist} that \(\lambda\) exists. So, \(H \subset
\OPN{Aut}(\lambda)\) (by the very definition of a universal measure). Moreover,
it follows from the universal property of \(\lambda\) that
\begin{equation}\label{eqn:PN-aux50}
\mu(a) = (\phi \circ \lambda)(a) \qquad (a \in \bbB)
\end{equation}
for some homomorphism defined on a respective group, which implies that
\(\OPN{Aut}(\lambda) \subset \OPN{Aut}(\mu)\ (= H)\) and therefore
\(\OPN{Aut}(\lambda) = H\). Finally, a partial \(\lambda\)-isomorphism defined
on a finite Boolean subalgebra of \(\bbB\) is a partial \(\mu\)-isomorphism as
well (thanks to \eqref{eqn:PN-aux50}) and hence it extends to a member of \(H\)
(by homogeneity of \(\mu\)), which shows that \(\lambda\) is homogeneous.
\end{proof}

\begin{rem}{PN-univ-homo}
Continuing the proof of \PRO{PN-univ-homo}, if \(\phi\) satisfies
\eqref{eqn:PN-aux50}, then it is one-to-one on \(\lambda(\bbB \setminus \{\zero,
\unit\})\) as otherwise we would get two proper elements \(a\) and \(b\) of
\(\bbB\) such that \(\lambda(a) \neq \lambda(b)\) and \(\mu(a) = \mu(b)\), and
then homogeneity of \(\mu\) would give us an automorphism \(h \in H\) that sends
\(a\) onto \(b\) and therefore \(h \notin \OPN{Aut}(\lambda)\).\par
Conversely, if \(\lambda\dd \bbB \to G\) is a universal measure (in the sense of
\DEF{PN-univ-meas}) and \(\phi\dd G \to Q\) is a homomorphism that is one-to-one
on \(\lambda(\bbB \setminus \{\zero,\unit\})\), then the measure \(\mu\) given
by \eqref{eqn:PN-aux50} is homogeneous and satisfies \(\OPN{Aut}(\mu) =
\OPN{Aut}(\lambda)\).
\end{rem}

To formulate our next result (which is one of the main goals of the paper), we
introduce the following notation (cf. \cite{IM16}). For a minimal homeomorphism
\(h\) of the \Cantor, we use \([h]\) to denote the so-called \emph{full group of
\(h\)} consisting of all \(g \in \OPN{Homeo}(\Can)\) such that for any \(x \in
\Can\) there is an integer \(k \in \ZZZ\) satisfying \(g(x) = h^k(x)\). Below
\(\overline{[h]}\) stands for the closure of \([h]\) in the uniform topology of
\(\OPN{Homeo}(\Can)\).

\begin{thm}{PN-homeo-mini}
Let \(h\dd \Can \to \Can\) be a minimal homeomorphism. Then a universal
\(h\)-invariant probability measure \(\mu\) is supersolid. Moreover,
\(\OPN{Aut}(\mu) = \overline{[h]}\); and \(\mu\) is good if and only if \(h\) is
uniquely ergodic.
\end{thm}

In the proof, we will apply two well-known results gathered in a single result
below (see also \cite{IM16} for more details):

\begin{thm}{PN-G-W}
Let \(\Delta\) be the set of all \(h\)-invariant probability Borel measures on
\(\Can\) where \(h\) is a minimal homeomorphism on the \Cantor.
\begin{enumerate}[\upshape(A)]
\item \cite{GW95} For two clopen sets \(a\) and \(b\) of \(\bbB\) \tfcae
 \begin{itemize}
 \item \(\mu(a) = \mu(b)\) for all \(\mu \in \Delta\);
 \item \(g[a] = b\) for some \(g \in [h]\).
 \end{itemize}
\item \cite{GPS99} The group \(\overline{[h]}\) consists precisely of all
 homeomorphisms of \(\Can\) that preserve each measure from \(\Delta\).
\end{enumerate}
\end{thm}

\begin{proof}[Proof of \THM{PN-homeo-mini}]
Let \(\Delta\) be as specified in the statement of \THM{PN-G-W}. Since
\(\Delta\) is non-empty, we know from \THM{PN-univ-exist} that \(\mu\) exists.
It follows from item (B) of \THM{PN-G-W} that
\begin{equation}\label{eqn:PN-aux51}
\overline{[h]} = \bigcap_{\rho\in\Delta} \OPN{Aut}(\rho).
\end{equation}
We know from the proof of part (C) of \THM{PN-univ-exist} that
\(\mu\restriction{\bbB}\) is of the form \(\phi \circ M\) where
\begin{enumerate}[(u1)]
\item \(M\dd \bbB \to \ell_{\infty}\) has the form \(M(a) =
 (\rho_n(a))_{n=1}^{\infty}\ (a \in \bbB)\) where
\item \(\rho_1,\rho_2,\ldots\) are measures from \(\Delta\) that form a dense
 set in \(\Delta\) (in the weak* topology);
\item \(\phi\dd Q \to \RRR\) is a one-to-one function and \(Q \supset M(\bbB)\).
\end{enumerate}
We infer from (u3) that \(\OPN{Aut}(\mu) = \OPN{Aut}(M)\) and that \(\mu\) is
homogeneous iff \(M\) is so; and from (u1) that \(\OPN{Aut}(M) =
\bigcap_{n=1}^{\infty} \OPN{Aut}(\rho_n)\). Finally, (u2) combined with
\eqref{eqn:PN-aux51} yield \(\bigcap_{n=1}^{\infty} \OPN{Aut}(\rho_n) =
\overline{[h]}\). All these remarks give \(\OPN{Aut}(\mu) = \overline{[h]}\).
We now briefly explain why \(M\) is homogeneous (and therefore \(\mu\) is
homogeneous as well). To this end, fix arbitrary proper elements \(a, b \in
\bbB\) such that \(M(a) = M(b)\). Then \(\rho_n(a) = \rho_n(b)\) for each \(n\)
(by (u1)) and, consequently, \(\nu(a) = \nu(b)\) for all \(\nu \in \Delta\), by
(u2). So, part (A) of \THM{PN-G-W} yields a homeomorphism \(g \in [h]\ (\subset
\OPN{Aut}(M))\) such that \(g(a) = b\), which implies homogeneity of \(M\).\par
Now it is clear that \(\mu\) is supersolid (as \(h \in \OPN{Aut}(\mu)\) is
minimal). So, it remains to verify when \(\mu\) is a good measure. A one
direction is simple: if \(h\) is uniquely ergodic, then \(\mu\) is a unique
\(h\)-invariant probability measure, which is known to be good by a result due
to Glasner and Weiss \cite{GW95}. Conversely, if \(\mu\) is good, then it is not
difficult to show that the identity is a unique real-valued \(\QQQ\)-linear
operator defined on the \(\QQQ\)-linear span of \(\mu(\bbB)\) that fixes \(1\)
(as then \(\mu(\bbB)+\ZZZ\) is a subgroup of \(\RRR\)). So, it follows from
the universal property of \(\mu\) that \(\Delta = \{\mu\}\) and hence \(h\) is
uniquely ergodic.
\end{proof}

As minimal homeomorphisms induce supersolid probability measures, a natural
question arises:

\begin{prb}{PN-which-super-mini} (cf. \PRO{PN-univ-homo})
Assume \(\mu\) is a supersolid probability measure on \(\Can\). Find sufficient
or necessary conditions for the existence of a minimal homeomorphism \(h\) on
the \Cantor\ such that \(\OPN{Aut}(\mu) = \overline{[h]}\).
\end{prb}

Now we briefly discuss the question of when a homogeneous measure is universal
w.r.t. its own automorphism group.

\begin{pro}{PN-homo-when-univ}
Let \(\mu\) be a homogeneous measure and let \(H \df \OPN{Aut}(\mu)\), \(V_* \df
\mu(\bbB \setminus \{\zero,\unit\})\) and \(V \df V_* \cup \{\mu(\zero),
\mu(\unit)\}\).
\begin{enumerate}[\upshape(A)]
\item If \(\mu\dd \bbB \to G\) is group-valued and either \(\mu(\unit) = 0_G \in
 V_*\) or \(\mu(\unit) \neq 0_G\), then \(\mu\) is a universal \(H\)-invariant
 group-valued measure iff \(G\) is generated by \(V\), and for any group \(Q\),
 each function \(\xi\dd \mu(\bbB) \to Q\) such that \(\xi \circ \mu\) is
 a measure extends to a homomorphisn from \(G\) into \(Q\).
\item If \(\mu\) is a probability Borel measure on \(\Can\), then it is
 a universal \(H\)-invariant probability measure iff each function \(\xi\dd V
 \to [0,1]\) such that \(\xi(1) = 1\) and \(\xi \circ (\mu\restriction{\bbB})\)
 is a measure \UP(on \(\bbB\)\UP) extends to a \(\QQQ\)-linear operator from
 the \(\QQQ\)-linear span of \(V\) into \(\RRR\).
\end{enumerate}
\end{pro}
\begin{proof}
First assume \(\nu\dd \bbB \to Q\) is an \(H\)-invariant measure (where \(Q\) is
an arbitrary group---the argument presented below works also for part (B)). For
any two proper elements \(a, b \in \bbB\) such that \(\mu(a) = \mu(b)\) there
exists \(\phi \in H\) such that \(\phi(a) = b\) (as \(\mu\) is homogeneous and
takes values in a group), and then \(\nu(a) = \nu(\phi(a)) = \nu(b)\). This
shows that
\[\mu(a) = \mu(b) \implies \nu(a) = \nu(b) \qquad (a,b \in \bbB \setminus
\{\zero,\unit\}).\]
So, there exists a function \(\xi\dd \mu(\bbB \setminus \{\zero,\unit\}) \to Q\)
such that \(\nu(a) = \xi(\mu(a))\) for all \(a\) different from \(\zero\) and
\(\unit\). For simplicity, we set \(\zZ \df \mu(\zero)\) and \(E \df
\mu(\unit)\). Observe that if \(E = \zZ\), then, according to the assumptions in
(A), \(\zZ \in V_*\). So, there are only two possibilities:
\begin{itemize}
\item \(\zZ \notin V_*\) and \(E \neq \zZ\); then also \(E \notin V_*\) (as \(E
 = \mu(a)+\mu(\BooD{\unit}{a})\)); and in that case we may simply extend \(\xi\)
 to a function \(\xi\dd V \to Q\) by assigning \(0_Q\) to \(\zZ\) and
 \(\nu(\unit)\) to \(E\)---then \(\nu = \xi \circ \mu\) (on \(\bbB\)).
\item \(\zZ \in V_*\); then also \(E \in V_*\), and we infer from
 \LEM{PN-zero-2} that there are two disjoint proper elements \(a\) and \(b\)
 such that \(\BooV{a}{b} \neq \unit\) and \(\mu(a) = \mu(b) = \zZ\); then
 \(\xi(\zZ) = \xi(\mu(\BooV{a}{b})) = \nu(\BooV{a}{b}) = \nu(a)+\nu(b) =
 \xi(\mu(a))+\xi(\mu(b)) = \xi(\zZ)+\xi(\zZ)\), which implies that \(\xi(\zZ) =
 0_H\); similarly, \(\xi(E) = \xi(\mu(\BooD{\unit}{a})) = \nu(\BooD{\unit}{a}) =
 \nu(\unit)-\nu(a) = \nu(\unit)\), and therefore \(\nu = \xi \circ \mu\).
\end{itemize}
In this way we have shows that in both cases (A) and (B) each group-valued
\(H\)-invariant measure \(\nu\dd \bbB \to Q\) has the form \(\nu = \xi \circ
(\mu\restriction{\bbB})\) for some \(\xi\dd V \to Q\). Observe also that if
\(\mu\) from part (A) is universal, then \(V\) generates \(G\) (which
follows e.g. from the construction presented in the proof of part (A) of
\THM{PN-univ-exist}). Now the whole assertion (of both (A) and (B)) simply
follows.
\end{proof}

\begin{rem}{PN-univ-non-zero}
In our subsequent paper we will show that if a group-valued homogeneous measure
\(\mu\) is universal with respect to its automorphism group and satisfies \(0_G
\notin \mu(\bbB \setminus \{\zero,\unit\})\), then \(\mu(\unit) \neq 0_G\); and
therefore item (A) of the above result exhausts all possibilities.
\end{rem}

\begin{cor}{PN-good-are-univ}
Let \(H\) denote the automorphism group of a group-valued measure \(\mu\dd \bbB
\to G\). If \(\mu\) is good-like \UP(w.r.t. some linear order `\(<\)' on
\(G\)\UP) and \(G\) is generated by \(\mu(\bbB)\), or \(\mu\) is \(G\)-filling,
then it is a universal \(H\)-invariant group-valued measure.
\end{cor}
\begin{proof}
\PRO{PN-homo-when-univ} applies to both the classes of measures specified in
the result. So, we assume \(\xi\dd \mu(\bbB) \to Q\) is a function such that
\(\nu \df \xi \circ \mu\) is a measure. In particular, \(\xi(0_G) = 0_Q\).\par
First assume \(\mu\) is \(G\)-filling. We only need to show that \(\xi\) is
a homomorphism, which is a simple observation: given abitrary elements \(g\) and
\(h\) of \(G\) there are two disjoint proper elements \(a\) and \(b\) of
\(\bbB\) such that \(\mu(a) = g\) and \(\mu(b) = h\); and then \(\xi(g+h) =
\xi(\mu(\BooV{a}{b})) = \nu(\BooV{a}{b}) = \nu(a)+\nu(b) = \xi(\mu(a))+
\xi(\mu(b)) = \xi(g)+\xi(h)\).\par
Now assume \(\mu\) is good-like and \(J \df \mu(\bbB)\) generates \(G\).
A similar argument as presented in the previous paragraph shows that \(\xi(g+h)
= \xi(g)+\xi(h)\) unless \(g, h \in J\) are such that \(g+h \in J\). Set
\(\Omega \df \mu(\unit)\), \(I \df J \setminus \{\Omega\}\) and note that any
element \(g \in G\) can uniquely be expressed in the form \(x = k\Omega+r\)
where \(k \in \ZZZ\) and \(r \in I\). (Indeed, the set of all \(x \in G\) such
that \(-n\Omega < x < n\Omega\) for some integer \(n > 0\) is a subgroup of
\(G\) that contains \(J\); and thus for any \(x \in G\) there exists a unique
integer \(k\) such that \(k\Omega \leq x\), then \(x-k\Omega \in I\).) So, we
define an extension \(\bar{\xi}\dd G \to Q\) of \(\xi\) by
\(\bar{\xi}(k\Omega+r) \df k \xi(\Omega)+\xi(r)\ (k \in \ZZZ,\ r \in I)\). It
remains to show that \(\bar{\xi}\) is a homomorphism. To this end, fix any \(k,
\ell \in \ZZZ\) and \(p, q \in I\). If \(p+q < \Omega\), then
\(\bar{\xi}((k\Omega+p)+(\ell\Omega+q)) = (k+\ell)\xi(\Omega)+(\xi(p)+\xi(q)) =
\bar{\xi}(k\Omega+p)+\bar{\xi}(\ell\Omega+q)\). Finally, we assume \(p+q \geq
\Omega\) (in particular, \(q \neq 0_G\)). Then \(p+q < 2\Omega\) and therefore
\(p+q-\Omega \in I\). Consequently, \(\bar{\xi}((k\Omega+p)+(\ell\Omega+q)) =
\bar{\xi}((k+\ell+1)\Omega+(p+q-\Omega)) =
(k+\ell+1)\xi(\Omega)+\xi(p+q-\Omega)\). So, we only need to show that
\(\xi(p)+\xi(q) = \xi(\Omega)+\xi(p+q-\Omega)\). Since \(\Omega-q \in J\) and
\(p = (\Omega-q)+(p+q-\Omega)\), we get \(\xi(p) = \xi(\Omega-q)+
\xi(p+q-\Omega)\). Similarly, \(\xi(\Omega) = \xi(q)+\xi(\Omega-q)\), and
the conclusion follows.
\end{proof}

\section{Dense conjugacy classes in automorphism groups}

Here we adapt classical concepts present in studies of the automorphism groups
of good measures (or in a more general context) to the context of
semigroup-valued homogeneous measures (consult, e.g.,
\cite[Sections~4--5]{DKMN1}), and generalise some of the theorems obtained in
the cited article. Since most of the proofs proceed in the same way as in
the classical case, some of the results of this section are presented without
proofs. Specifically, we omit the proofs of all the results up to
\THM{PN-dense-conj} below.\par
As previously, \(S\) will denote a semigroup and \(G\) is reserved to denote
a group that is at most countable. Additionally, \(\lambda\) is a fixed
\(S\)-valued (or \(G\)-valued if it is explicitly stated) homogeneous measure on
\(\bbB\). We follow the conception introduced in
Section~\ref{sec:group}---specifically, if \(Q\) is a (semi)group, then
\(\varempty \notin Q\), \(Q^{\circ} = Q \cup \{\varempty\}\) and \(g+\varempty =
\varempty+g = g\) for any \(g \in Q^{\circ}\). Further, any \(Q\)-valued measure
\(\mu\) on \(\bbB\) will be identified with a \(Q^{\circ}\)-valued measure
\(\mu^{\circ}\) on \(\bbB\) given by \(\mu^{\circ}(\zero) = \varempty\) and
\(\mu^{\circ}(a) = \mu(a)\) for all non-zero \(a \in \bbB\) (note that the class
of all partial \(\mu\)-isomorphisms coincides with the class of all partial
\(\mu^{\circ}\)-isomorphisms; in particular, \(\OPN{Aut}(\mu) =
\OPN{Aut}(\mu^{\circ})\)). So, whenever (in this section) we will speak about
(semi)group-valued measures, we will always understand them as taking the value
\(\varempty\) at \(\zero\). In particular, \(\lambda(\zero) = \varempty\).

\begin{dfn}{PN-matrix}
Given two square matrices \(A = [a_{\alpha\beta}]_{\alpha,\beta \in I}\) and
\(B = [b_{\gamma\delta}]_{\gamma,\delta \in J}\) with entries in \(S^{\circ}\)
(where \(I\) and \(J\) are two arbitrary finite non-empty sets of indices), we
will write \(A \equiv B\) to express that there is a bijection \(\sigma\dd
I \to J\) such that \(a_{\alpha\beta} = b_{\sigma(\alpha)\sigma(\beta)}\) for
all \(\alpha,\beta \in I\). In particular, if \(A \equiv B\), then \(A\) and
\(B\) have the same dimension. If \(A \equiv B\), we call \(A\) and \(B\)
\emph{equivalent}. We call the entry \(a_{\alpha\beta}\) of \(A\) \emph{empty}
if \(a_{\alpha\beta} = \varempty\), otherwise it is \emph{non-empty}.
\end{dfn}

The following result is a direct consequence of \THM{PN-4EP}.

\begin{lem}{PN-conjugate-partial}
For \(j=1,2\) let \(\aaA_j\) and \(\phi_j\dd \aaA_j \to \bbB\) be, respectively,
a finite Boolean subalgebra of \(\bbB\) and a partial \(\lambda\)-isomorphism.
Further, let \(\{a_1,\ldots,a_N\}\) and \(\{b_1,\ldots,b_M\}\) be partitions of
\(\bbB\) that generate \(\aaA_1\) and \(\aaA_2\), respectively. Set \(A \df
[\lambda(\BooW{a_j}{\phi_1(a_k)})]_{j,k=1}^N\) and, similarly, \(B \df
[\lambda(\BooW{b_j}{\phi_2(b_k)})]_{j,k=1}^M\). \TFCAE
\begin{enumerate}[\upshape(i)]
\item there exists \(\psi \in \OPN{Aut}(\lambda)\) such that \(\psi(\aaA_1) =
 \aaA_2\) and \(\psi \circ \phi_1 = \phi_2 \circ (\psi\restriction{\aaA_1})\);
\item \(A \equiv B\).
\end{enumerate}
More specifically, \(M = N\) and the assignments \(\BooW{a_j}{\phi_1(a_k)}
\mapsto \BooW{b_j}{\phi_2(b_k)}\ (j,k=1,\ldots,N)\) extend to \(\psi \in
\OPN{Aut}(\lambda)\) \UP(and then automatically \UP{(i)} holds with this
\(\psi\)\UP) iff \(A = B\).
\end{lem}

\begin{dfn}{PN-Gamma(f)}
Let \(\aaA\) be a finite Boolean subalgebra of \(\bbB\) generated by a partition
\(\{a_j\dd\ j \in J\}\) of \(\bbB\). For any partial \(\lambda\)-isomorphism
\(\phi\dd \aaA \to \bbB\) we set \(\Gamma(\phi) \df
[\lambda(\BooW{a_j}{\phi(a_k)})]_{j,k \in J}\). Note that \(\Gamma(\phi)\) does
depend on the way we index the unique partition that generates \(\aaA\)
(however, no matter how we do it, we obtain equivalent matrices).
\end{dfn}

\begin{lem}{PN-when-Gamma(f)}
Let \(\Gamma = [\gamma_{\alpha\beta}]_{\alpha,\beta \in I}\) be a matrix with
entries in \(S^{\circ}\) \UP(where \(I\) is a finite non-empty set\UP). Assume
\(s_1,\ldots,s_p \in S\) are such that each element of \(S\) appears among all
\(s_k\) as many times as it appears among all \(\gamma_{\alpha\beta}\). Then
\tfcae
\begin{enumerate}[\upshape(i)]
\item there exists a partial \(\lambda\)-isomorphism \(\phi\) defined on
 a finite Boolean subalgebra of \(\bbB\) such that \(\Gamma(\phi) \equiv
 \Gamma\);
\item \((s_1,\ldots,s_p) \in \Delta_p(\lambda)\) \UP(cf. \EXM{PN-delta-k}\UP),
 and \(\sum_{\beta \in I} \gamma_{\alpha\beta} = \sum_{\beta \in I}
 \gamma_{\beta\alpha} \neq \varempty\) for any \(\alpha \in I\).
\end{enumerate}
\end{lem}

\begin{dfn}{PN-compatible}
We call a square matrix \(\Gamma = [\gamma_{\alpha\beta}]_{\alpha,\beta \in I}\)
(with entries in \(S^{\circ}\)) \emph{\compatible{\lambda}} if it has all
the properties stated in item (ii) of \LEM{PN-when-Gamma(f)}.\par
Let \(A = [a_{uv}]_{u,v \in I}\) and \(B = [b_{pq}]_{p,q \in J}\) be two
\(S^{\circ}\)-valued matrices. A \emph{\morphism} from \(A\) into \(B\) is any
function \(f\dd J \to I\) such that
\begin{equation}\label{eqn:PN-morph}
a_{uv} = \sum_{(p,q) \in F_{uv}} b_{pq}
\end{equation}
for any \(u, v \in I\) where \(F_{uv} \df \{(p,q) \in J \times J\dd\ f(p) = u,\
f(q) = v\}\) and \(\sum_{(p,q) \in \varempty} b_{pq} \df \varempty\). It is
worth noticing here that each \morphism{} is a surjection between the respective
sets of indices.
\end{dfn}

\begin{lem}{PN-is-arrow}
Let \(\phi\dd \aaA \to \bbB\) be a partial \(\lambda\)-isomorphism defined on
a finite Boolean subalgebra \(\aaA\) generated by a partition \(\{a_1,\ldots,
a_N\}\) of \(\bbB\).
\begin{enumerate}[\upshape(A)]
\item  Let \(\{b_1,\ldots,b_M\}\) be a partition of \(\bbB\) contained in
 \(\aaA\). For each \(j \in \{1,\ldots,N\}\) let \(f(j)\) denote the unique
 index \(k \in \{1,\ldots,M\}\) such that \(\BooW{a_j}{b_k} = a_j\). Then
 \(f\dd \{1,\ldots,N\} \to \{1,\ldots,M\}\) is a \morphism{} from
 \(\Gamma(\phi\restriction{\aaA'})\) into \(\Gamma(\phi)\) where \(\aaA'\) is
 the Boolean subalgebra of \(\bbB\) generated by \(\{b_1,\ldots,b_M\}\) \UP(and
 both the \compatible{\lambda} matrices appearing here are computed with respect
 to the current indexing of the partitions\UP).
\item If \(g\dd I \to \{1,\ldots,N\}\) is a \morphism{} from \(\Gamma(\phi)\)
 into a \compatible{\lambda} matrix \(\Gamma =
 [\gamma_{\alpha\beta}]_{\alpha,\beta \in I}\), then there are a partition
 \(\{c_{\alpha}\dd \alpha \in I\}\) of \(\bbB\) that generates a Boolean algebra
 \(\ddD \supset \aaA\) and a partial \(\lambda\)-isomorphism \(\psi\dd \ddD \to
 \bbB\) such that \(\psi\) extends \(\phi\), \(\Gamma(\psi) = \Gamma\) and
 \(\BooW{a_{g(\alpha)}}{c_{\alpha}} = a_{g(\alpha)}\) for any \(\alpha \in I\)
 \UP(so, \(g\) is ``the function \(f\) from \UP{(A)}'' constructed for \(\psi\)
 and \(\psi\restriction{\aaA}\)\UP).
\end{enumerate}
\end{lem}

As a corollary from the above result, we obtain the following counterpart of
\cite[Theorem~4.8]{DKMN1}.

\begin{thm}{PN-dense-conj}
The group \(\OPN{Aut}(\lambda)\) has a dense conjugacy class iff for any two
\compatible{\lambda} matrices \(\Gamma_1\) and \(\Gamma_2\) there exists
a \compatible{\lambda} matrix \(\Gamma_0\) such that both \(\Gamma_1\) and
\(\Gamma_2\) admit a \morphism{} into \(\Gamma_0\).
\end{thm}

\begin{dfn}{PN-dense-conj}
We say that \(\lambda\) has the \emph{Rokhlin property} if
\(\OPN{Aut}(\lambda)\) contains a dense conjugacy class. If
\(\lambda\restriction{j}\) has the Rokhlin property for each non-zero \(j \in
\bbB\), we say that \(\lambda\) has the \emph{hereditary Rokhlin property}.
\end{dfn}

\begin{dfn}{equi-measured}
Let \(N\) be a positive integer. We say a non-zero element \(j \in \bbB\)
\emph{admits an equi-measured \(N\)-partition} (w.r.t. \(\lambda\)) if there
exists a partition \(\{b_1,\ldots,b_N\}\) of \(\bbB\restriction{j}\) (consisting
of precisely \(N\) elements) such that \(\lambda(b_k) = \lambda(b_1)\) for each
\(k=2,\ldots,N\). In the above situation we call \(c \df \lambda(b_1)\)
the \emph{common measure of the partition} and we say that \(\{b_1,\ldots,
b_N\}\) is an \emph{equi-measured partition of \(j\) with common measure \(c\)}.
\end{dfn}

The following two results are natural generalisations of
\cite[Lemmas~5.4 and 5.2]{DKMN1} (respectively).

\begin{lem}{PN-1/N}
Assume \(\lambda\) has the Rokhlin property and \(\unit\) admits
an equi-measured \(N\)-partition \(\{c_1,\ldots,c_N\}\) \UP(for some with \(N >
1\)\UP). Then each proper \(j \in \bbB\) admits an equi-measured \(N\)-partition
as well.
\end{lem}
\begin{proof}
Let \(h \in \OPN{Aut}(\lambda)\) have dense conjugacy class, from which we infer
that we may and do assume \(h(j) = j\). It follows from ultrahomogeneity of
\(\lambda\) (cf. \THM{PN-4EP}) that there is \(g \in \OPN{Aut}(\lambda)\) such
that \(g(c_N) = c_1\) and \(g(c_k) = c_{k+1}\) for \(k=1,\ldots,N-1\). Since
\(h\) has dense conjugacy class, we conclude that there is a partition \(\{a_1,
\ldots,a_N\}\) of \(\bbB\) such that \(h(a_N) = a_1\) and \(h(a_k) = a_{k+1}\)
for \(k < N\). But then \(b_k \df \BooW{a_k}{j}\ (k=1,\ldots,N)\) are pairwise
disjoint and satisfy \(h(b_N) = b_1\) and \(h(b_k) = b_{k+1}\) for \(k < N\).
Consequently, all \(b_k\)'s are non-zero (and hence they form a partition of
\(\bbB\restriction{j}\)) and \(\lambda(b_k) = \lambda(b_1)\) for each \(k\).
\end{proof}

\begin{lem}{PN-ring}
Let \(I\) and \(J\) stand for, respectively, \(\lambda(\bbB \setminus \{\zero,
\unit\})\) and \(\lambda(\bbB)\). Assume that there exists a one-to-one function
\(h\dd J \to R\) into a \UP(not necessarily commutative\UP) ring \((R,+,\cdot)\)
with unit such that all the following conditions are fulfilled:
\begin{enumerate}[\upshape(h1)]
\item \(h(0_S) = 0_R\), and \(h(\lambda(\unit))\) is invertible in \(R\);
\item \(h(I) \cdot h(\lambda(\unit))^{-1} \cdot h(I) \subset h(I)\);
\item if \(a, b \in I\) are such that \(a+b \in J\), then \(h(a+b) =
 h(a)+h(b)\);
\item for any \(s_1,\ldots,s_N \in I\) \UP(where \(N > 0\)\UP) such that
 \(\sum_{k=1}^N h(s_k) = h(\lambda(\unit))\) there exists a partition \(\{c_1,
 \ldots,c_N\}\) of \(\bbB\) satisfying \(\lambda(c_k) = s_k\) for \(k=1,\ldots,
 N\).
\end{enumerate}
Then \(\lambda\) has the Rokhlin property.
\end{lem}
\begin{proof}
It follows from (h3) that \(\mu \df h \circ \lambda\) is an \(R\)-valued
measure. Moreover, since \(h\) is one-to-one, we infer that the class of all
partial \(\mu\)-isomorphisms coincides with the one of all partial
\(\lambda\)-isomorphisms. In particular, \(\mu\) is homogeneous as well and
\(\OPN{Aut}(\mu) = \OPN{Aut}(\lambda)\). Thus, we may and do assume that \(\mu =
\lambda\) (and \(h\) is the identity map). Then \(E \df \lambda(\unit)\) is
invertible in \(R\) and \(I \cdot E^{-1} \cdot I \subset I\). Moreover, (h4)
says that if \(s_1,\ldots,s_N \in I\) are such that \(\sum_{k=1}^N s_k = E\),
then there is a partition \(\{c_1,\ldots,c_N\}\) of \(\bbB\) satisfying
\(\lambda(c_k) = s_k\) for all \(k\).\par
Fix two \compatible{\lambda} matrices \(A = [a_{jk}]_{j,k=1}^m\) and \(B =
[b_{pq}]_{p,q=1}^n\). Thanks to \THM{PN-dense-conj}, it is sufficient to show
that both \(A\) and \(B\) admit \morphism{}s into a common \compatible{\lambda}
matrix. We may and do assume that both \(n\) and \(m\) are greater than \(1\).
Then all non-empty entries of both these matrices are from \(I\). We set \(T \df
\{1,\ldots,m\} \times \{1,\ldots,n\}\) and for any \(s = (j,p)\) and \(t =
(k,q)\) from \(T\) we define
\[c_{st} = \begin{cases}
\varempty & \UP{if } a_{jk} = \varempty \UP{ or } b_{pq} = \varempty\\
a_{jk} E^{-1} b_{pq} & \UP{otherwise}
\end{cases}.\]
Since \(\sum_{j,k} a_{jk} = E = \sum_{p,q} b_{pq}\), one easily checks that \(C
= [c_{st}]_{s,t \in T}\) is \compatible{\lambda} (thanks to (h4)). Now we define
\(f\dd T \to \{1,\ldots,m\}\) and \(g\dd T \to \{1,\ldots,n\}\) by \(f(j,p) \df
j\) and \(g(j,p) = p\). As there is no difficulty in checking that \(f\) and
\(g\) are \morphism{}s from \(C\) into \(A\) and \(B\), respectively, the proof
is finished.
\end{proof}

Recall that a set \(I \subset [0,1]\) is \emph{group-like} if \(\{0,1\} \subset
I\) and \(I+\ZZZ\) is an additive subgroup of \(\RRR\).

\begin{thm}{PN-ring-vector}
Let \(I_1,\ldots,I_N\) be countably infinite group-like subsets of \([0,1]\)
such that \(I_k \cdot I_k \subset I_k\) for each \(k\), and let \(\Sigma\)
consist of all triples \((x,y,z)\) of vectors \(x, y, z \in (0,1)^N \cap (I_1
\times \ldots \times I_N)\) such that \(x+y+z = (1,\ldots,1)\). Let \(\mu\dd
\bbB \to \RRR^N\) be a homogeneous measure such that \(\Sigma(\nu) = \Sigma\).
\begin{enumerate}[\upshape(A)]
\item The measure \(\mu\) is supersolid and has the Rokhlin property.
\item If, in addition, \(\frac{ab}{c} \in I_k\) for all \(a, b, c \in I_k\) such
 that \(a < c\) and \(b < c\), and any \(k = 1,2,\ldots,N\), then \(\mu\) has
 the hereditary Rokhlin property.
\end{enumerate}
\end{thm}
\begin{proof}
First of all, note that \(\Sigma\) is a supersolid \homspec{3} (as it is
the skew product of \(N\) such \homspec{3}s: \(\{(a,b,c) \in (0,1)^3\dd\ a,b,c
\in I_k,\ a+b+c = 1\}\), cf. \EXM{PN-good-super}). This explains why \(\mu\)
exists (and is supersolid). Now consider \(\RRR^N\) as a ring with
the coordinatewise product (that is, \((x_1,\ldots,x_N) \cdot (y_1,\ldots,y_N)
\df (x_1 y_1,\ldots,x_N y_N)\)). Substituting the identity map for \(h\), it is
easily seen that all the conditions (h1)--(h4) of \LEM{PN-ring} are satisfied,
which shows (A). To prove (B), fix non-zero \(j \in \bbB\), write \(\mu =
(\mu_1,\ldots,\mu_N)\) (where \(\mu_k\dd \bbB \to I_k\) is a certain good
measure) and note that \(\Sigma(\mu\restriction{j}) = \bar{\oplus}_k
\Sigma(\mu_k\restriction{j})\). Again, we may apply the same lemma (with
\(\lambda \df \mu\restriction{j}\)) to get the conclusion of (B) (notice that
the assumption of (B) is nothing else than (h2) of the cited lemma).
\end{proof}

Repeating the argument presented in \REM{PN-uncount} and in the proof of
\THM{PN-super-no-good} (and using the fact that there are uncountably many
countable subfields of \(\RRR\)), one obtains the following consequence of
the above result (we skip the proof).

\begin{cor}{PN-super-nogood-dense}
There exist uncountably many probability Borel measures on the \Cantor\ that are
supersolid, have the hereditary Rokhlin property and whose restrictions to
\(\bbB\restriction{j}\) \UP(where \(j\) runs over all non-zero elements of
\(\bbB\)\UP) are all non-good.\par
Similarly, there exist uncountably many \UP(essentially\UP) signed Borel
measures on the \Cantor\ that are supersolid and have the hereditary Rokhlin
property.
\end{cor}

In the next results we study the question of when a \(G\)-filling \(E\)-measure
(cf. \DEF{PN-G-filling}) has the Rokhlin property. We recall that each such
a measure is supersolid. We begin with

\begin{thm}{PN-Z-dense}
Let \(d \geq 1\) and \(E = (k_1,\ldots,k_d) \in \ZZZ^d\). A \(\ZZZ^d\)-filling
\(E\)-measure \(\mu\) has the Rokhlin property iff \(E \neq 0\) and
\(\OPN{GCD}(k_1,\ldots,k_d) = 1\).
\end{thm}
\begin{proof}
If \(E = 0\) or \(\OPN{GCD}(k_1,\ldots,k_d) > 1\), then \(\mu\) fails to have
the Rokhlin property, which simply follows from \LEM{PN-1/N} (as an element
\(a\) for which \(\mu(a) = (1,0,\ldots,0)\) cannot be partioned into parts of
equal measure, whereas \(\mu(E)\) can). On the other hand, if \(\OPN{GCD}(k_1,
\ldots,k_d) = 1\), then there exists a group automorphism \(\phi\dd \ZZZ^d \to
\ZZZ^d\) such that \(\phi(E) = (1,\ldots,1)\) and then \LEM{PN-ring} applies (as
\(\ZZZ^d\) is a unital ring, with pointwise actions).
\end{proof}

\begin{lem}{PN-1/N-uniq}
Assume \(\lambda\) has the Rokhlin property and \(\unit\) admits two
equi-measured \(N\)-partitions \(\{a_1,\ldots,a_N\}\) and \(\{b_1,\ldots,b_N\}\)
with common measures, respectively, \(\alpha\) and \(\beta\) \UP(where \(N >
1\)\UP). Then \(\alpha = \beta\).
\end{lem}
\begin{proof}
Let \(h \in \OPN{Aut}(\lambda)\) have a dense conjugacy class. A reasoning
presented in the proof of \LEM{PN-1/N} shows that we may and do assume (after
replacing the given partitions by new equi-measured ones with the same common
measures as assumed) that \(h(a_N) = a_1\), \(h(a_k) = a_{k+1}\), \(h(b_N) =
b_1\) and \(h(b_k) = b_{k+1}\) for any \(k < N\). Then \(\alpha = \lambda(a_1) =
\sum_{k=1}^N \lambda(\BooW{a_1}{b_k}) = \sum_{k=1}^N \lambda(\BooW{a_1}{h^{k-1}
(b_1)}) = \sum_{k=1}^N \lambda(\BooW{h^{1-k}(a_1)}{b_1}) = \sum_{k=1}^N
\lambda(\BooW{h^{N-k+1}(a_k)}{b_1}) = \lambda(\BooW{h^N(a_1)}{b_1}) +
\sum_{s=1}^{N-1} \lambda(\BooW{h^s(a_1)}{b_1}) = \lambda(\BooW{a_1}{b_1}) +
\sum_{s=2}^N \lambda(\BooW{a_s}{b_1}) = \lambda(b_1) = \beta\), and we are done.
\end{proof}

\begin{cor}{PN-equal-part-uniq}
Assume \(\lambda\) is group-valued and has the Rokhlin property. If \(\unit\)
admits an equi-measured \(N\)-partition \UP(for some \(N > 1\)\UP), then for any
non-zero element \(j \in \bbB\) all equi-measured \(N\)-partitions of \(j\) have
the same common measure.
\end{cor}
\begin{proof}
If \(j = \unit\), the assertion follows from the previous result. So, assume
\(j\) is proper and take two equi-measured \(N\)-partitions \(\{a_1,\ldots,
a_N\}\) and \(\{b_1,\ldots,b_N\}\) of \(j\). It follows from \LEM{PN-1/N} that
there exists an equi-measured \(N\)-partition \(\{c_1,\ldots,c_N\}\) of
\(\BooD{\unit}{j}\). Then \(\{\BooV{a_1}{c_1},\ldots,\BooV{a_N}{c_N}\}\) and
\(\{\BooV{b_1}{c_1},\ldots,\BooV{b_N}{c_N}\}\) are two equi-measured
\(N\)-partitions of \(\unit\). So, we infer from \LEM{PN-1/N-uniq} that
\((\lambda(a_1)+\lambda(c_1) =)\ \lambda(\BooV{a_1}{c_1}) =
\lambda(\BooV{b_1}{c_1})\ (= \lambda(b_1)+\lambda(c_1))\) and thus
\(\lambda(a_1) = \lambda(b_1)\) as well.
\end{proof}

\begin{thm}{PN-0-filling}
\begin{enumerate}[\upshape(A)]
\item If a \(G\)-filling \(0_G\)-measure has the Rokhlin property, then \(G\) is
 a vector space over \(\QQQ\).
\item \TFCAE
 \begin{itemize}
 \item for each non-zero \(E \in G\), a \(G\)-filling \(E\)-measure has
  the Rokhlin property;
 \item \(G\) is a vector space over \(\QQQ\) or over \(\ZZZ_p \df \ZZZ / p\ZZZ\)
  for some prime \(p\).
 \end{itemize}
\end{enumerate}
\end{thm}
\begin{proof}
Everywhere below \(\mu_E\) is a \(G\)-filling \(E\)-measure on \(\bbB\). Observe
that if \(E = N g\) where \(N > 1\) and \(g \in G\), then there exists
a partition \(\{a_1,\ldots,a_N\}\) of \(\bbB\) such that \(\mu(a_k) = g\) for
each \(k = 1,\ldots,N\).\par
Now if \(E = 0_G\), it follows from the above remark and \LEM{PN-1/N-uniq} that
\(G\) is torsion-free, and from \LEM{PN-1/N} that \(G\) is divisible. So, \(G\)
is a vector space over \(\QQQ\).\par
Now assume that each \(\mu_E\) has the Rokhlin property for \(E \neq 0_G\).
First assume there exists an element \(g \in G\) that has infinite order. Then
it follows from the first paragraph of the proof, applied to \(E = Ng\)
(separately for each \(N > 1\)), combined with \LEM{PN-1/N} that \(G\) is
divisible. On the other hand, \LEM{PN-1/N-uniq} implies that the equation
\(Ng = Nx\) (with unknown \(x\)) has a unique solution. Consequently, \(G\) is
torsion-free.\par
Now let \(G\) be torsion and take a non-zero element \(h\) of \(G\) whose order
is a prime \(p\). Then \LEM{PN-1/N-uniq} implies that \(px = 0_G\) for each \(x
\in G\) (because otherwise \(\mu_E\) with \(E = px\) will fail to have
the Rokhlin property, since \(p(x+h) = px\); cf. again the remark formulated in
the first paragraph of the proof). But this is equivalent to the statement that
\(G\) is a vector space over \(\ZZZ_p\).\par
It remains to show that if \(G \neq \{0_G\}\) is a vector space over one of
the fields \(\KKK\) specified in (B), then all \(\mu_E\) (with \(E \neq 0_G\))
have the Rokhlin property. As we will see, it is an almost immediate consequence
of \LEM{PN-ring}. First of all, being a vector space, \(G\) has the following
property: for any two non-zero \(x\) and \(y\) in \(G\) there exists
an automorphism \(\phi\dd G \to G\) such that \(\phi(x) = y\). Then \(\phi \circ
\mu_x = \mu_y\) and thus our assertion holds iff at least one of the measures
\(\mu_E\) has the Rokhlin property. But this is a consequence of the cited lemma
as \(G\) can be endowed with a unital ring structure. Indeed, if \(G\) is
finite-dimensional (as a vector space over \(\KKK\)), it is isomorphic to
\(\KKK^d\) for some finite \(d\) and can be endowed with a multiplication
defined coordinatewise; and if \(G\) is infinite-dimensional, it is isomorphic
(as a vector space over \(\KKK\)) to the ring \(\KKK[X]\) of all polynomials.
\end{proof}

\begin{thm}{PN-Q-vector}
Let \(G\) be a countable vector space over \(\QQQ\). Then each \(G\)-filling
measure has the hereditary Rokhlin property.
\end{thm}
\begin{proof}
For each \(E \in G\) denote by \(\mu_E\) a \(G\)-filling \(E\)-measure on
\(\bbB\). It follows from \THM{PN-0-filling} that \(\mu_E\) has the Rokhlin
property for each \(E \neq 0\). And since for each non-zero \(j \in \bbB\),
\(\mu_E\restriction{j}\) is a \(G\)-filling \(E'\)-measure (on
\(\bbB\restriction{j}\)) with \(E' \df \mu_E(j)\), we infer that the conclusion
of the theorem is equivalent to the statement that \(\lambda \df \mu_{0_G}\) has
the Rokhlin property. The proof that \(\lambda\) has this property is given
below.\par
Fix two \compatible{\lambda} matrices \(A = [a_{jk}]_{j,k=1}^m\) and \(B =
[b_{pq}]_{p,q=1}^n\). Similarly as argued in the proof of \LEM{PN-ring}, it is
sufficient to show that both \(A\) and \(B\) admit \morphism{}s into a common
\compatible{\lambda} matrix. To this end, observe that there is a finite system
\(\nu_0,\ldots,\nu_M\) of indices from \(\{1,\ldots,m\}\) such that \(M > 0\),
\(\nu_1,\ldots,\nu_M\) are all different, \(\nu_0 = \nu_M\), and
\(a_{\nu_{k-1}\nu_k} \neq \varempty\) for any \(k=1,\ldots,M\). (Indeed, start
from arbitrary \(s_0,s_1\) such that \(a_{s_0s_1}\) is non-empty and then
inductively add new indices \(s_k\) so that \(a_{s_{k-1}s_k}\) is non-empty,
until \(s_k\) appears among \(s_0,\ldots,s_{k-1}\), say \(s_k = s_j\) with \(j <
k\). Then set \(\nu_0,\ldots,\nu_M\) as \(s_j,\ldots,s_k\). Note: it may happen
that \(M = 1\).) Similarly, there are indices \(\eta_0,\ldots,\eta_N\) among
\(\{1,\ldots,n\}\) such that \(N > 0\), \(\eta_1,\ldots,\eta_N\) are all
different, \(\eta_0 = \eta_N\), and \(b_{\eta_{q-1}\eta_q} \neq \varempty\) for
\(q = 1,\ldots,N\). To simplify further statements, for a tuple \((z_0,\ldots,
z_K)\) of indices and two indices \(u\) and \(v\) we will write:
\begin{itemize}
\item \(u \in (z_0,\ldots,z_K)\) to express that \(u = z_j\) for some \(j=0,
 \ldots,K\);
\item \((u,v) \subset (z_0,\ldots,z_K)\) to express that \(u = z_{j-1}\) and \(v
 = z_j\) for some \(j \in \{1,\ldots,K\}\).
\end{itemize}
Now we define \([\delta_{jk}]_{j,k=1}^m\) and \([\epsi_{pq}]_{p,q=1}^n\) as
folows:
\[\delta_{jk} \df \begin{cases}1 & \UP{if } (j,k) \subset (\nu_0,\ldots,\nu_M)\\
0 & \UP{otherwise}\end{cases}, \qquad
\epsi_{pq} \df \begin{cases}1 & \UP{if } (p,q) \subset (\eta_0,\ldots,\eta_N)\\
0 & \UP{otherwise}\end{cases}.\]
Now we set \(T \df \{1,\ldots,m\} \times \{1,\ldots,n\}\) and
\(C = [c_{st}]_{s,t \in T}\) where for \(s = (j,p)\) and \(t = (k,q)\):
\[c_{st} \df \begin{cases}\varempty & \UP{if } a_{jk} = \varempty \UP{ or }
b_{pq} = \varempty\\\frac1N\epsi_{pq}a_{jk}+\frac1M\delta_{jk}b_{pq} &
\UP{otherwise}\end{cases}.\]
Finally, as in the proof of \LEM{PN-ring}, we define \(f\dd T \to \{1,\ldots,
m\}\) and \(g\dd T \to \{1,\ldots,n\}\) by \(f(j,p) = j\) and \(g(j,p) = p\). It
remains to check that \(C\) is \compatible{\lambda} and that \(f\) and \(g\) are
\morphism{}s from \(C\) into \(A\) and \(B\), respectively. To this end, observe
that \(\sum_{j,k} \delta_{jk} = M\) and for each \(j \in \{1,\ldots,m\}\):
\[\sum_{k=1}^m \delta_{jk} = \begin{cases}1 & \UP{if } j \in (\nu_0,\ldots,
\nu_M)\\0 & \UP{otherwise}\end{cases} = \sum_{k=1}^m \delta_{kj},\]
and, similarly, \(\sum_{p,q} \epsi_{pq} = N\) and \(\sum_{q=1}^n \epsi_{pq} =
\sum_{q=1}^n \epsi_{qp}\) for \(p = 1,\ldots,n\). Now all the properties
postulated by us easily follow, and we are done.
\end{proof}

It is not a coincidence that among `group-filling' measures only those whose
underlying groups are vector spaces over \(\QQQ\) have the hereditary Rokhlin
property. Actually, a more general phenomenon occurs, as shown by

\begin{pro}{PN-hered-1/N}
Let \(\mu\dd \bbB \to S\) be a semigroup-valued measure that is supersolid and
has the hereditary Rokhlin property. Then for any integer \(N > 1\) there is
a unique measure \(\mu_N\dd \bbB \to S\) with the following property: each
non-zero \(j \in \bbB\) admits an equi-measured \(N\)-partition \UP(w.r.t.
\(\mu\)\UP) with common measure \(\mu_N(j)\). Moreover:
\begin{itemize}
\item \(\mu_N\) is supersolid as well and \(\OPN{Aut}(\mu_N) = \OPN{Aut}(\mu)\);
\item \(n \cdot \mu_{nm}(a) = \mu_m(a)\) for all integers \(n, m > 0\) \UP(where
 \(\mu_1 \df \mu\)\UP) and any \(a \in \bbB\);
\item if \(S = G\) is a group, then the subgroup of \(G\) generated by
 \(\mu(\bbB)\) is divisible.
\end{itemize}
\end{pro}
\begin{proof}
Fix \(N > 1\). We claim that each non-zero \(j \in \bbB\) admits
an equi-measured \(N\)-partition (w.r.t. \(\mu\)). To show this property, it is
sufficient to prove that there exists a partition \(\Pp = \{b_1,\ldots,b_p\}\)
of \(\bbB\) such that each of \(b_k\) admits an equi-measured \(N\)-partition.
Indeed, if \(\Pp\) exists, then \(\unit\) admits an equi-measured
\(N\)-partition (cf. the argument used in the proof of \COR{PN-equal-part-uniq})
and then also each non-zero element of \(\bbB\) does so, thanks to \LEM{PN-1/N}.
Thus, below we show the existence of \(\Pp\).\par
Firstly, we use induction to show that for each \(n \geq 2\) there is non-zero
\(a_n \in \bbB\) such that \(a_n\) admits an equi-measured \(n\)-partition. To
this end, assume \(c_1,\ldots,c_{n-1}\) are pairwise disjoint non-zero elements
of \(\bbB\) such that \(\mu(c_k) = \mu(c_1)\) for all \(k < n\) (note that when
\(n = 2\), such a collection obviously exists). Observe that also
\(\mu(\BooD{\unit}{c_k}) = \mu(\BooD{\unit}{c_1})\) for each \(k < n\), and,
consequently, there are \(\phi_2,\ldots,\phi_{n-1} \in \OPN{Aut}(\mu)\) such
that \(\phi_k(c_1) = c_k\ (1 < k < n)\). Express \(c_1\) as \(\BooV{p}{q}\)
where \(p\) and \(q\) are two disjoint proper elements of \(\bbB\). Since \(H
\df \OPN{Aut}(\mu)\) act transitively on \(\bbB\), we infer that there exist
a proper element \(r \in \bbB\) and an automorphism \(\phi_1 \in H\) such that
\(\BooW{r}{p} = r\) and \(\BooW{\phi_1(r)}{q} = \phi_1(r)\). Then the elements
\(r,\phi_1(r),\ldots,\phi_{n-1}(r)\) are pairwise disjoint and have equal
measure, and hence they form an equi-measured \(n\)-partition of a certain
non-zero element \(a_n \in \bbB\).\par
Now set \(d \df a_N\) and use the property that \(\mu\) is supersolid to
conclude that there are a finite number of automorphisms \(h_1,\ldots,h_q \in
H\) such that \(\unit = \bigvee_{k=1}^q h_k(d)\). Finally, let \(\Pp\) be any
partition of \(\bbB\) whose each member is contained in some \(h_k(d)\). It then
follows from \LEM{PN-1/N} that each member of \(\Pp\) admits an equi-measured
\(N\)-partition, as all \(h_k(d)\) do so (as \(d\) does so).\par
So, we have shown that each non-zero element \(j \in \bbB\) admits
an equi-measured \(N\)-partition. As \(\mu\restriction{j}\) has the Rokhlin
property, we infer from \LEM{PN-1/N-uniq} that the common measure of each such
a parition of \(j\) is independent of this partition. Thus, we denote this
unique common measure by \(\mu_N(j)\). Additionally, we set \(\mu_N(\zero) \df
0_S\). In this way we obtain a function \(\mu_N\dd \bbB \to S\). It follows from
the uniqueness of common values of equi-measured partitions that \(\mu_N\) is
a measure. Moreover,
\begin{equation}\label{eqn:PN-aux70}
N \cdot \mu_N = \mu
\end{equation}
(as each \(j \in \bbB\) can be divided into \(N\) disjoint elements whose
measure \(\mu\) is equal to \(\mu_N(j)\)). In particular, \(\OPN{Aut}(\mu_N)
\subset H\) and each partial \(\mu_N\)-isomorphism is a partial
\(\mu\)-isomorphism. Thus, to conclude that \(\mu_N\) is supersolid and \(H =
\OPN{Aut}(\mu_N)\), it is enough to show that each \(\phi \in H\) belongs to
the latter group. But if \(\phi \in H\) and \(a \in \bbB\) is proper, then
both \(a\) and \(\phi(a)\) admit equi-measured \(N\)-partitions with common
measure \(\mu_N(a)\), which implies that \(\mu_N(a) = \mu_N(\phi(a))\).\par
A verification that \(n \cdot \mu_{nm}(a) = \mu_m(a)\) is immediate: \(a\)
admits an equi-measured \(nm\)-partition with common measure \(\mu_{nm}(a)\) and
therefore it admits an equi-mea\-sured \(m\)-partition with common measure
\(n \cdot \mu_{nm}(a)\). Finally, when \(S = G\) is a group, the divisibility of
the subgroup of \(G\) generated by \(\mu(\bbB)\) follows from
\eqref{eqn:PN-aux70}, as \(\mu_N(\bbB) \subset \mu(\bbB)\) for each \(N\).
\end{proof}

As an immediate consequence, we obtain

\begin{cor}{PN-non-hered}
If \(G\) is a reduced group, then the measure constantly equal to \(0_G\) is
the only \(G\)-valued supersolid measure with the hereditary Rokhlin property.
\end{cor}

\begin{cor}{PN-hered-univ}
If \(\mu\) is a group-valued supersolid measure with the hereditary Rokhlin
property and \(\lambda\dd \bbB \to G\) is a universal \(H\)-invariant
group-valued measure for \(H = \OPN{Aut}(\mu)\), then \(G\) is a vector space
over \(\QQQ\).
\end{cor}
\begin{proof}
We know form \PRO{PN-univ-homo} that \(\lambda\) is homogeneous and
\(\OPN{Aut}(\lambda) = H\) (recall also that \(G\) is generated by
\(\lambda(\bbB)\). So, \(\lambda\) is supersolid and has the hereditary Rokhlin
property. So, we infer from \PRO{PN-hered-1/N} that there are measures
\(\lambda_n\dd \bbB \to G\ (n > 1)\) with all the properties specified in that
result. Additionally, we set \(\lambda_1 \df \lambda\). Let \(G^{\omega}\) stand
for the (full) countable Cartesian power of \(G\). Define \(\xi\dd \bbB \to
G^{\omega}\) as \(\xi(a) \df (\lambda_n(a))_{n=1}^{\infty}\) and denote by \(Q\)
the subgroup of \(G^{\omega}\) generated by \(\xi(\bbB)\). Note that \(\xi\) is
an \(H\)-invariant measure. So, it follows from the universal property of
\(\lambda\) that \(\xi = \phi \circ \lambda\) for some homomorphism \(\phi\dd
G \to Q\). On the other hand, \(\lambda = \pi_1 \circ \xi\) where \(\pi_1\dd Q
\to G\) is the projection onto the first coordinate. So, \(\lambda = (\pi_1
\circ \phi) \circ \lambda\) and hence \(\pi_1 \circ \phi\) is the identity on
\(G\), and \(\phi\) is one-to-one. As \(\phi\) is also surjective (because
\(\phi(G)\) contains \(\xi(\bbB)\)), it is sufficient to show that the group
\(Q\) is a vector space over \(\QQQ\). It follows from \PRO{PN-hered-1/N} that
\(G\) is divisible and, consequently, \(Q\) is divisible as well. So, we only
need to show that \(Q\) is torsion-free, which we do below.\par
Assume \(q \in Q\) and an integer \(p > 0\) are such that \(p \cdot q = 0_G\).
Then there are non-zero elements \(a_1,\ldots,a_r \in \bbB\) and integers
\(s_1,\ldots,s_r\) such that \(q = \sum_{k=1}^r s_k \xi(a_k)\). In particular,
\(\sum_{k=1}^r s_k p \lambda_n(a_k) = 0_G\) for all \(n > 0\). Substituting
\(pn\) in place of \(n\), we obtain \(0_G = \sum_{k=1}^r s_k p \lambda_{pn}(a_k)
= \sum_{k=1}^r s_k \lambda_n(a_k)\), and therefore \(q = 0_G\), and we are done.
\end{proof}

\begin{exm}{PN-hered-non-torsion}
The assumption in the last result that \(\lambda\) is universal is essential.
More specifically, if \(\mu\dd \bbB \to G\) is a group-valued supersolid measure
with the hereditary Rokhlin property and \(G\) is generated by \(\mu(\bbB)\),
then we cannot---in general---conclude that \(G\) is torsion-free, although we
know from \PRO{PN-hered-1/N} that in this case \(G\) is divisible. To convince
oneself of that, consider a good measure \(\lambda\) such that \(\lambda(\bbB) =
[0,1] \cap \QQQ\) and define \(\mu\dd \bbB \to \RRR/\ZZZ\) as \(\pi \circ
\lambda\) where \(\pi\dd \RRR \to \RRR/\ZZZ\) is the quotient epimorphism. Since
\(\pi\) is one-to-one on \(\lambda(\bbB) \setminus \{0\}\), it follows that
\(\mu\) is supersolid and \(\OPN{Aut}(\mu) = \OPN{Aut}(\lambda)\). Moreover,
according to Akin's result \cite{Akin05}, \(\lambda\) has the hereditary Rokhlin
property. Consequently, so has \(\mu\), although the group \(G \df \RRR/\ZZZ\)
is torsion. In this example a strange phenomenon appears: on the one hand, for
\textbf{each} element \(g\) of \(G\) and any \(n > 1\) the equation \(n \cdot x
= g\) has more than one solution; on the other hand, \(\mu(\bbB \setminus
\{\zero\})\) coincides with the whole group \(G\) and the measures \(\mu_N\),
introduced in \PRO{PN-hered-1/N}, serve as `selectors' of such solutions.
\end{exm}

\section{Co-meager conjugacy classes for good-like measures}

In this section we generalise Akin's result \cite{Akin05} (on co-meager
conjugacy class of the automorphism group of a good measure whose clopen values
set is \(\QQQ\)-like) and some of results from \cite{IM16} and \cite{DKMN1}
concerning co-meager conjugacy classes in automorphism groups of some good
measures. Here we will develop a method of cyclic matrices introduced in
the last cited article. We keep the notation from the previous section.

\begin{dfn}{PN-cycle}
We call a \compatible{\lambda} matrix \emph{cyclic} if each of its rows as
well as each of its columns contains a unique non-empty entry. If \(\Gamma =
[\gamma_{pq}]_{p,q \in I}\) is such a matrix, then there exists a bijection
\(\sigma\dd I \to I\) such that for any \(p,q \in I\), \(\gamma_{pq} \neq
\varempty\) iff \(q = \sigma(p)\). If this \(\sigma\) is a cycle (that is, if
\(\sigma\) acts transitively on \(I\)), we call \(\Gamma\) \emph{singly
cyclic}. We call the above permutation \(\sigma\) \emph{induced} by
\(\Gamma\).\par
For simplicity, we will say that \(\lambda\) has \emph{cyclic property} if
every \compatible{\lambda} matrix admits a morphism into a certain cyclic
\compatible{\lambda} matrix.
\end{dfn}

\begin{lem}{PN-cyclic-inherit}
Cyclic property is hereditary. That is, if \(\lambda\) has this property,
then so has \(\lambda\restriction{j}\) for each non-zero \(j \in \bbB\).
\end{lem}
\begin{proof}
Assume \(\lambda\) has the cyclic property and fix a proper element \(j \in
\bbB\) as well as a \compatible{\mu} matrix \(A = [a_{st}]_{s,t \in I}\) where
\(\mu = \lambda\restriction{j}\). In particular, the non-empty entries of \(A\)
correspond to the values of \(\mu\) at all the elements of certain partition of
\(\bbB_0 \df \bbB\restriction{j}\). Let \(\omega \notin I\), \(J \df I \cup
\{\omega\}\) and \(B = [b_{st}]_{s,t \in J}\) be given by \(b_{st} \df a_{st}\)
for \(s,t \in I\), \(b_{\omega\omega} \df \lambda(\BooD{\unit}{j})\) and
\(b_{st} = \varempty\) otherwise. It is clear that all non-empty entries of
\(B\) correspond to the values of \(\lambda\) at all the elements of certain
partition of \(\bbB\) and that the latter condition of item (ii) of
\LEM{PN-when-Gamma(f)} is fulfilled. Therefore \(B\) is \compatible{\lambda}.
So, we infer that there are a cyclic \compatible{\lambda} matrix \(C =
[c_{pq}]_{p,q \in Q}\) and a morphism \(f\dd Q \to J\) from \(B\) into \(C\).
Set \(P \df f^{-1}(I)\) and note that if \(p \in P\) and \(q \in Q \setminus P\)
(or conversely), then \(b_{f(p)f(q)} = \varempty\) (because \(f(q) = \omega\)).
This means that \(C\) is a ``direct sum'' of two matrices: \(C' = [c_{pq}]_{p,q
\in P}\) and \(C'' \df [c_{pq}]_{p,q \in Q \setminus P}\). Since
\(\lambda(\BooD{\unit}{j}) = b_{\omega\omega} = \sum_{p,q \notin P} c_{pq}\) and
\(\sum_{p,q \in P} c_{pq} = \sum_{s,t \in I} a_{st} = \mu(j)\), it follows from
the homogeneity of \(\lambda\) that all non-empty entries of \(C'\) correspond
to the values of \(\mu\) at all the elements of some partition of \(\bbB_0\). It
is now easily seen that \(C'\) is a cyclic \compatible{\mu} matrix and
\(f\restriction{P}\) is a morphism of \(A\) into \(C'\).
\end{proof}

The material of this section applies mainly to good-like measures (introduced in
\DEF{PN-good-like}), as shown by the following result, which is a natural
generalisation of \cite[Corollary~4.2]{DKMN1}.

\begin{pro}{PN-cyclic-good-like}
Let \(\Omega\) be a positive element in a countable group \(G\) densely ordered
by `\(<\)', and let \(\mu\) be a \(G\)-valued good-like measure with
\(\mu(\unit) = \Omega\). Then \(\mu\) has the cyclic property.
\end{pro}
\begin{proof}
We begin with an observation that \(0_G \notin \mu(\bbB \setminus \{\zero\})\)
and hence \compatible{\mu} matrices do not have entries with value \(0_G\).
Thus, to simplify the proof, let us agree that (here, in this proof) empty
entries of such matrices will now be represented as those with value \(0_G\).
So, all \compatible{\mu} matrices are \(G\)-valued; and a \(G\)-valued matrix
\(B = [b_{st}]_{s,t \in T}\) is \compatible{\mu} iff all the following
conditions are met:
\begin{itemize}
\item \(0_G \leq b_{st} \leq \Omega\) for all \(s, t \in T\);
\item \(\sum_{s,t} b_{st} = \Omega\);
\item \(\sum_{t \in T} b_{st} = \sum_{t \in T} b_{ts} \neq 0_G\) for each
 \(s \in T\).
\end{itemize}
Now consider any \compatible{\mu} matrix \(A = [a_{st}]_{s,t \in T}\). We may
and do assume that \(|T| > 1\). In that case all entries of \(A\) are less than
\(\Omega\). We adapt the notation introduced in the proof of \THM{PN-Q-vector}.
Namely, for a tuple \((z_0,\ldots,z_K)\) of indices and two indices \(u\) and
\(v\) we will write:
\begin{itemize}
\item \(u \in (z_0,\ldots,z_K)\) to express that \(u = z_j\) for some \(j=0,
 \ldots,K\);
\item \((u,v) \subset (z_0,\ldots,z_K)\) to express that \(u = z_{j-1}\) and \(v
 = z_j\) for some \(j \in \{1,\ldots,K\}\).
\end{itemize}
We will construct inductively a finite collection of pairs of the form
\((\Gamma,g)\) where \(\Gamma\) is a tuple \((s_0,\ldots,s_M)\ (M > 0)\) of
indices from \(T\) such that \(s_1,\ldots,s_M\) are all distinct and \(s_0 =
s_M\), and \(g\) is a positive element of \(G\). At the \(k\)-th step of
induction we will build such a pair \((\Gamma_k,g_k)\) and a matrix \(B_k =
\bigl[b^{(k)}_{st}\bigr]_{s,t \in T}\) whose all entries are non-negative
elements of \(G\), in a way such that:
\begin{enumerate}[(1\({}_k\))]
\item \(B_k\) has more zero entries than \(B_{k-1}\) (where \(B_0 \df A\);
 the inductive construction will finish when \(B_k\) consists only of zeros);
 and
\item \(\sum_{s \in T} b^{(k)}_{s,t} = \sum_{s \in T} b^{(k)}_{ts}\) for any
 \(t \in T\); and
\item for each \(s,t \in T\), \(a_{st}\) coincides with the sum of
 \(b^{(k)}_{st}\) and all \(g_j\) from the pairs \((\Gamma_j,g_j)\) whose index
 \(j \in \{1,\ldots,k\}\) satisfies \((s,t) \subset \Gamma_j\).
\end{enumerate}
So, we start from \(B_0 = A\) and an empty collection of pairs. Now assume that
\(k > 0\) and \(B_{k-1}\) is defined as well as the pairs \((\Gamma_j,g_j)\)
for all positive \(j < k\) so that conditions \((1_{k-1})\)--\((3_{k-1})\) hold.
If \(B_{k-1}\) consists only of zeros, the construction is finished. So, assume
there are \(t_0,t_1 \in T\) such that \(b^{k-1}_{t_0t_1} > 0_G\). Now we
repeat the argument used in the proof of \THM{PN-Q-vector}, making use of
\((2_{k-1})\): if \(t_1 \neq t_0\), we inductively find indices \(t_j \in T\)
such that \(b^{(k-1)}_{t_{j-1}t_j} > 0_G\), until some \(t_p\) appears among
\(t_0,\ldots,t_{p-1}\), say \(t_p = t_q\) where \(0 \leq q < p\); then we define
\(\Gamma_k\) as \((t_q,\ldots,t_p)\), and \(g_k\) as the least entry (w.r.t.
`\(<\)') among \(b^{(k-1)}_{t_{j-1}t_j}\) (where \(q < j \leq p\)). Finally, we
define the matrix \(B^{(k)} = \bigl[b^{(k)}_{st}\bigr]_{s,t \in T}\) as follows:
\[b^{(k)}_{st} \df \begin{cases}b^{(k-1)}_{st} & (s,t) \not\subset \Gamma_k\\
b^{(k-1)}_{st}-g_k & (s,t) \subset \Gamma_k\end{cases}.\]
It follows from the construction that all the conditions \((1_k)\)--\((3_k)\)
are satisfied. In particular, after a finite number of steps, say \(p\),
the matrix \(B^{(p)}\) will consist only of zeros; and then it follows from
\((3_p)\) that:
\begin{equation}\label{eqn:PN-aux71}
a_{st} = \sum \bigl\{g_k\dd\ (s,t) \subset \Gamma_k,\
k \in \{1,\ldots,p\}\bigr\} \qquad (s,t \in T)
\end{equation}
(with the convention that \(\sum \varempty = 0_G\)). Fix for a while \(q \in
\{1,\ldots,p\}\) and say \(\Gamma_q = (s_0,\ldots,s_M)\). Let \(C_q =
[c_{k\ell}]_{k,\ell \in R_q}\) be a \(G\)-valued \(M \times M\) matrix where
\(R_q = \{r_1,\ldots,r_M\}\) and \(c_{r_jr_{j+1}} = c_{r_Mr_1} = g_q\) and
\(c_{k\ell} = 0_G\) otherwise. Additionally, let \(u_q\dd R_q \to T\) be given
by \(u_q(r_j) = s_j\). Note that \(C_q\) at each row and each column has
a unique non-zero entry and its value is \(g_q\). Moreover, it follows from
the above construction that for any \(k,\ell \in R_q\):
\begin{equation}\label{eqn:PN-aux73}
c_{k\ell} > 0_G \iff c_{k\ell} = g_q \iff (u_q(k),u_q(\ell)) \subset \Gamma_q.
\end{equation}
Further, we may and do assume that the sets \(R_1,\ldots,R_p\) are pairwise
disjoint, and thus all the functions \(u_k\) extend to a common function \(u\dd
R \to T\) where \(R\) is the union of all \(R_k\). Now we define our final
(cyclic) matrix \(C\) as the diagonal block matrix with \(C_1,\ldots,C_p\) on
the diagonal. In particular, \(R\) serves as the set of indices of \(C\). We
infer from \eqref{eqn:PN-aux71} that when computing the sum \(\sum_{s,t \in S}
a_{st}\) with \(a_{st}\) replaced by the right-hand side expression of this
formula, then each of \(g_k\) appears precisely \(|R_k|\) times, which yields
that the sum of all entries of \(C\) equals \(\Omega\), and hence \(C\) is
\compatible{\mu}, thanks to the preliminary remark of this whole proof.
Moreover, it is easily seen that \(C\) is cyclic. So, it remains to show that
\(u\) is a morphism from \(A\) to \(C\). But this is an immediate consequence of
\eqref{eqn:PN-aux71} and \eqref{eqn:PN-aux73}.
\end{proof}

\begin{rem}{PN-cyclic-pospec}
The above proof, without any major change (however, one needs to skip the first
paragraph), works also for measures induced by linearly ordered solid pospecs
\((J,\prec)\) in groups \(G\) (after replacing `\(<\)' by `\(\prec\)'). To
convince oneself of that, after starting from a \compatible{\mu} matrix \(A\),
one fixes a partition of \(\bbB\) such that the measures of its members coincide
with the values of non-empty entries of \(A\). Then, when constructing pairs
\((\Gamma_k,g_k)\) (here \(g_k \in J \setminus \{\varempty,\Omega\}\)) and
a matrix \(B_k\) (with entries in \(J\)), one decomposes members of
this partition into proper parts so that their measures have appropriate values
among \(g_k\) and \(b^{(k)}_{st}\). This will guarantee that the final matrix
\(C\) will also be \compatible{\mu}. We leave the details to interested readers.
So, measures induced by linearly ordered solid pospecs have the cyclic property.
\end{rem}

\begin{exm}{PN-non-cyclic}
Below we show that if \(\lambda(a) = 0_S\) for some proper element \(a \in
\bbB\), then \(\lambda\) does not have the cyclic property. This applies,
e.g., to \(G\)-filling measures introduced in \DEF{PN-G-filling} that seem to be
most appropriate counterexamples of good measures in the abstract context of
measures with values in groups without ordering.\par
So, assume \(\lambda(a) = 0_S\) where \(a\) is proper. Express
\(\BooD{\unit}{a}\) in the form \(\BooV{b}{c}\) where \(b\) and \(c\) are two
disjoint proper elements of \(\bbB\). Define \(D = [d_{jk}]_{j,k=1}^2\) as
follows: \(d_{11} = \lambda(b)\), \(d_{22} = \lambda(c)\), \(d_{12} = 0_S\ (=
\lambda(a))\) and \(d_{21} = \varempty\). Since \(\{a,b,c\}\) is a partition of
\(\bbB\), it may easily be checked that \(D\) is \compatible{\lambda}. We claim
that \(D\) admits no morphisms into cyclic \compatible{\lambda} matrices. To
convince oneself of that, assume \(R = [r_{jk}]_{j,k=1}^n\) is cyclic and \(f\dd
\{1,\ldots,n\} \to \{1,2\}\) is a morphism from \(D\) into \(R\). Denote by
\(\sigma\) the permutation of \(\{1,\ldots,n\}\) induced by \(R\); that is,
\begin{equation}\label{eqn:PN-aux72}
r_{jk} \neq \varempty \iff k = \sigma(j).
\end{equation}
Set \(I \df f^{-1}(\{1\})\). Since \(a_{21} = \varempty\), \eqref{eqn:PN-morph}
implies that \(r_{jk} = \varempty\) whenever \(j \in J \df \{1,\ldots,n\}
\setminus I\) and \(k \in I\). So, \(\sigma(J) \subset J\), by
\eqref{eqn:PN-aux72}. But \(\sigma\) is a bijection, hence \(\sigma(J) = J\) and
\(\sigma(I) = I\) as well. On the other hand, \(d_{12} \neq \varempty\), thus,
again by \eqref{eqn:PN-morph}, \(r_{jk} \neq \varempty\) for some \(j \in I\)
and \(k \in J\). We infer from \eqref{eqn:PN-aux72} that \(k = \sigma(j)\),
which is impossible as \(\sigma(j) \in \sigma(I) = I\).
\end{exm}

Since the structure of cyclic matrices is quite simple, it is more natural to
work with their more handy `models'. Following \cite{DKMN1}, each such a matrix
can be `encoded' as a tuple of the form \(((g_1,n_1),\ldots,(g_m,n_m))\) where
\(g_1,\ldots,g_m\) are elements of the respective (semi)group (in which a given
matrix takes values) and \(n_1,\ldots,n_m\) are positive integers. Each of
the pairs \((g_j,n_j)\) corresponds to a single orbit of the permutation (of
the set of indices) induced by non-empty entries of a given cyclic matrix.
Conversely, any tuple of the above form induces a cyclic matrix in an obvious
way (each pair \((g,n)\) from the tuple produces a \(n \times n\) singly cyclic
submatrix, put on the diagonal of the entire matrix, with a unique non-empty
entry equal to \(g\)). Under such an identification, two cyclic matrices \(C\)
and \(D\) corresponding to two tuples \(((g_1,n_1),\ldots,(g_m,n_m))\) and
\(((h_1,k_1),\ldots,(h_p,k_p))\) (respectively) are equivalent precisely when
there exists a bijection \(\sigma\dd \{1,\ldots,m\} \to \{1,\ldots,p\}\) such
that \((h_{\sigma(j)},k_{\sigma(j)}) = (g_j,n_j)\) for all \(j \in \{1,\ldots,
m\}\); and \(C\) is \compatible{\mu} iff
\begin{equation}\label{eqn:PN-poly-compat}
(\underbrace{g_1,\ldots,g_1}_{n_1},\ldots,\underbrace{g_m,\ldots,g_m}_{n_m})
\in \Delta_{n_1+\ldots+n_m}(\mu).
\end{equation}
Moreover, one may check that if \(f\dd P \to Q\) is a morphism (where \(P\) and
\(Q\) are the sets of indices of, respectively, \(D\) and \(C\)) from \(C\) into
\(D\), and \(\sigma\) and \(\tau\) are permutations (of \(Q\) and \(P\)) induced
by \(C\) and \(D\), then \(f \circ \tau = \sigma \circ f\) and, consequently,
the image of each \(\tau\)-orbit \(o\) under \(f\) is a \(\sigma\)-orbit \(o'\),
and the size of \(o'\) divides the size of \(o\). This implies that \(f\)
induces a surjection \(u\dd \{1,\ldots,p\} \to \{1,\ldots,m\}\) such that:
\begin{itemize}
\item[(mor)] for each \(s \in \{1,\ldots,m\}\), \(n_s\) divides \(k_j\) for any
 \(j \in F_s \df u^{-1}(\{s\})\), and \(g_s = \sum_{j \in F_s} \frac{k_j}{n_s}
 h_j\).
\end{itemize}
Conversely, any surjection \(u\) satisfying (mor) naturally induces a morphism
from \(C\) into \(D\). (Moreover, up to permutation of \(P\) and \(Q\), each
morphism from \(C\) into \(D\) is obtained from some function \(u\) described
above.)\par
Since in our research we consider all \compatible{\mu} matrices up to
equivalence `\(\equiv\)', the above observations enable us to consider tuples of
pairs instead of cyclic matrices, as we will do from now on to the end of
the paper. For simplicity, we call each tuple \(\gamma = ((g_1,n_1),\ldots,
(g_m,n_m))\) of the form discussed above a \emph{polycycle}. It is
\emph{\compatible{\mu}} if \eqref{eqn:PN-poly-compat} is fulfilled.
A \emph{morphism} from \(\gamma\) into a polycycle \(\eta = ((h_1,k_1),\ldots,
(h_p,k_p))\) is any surjection \(u\dd \{1,\ldots,p\} \to \{1,\ldots,m\}\)
satisfying (mor). It is easily seen that the composition of two morphisms
between polycycles is a morphism as well (however, as morphisms between two
polycycles are actually functions coming in reverse directions, formally
the composition of two morphism composes corresponding two functions in
a reverse order). So, all (\compatible{\mu}) polycycles form a category.
The most important thing to remember is that if this category or some its
subcategory has amalgamation (or some of its variants), then so has
the corresponding (sub)category of cyclic matrices.\par
The following is a direct consequence of \THM{PN-dense-conj}. We leave its proof
to the reader.

\begin{lem}{PN-polycyc-dense}
If \(\lambda\) has the cyclic property, then \(\OPN{Aut}(\lambda)\) contains
a dense conjugacy class iff for any two \compatible{\lambda} polycycles
\(\gamma_1\) and \(\gamma_2\) there exists a \compatible{\lambda} polycycle
\(\gamma_0\) such that both \(\gamma_1\) and \(\gamma_2\) admit morphisms into
\(\gamma_0\).
\end{lem}

\begin{dfn}{PN-strong-Rokhlin}
We say \(\lambda\) has \emph{strong Rokhlin property} if \(\OPN{Aut}(\lambda)\)
admits a dense \(\ggG_{\delta}\) conjugacy class. If \(\lambda\restriction{j}\)
has the strong Rokhlin property for each non-zero \(j \in \bbB\), we say that
\(\lambda\) has \emph{hereditary strong Rokhlin property}.
\end{dfn}

The role of the cyclic property is explained in the following

\begin{thm}{PN-conj-class}
If \(\lambda\) has the cyclic property, then \(\lambda\) has the hereditary
Rokhlin property iff it has the hereditary strong Rokhlin property.
\end{thm}
\begin{proof}
We only need to show the `only if' part of the result. So, assume \(\lambda\)
has the hereditary Rokhlin property and fix a non-zero \(E \in \bbB\). Thanks to
\LEM{PN-cyclic-inherit}, \(\lambda\restriction{E}\) has the cyclic property and
all that has been assumed about \(\lambda\) is valid for
\(\lambda\restriction{E}\). Thus, we may and do assume that \(E = \unit\).\par
It is a kind of folklore (see, e.g., \cite{DKMN1}) that under the assumption of
the theorem, it is sufficient to prove that the category of \compatible{\lambda}
polycycles has amalgamation (since \(\lambda\) is actually ultrahomogeneous---by
\THM{PN-4EP}---and \(((\lambda(\unit),1))\) is \compatible{\lambda} and admits
a morphism into every \compatible{\lambda} polycycle). (Then
the Fra\"{\i}ss\'{e} limit-like of the class of partial \(\lambda\)-isomorphisms
defined on finite Boolean subalgebras of \(\bbB\) that correspond to cyclic
matrices will be a \(\lambda\)-automorphism with dense \(\ggG_{\delta}\)
conjugacy class.) To this end, we fix three \compatible{\lambda} polycycles
\(\epsi = ((f_1,\nu_1),\ldots,(f_r,\nu_r))\), \(\gamma = ((g_1,\alpha_1),\ldots,
(g_p,\alpha_p))\) and \(\delta = ((h_1,\beta_1),\ldots,(h_q,\beta_q))\) and two
morphisms \(u\dd \{1,\ldots,p\} \to \{1,\ldots,r\}\) (from \(\gamma\) into
\(\epsi\)) and \(v\dd \{1,\ldots,q\} \to \{1,\ldots,r\}\) (from \(\delta\) into
\(\epsi\)). Since \(\epsi\) satisfies \eqref{eqn:PN-poly-compat}, there is
a partition \(\{a^1,\ldots,a^r\}\) of \(\bbB\) such that each \(a^j\) admits
an equi-measured \(\nu_j\)-partition with common measure \(f_j\), say \(\{a^j_1,
\ldots,a^j_{\nu_j}\}\). In particular, for any \(k \in \{1,\ldots,\nu_j\}\)
there exists \(\phi \in \OPN{Aut}(\lambda)\) that sends \(a^j_1\) onto
\(a^j_k\).\par
Fix arbitrary \(z \in \{1,\ldots,r\}\) and let \(\{j_1,\ldots,j_x\}\) and
\(\{k_1,\ldots,k_y\}\) be the inverse images, respectively, \(u^{-1}(\{z\})\)
and \(v^{-1}(\{z\})\) (we assume here that all \(j_s\) are distinct and
similarly for all \(k_s\)). Consider polycycles \(\gamma_z \df ((g_{j_1},
\alpha_{j_1}/\nu_z),\ldots,(g_{j_x},\alpha_{j_x}/\nu_z))\) and \(\delta_z \df
(h_{k_1},\beta_{k_1}/\nu_z),\ldots,(h_{k_y},\beta_{k_y}/\nu_z))\). (To simplify
reasoning, here we allow polycycles to have indices that are not consecutive
integers.) It follows from homogeneity of \(\lambda\) and the property that
\(u\) and \(v\) are morphisms that \(\gamma_z\) and \(\delta_z\) are
\compatible{\rho} where \(\rho \df \lambda\restriction{a^z_1}\) (here one
applies \LEM{PN-1/N-uniq}). Now since \(\rho\) has the Rokhlin property, we
infer from \LEM{PN-polycyc-dense} that there are a \compatible{\rho} polycycle
\(\eta_z = ((w_1,\xi_1),\ldots,(w_{m_z},\xi_{m_z}))\) and two morphisms \(s_z\dd
\{1,\ldots,m_z\} \to \{j_1,\ldots,j_x\}\) (from \(\gamma_z\) into \(\eta_z\))
and \(t_z\dd \{1,\ldots,m_z\} \to \{k_1,\ldots,k_y\}\) (from \(\delta_z\) into
\(\eta_z\)).\par
We now define a final polycycle \(\eta\) as \((\underbrace{\eta_1,\ldots,
\eta_1}_{\nu_1},\ldots,\underbrace{\eta_r,\ldots,\eta_r}_{\nu_r})\). Since
\(\eta_z\) induces a partition of \(a^z_1\), each of its \(\nu_k\) `copies' in
\(\eta\) induces a partition of \(a^z_k\) (\(k=1,\ldots,\nu_z\)). Hence \(\eta\)
is \compatible{\lambda}. Now changing the indices in each \(\eta^z\) appearing
in \(\eta\) (in a way consistent with the order of pairs in \(\eta\)), we may
and assume that all the corresponding functions \(s_z\) have pairwise disjoint
domains (and similarly for all \(t_z\)) so that all of them extend to a function
\(s\) from the set \(T\) of indices of \(\eta\) to \(\{1,\ldots,p\}\) (and
similarly, all \(t_z\) extend to a function \(t\dd T \to \{1,\ldots,q\}\)). It
easily follows from the construction that both \(s\) and \(t\) are morphisms
(from \(\gamma\) and \(\delta\), respectively, into \(\eta\)), and that \(u
\circ s = v \circ t\), which finishes the proof.
\end{proof}

\begin{thm}{PN-strong-good-like}
Let \(G\) be a countable group densely ordered by `\(<\)'. \TFCAE
\begin{enumerate}[\upshape(i)]
\item all \(G\)-valued good-like measures have the Rokhlin property;
\item all \(G\)-valued good-like measures have the hereditary strong Rokhlin
 property;
\item \(G\) is a vector space over \(\QQQ\).
\end{enumerate}
\end{thm}
\begin{proof}
Since the property of being a good-like measure is hereditary, it follows from
\PRO{PN-cyclic-good-like} and \THM{PN-conj-class} that conditions (i) and (ii)
are equivalent.\par
First assume (i) holds. Recall that \(G\) is torsion-free (as a linearly ordered
group). So, (iii) holds iff \(G\) is divisible. To show that it is the case, it
is enough to prove that for any positive element \(g \in G\) and each \(n > 1\)
there is \(x \in G\) such that \(g = nx\). To this end, we fix \(g > 0_G\) and
\(n > 1\). Set \(\Omega \df ng\), note that \(\Omega > 0_G\), and take
a good-like measure \(\mu\dd \bbB \to G\) with \(\mu(\unit) = \Omega\). Since
\(\mu\) has the Rokhlin property and \(\unit\) admits an equi-measured
\(n\)-partition w.r.t. \(\mu\) (because \(\mu(\unit) = ng\) and \(0_G < g <
\Omega\)), we infer from \LEM{PN-1/N} that each proper element of \(\bbB\) also
admits such a partition, which implies that the equation \(g = nx\) has
a solution in \(G\).\par
Now assume that \(G\) is divisible, and fix a \(G\)-valued good-like measure
\(\mu\) with \(\mu(\unit) = \Omega\). To prove that \(\mu\) has the Rokhlin
property, it is enough to show, thanks to \LEM{PN-polycyc-dense}, that any two
\compatible{\mu} polycycles admit a morphism into a common \compatible{\mu}
polycycle. To this end, fix two such polycycles \(\gamma = ((g_1,n_1),\ldots,
(g_m,n_m))\) and \(\delta = ((h_1,k_1),\ldots,(h_p,k_p))\). So, all \(g_j\) and
\(k_s\) are positive elements of \(G\) and \(\sum_{j=1}^m n_j g_j = \Omega =
\sum_{s=1}^p k_s h_s\). Choose positive integer \(N > 1\) that is divisible by
all \(n_j\) as well as all \(k_s\). It follows from divisibility of \(G\) that
there are \(g_1',\ldots,g_m'\) and \(h_1',\ldots,h_p'\) such that \(g_j =
\frac{N}{n_j} g_j'\) and \(h_s = \frac{N}{k_s} h_s'\) (for all \(j\) and \(s\)).
Then all \(g_j'\) and all \(h_s'\) are positive elements of \(G\) less than
\(\Omega\); and \(N \sum_{j=1}^m g_j' = \Omega = N \sum_{s=1}^p h_s'\).
Consequently, \(\sum_{j=1}^m g_j' = \sum_{s=1}^p h_s'\). So, there is
a non-zero element \(z \in \bbB\), and two partitions \(\{x_1,\ldots,x_m\}\) and
\(\{y_1,\ldots,y_p\}\) of \(\bbB\restriction{z}\) such that \(\mu(x_j) = g_j'\)
and \(\mu(y_s) = h_s'\) (for all \(j\) and \(s\)). For any \(j \in \{1,\ldots,
m\}\) and \(s \in \{1,\ldots,p\}\) set \(w_{js} \df \BooW{x_j}{y_s}\); and let
\(T \df \{(j,s)\dd\ w_{js} \neq \zero\}\). Further, fix a bijection \(\sigma\dd
\{1,\ldots,|T|\} \to T\) and denote its coordinates by \(\tau\) and \(\rho\);
that is, \(\sigma(t) = (\tau(t),\rho(t))\ (1 \leq t \leq |T|)\). Finally, for
any \(t \in I \df \{1,\ldots,|T|\}\) we set \(f_t \df \mu(w_{\tau(t)\rho(t)})\),
and denote by \(\epsi\) the polycycle \(((f_1,N),\ldots,(f_{|T|},N))\). Note
that \(f_t > 0_G\) and \(N \sum_{t=1}^{|T|} f_t = N \mu(\bigvee_{(j,s) \in T}
w_{js}) = N \mu(z) = N \sum_{j=1}^m g_j' = \Omega\), which means that \(\epsi\)
is \compatible{\mu}. Finally, for any \(j \in \{1,\ldots,m\}\) we have
\(\sum_{t \in \tau^{-1}(\{j\})} \frac{N}{n_j} f_t = \frac{N}{n_j} \mu(x_j) =
g_j\); similarly, \(\sum_{t \in \rho^{-1}(\{s\})} \frac{N}{k_s} f_t = h_s\) for
any \(s \in \{1,\ldots,p\}\), which shows that \(\tau\) and \(\rho\) are
morphisms from, respectively, \(\gamma\) and \(\delta\) to \(\epsi\), and we are
done.
\end{proof}

As good-like measures are solid (and most of them are not supersolid---cf.
\REM{PN-lin-ord-super} and item (B) of \PRO{PN-lin-ord}), we get

\begin{cor}{PN-strong-not-super}
There exist solid measures with the hereditary strong Rokhlin property that are
not supersolid.
\end{cor}

It follows from Theorems~\ref{thm:PN-0-filling} and \ref{thm:PN-Q-vector} that
countable vector spaces over \(\QQQ\) are precisely those groups \(G\) such that
all \(G\)-filling measures have the Rokhlin property; or, equivalently, that
admit a \(G\)-filling measure with the hereditary Rokhlin property. Note also
that if one such a measure has the latter property, then all of them have it.
\THM{PN-strong-good-like} has a similar spirit, but asserts more. So,
the following question naturally arises:\vspace{2mm}

\begin{prb}{PN-prb2}
Do all \(\QQQ\)-filling measures have the hereditary strong Rokhlin property?
\end{prb}

We also pose a related problem.

\begin{prb}{PN-prb3}
Characterise good-like measures with the hereditary Rokhlin property.
\end{prb}

We wish to underline that \THM{PN-strong-good-like} does not solve the above
problem, as good-like measures are, in general, not supersolid (cf.
\REM{PN-lin-ord-super}), so, \COR{PN-hered-univ} cannot be applied in this case.

\bibliographystyle{siam}
\bibliography{ref}
\end{document}